\documentclass[11pt,onecolumn]{IEEEtran}

\usepackage{macro}

\title{Multiple Packing:\\Lower Bounds via Error Exponents}
\author{
\IEEEauthorblockN{
Yihan Zhang\IEEEauthorrefmark{1}
Shashank Vatedka\IEEEauthorrefmark{2}
}\\
\IEEEauthorblockA{
\IEEEauthorrefmark{1}Institute of Science and Technology Austria \\
\IEEEauthorrefmark{2}Department of Electrical Engineering, Indian Institute of Technology Hyderabad
}
}

\begin{document}
\maketitle

\begin{abstract}
We derive lower bounds on the maximal rates for multiple packings in high-dimensional Euclidean spaces. Multiple packing is a natural generalization of the sphere packing problem. For any $ N>0 $ and $ L\in\mathbb{Z}_{\ge2} $, a multiple packing is a set $\mathcal{C}$ of points in $ \mathbb{R}^n $ such that any point in $ \mathbb{R}^n $ lies in the intersection of at most $ L-1 $ balls of radius $ \sqrt{nN} $ around points in $ \mathcal{C} $. We study this problem for both bounded point sets whose points have norm at most $\sqrt{nP}$ for some constant $P>0$ and unbounded point sets whose points are allowed to be anywhere in $ \mathbb{R}^n $. Given a well-known connection with coding theory, multiple packings can be viewed as the Euclidean analog of list-decodable codes, which are well-studied for finite fields. We derive the best known lower bounds on the optimal multiple packing density. This is accomplished by establishing a curious inequality which relates the list-decoding error exponent for additive white Gaussian noise channels, a quantity of average-case nature, to the list-decoding radius, a quantity of worst-case nature. We also derive various bounds on the list-decoding error exponent in both bounded and unbounded settings which are of independent interest beyond multiple packing.
\end{abstract}


\section{Introduction}
\label{sec:intro}
We {study} the problem of \emph{multiple packing} {in Euclidean space}, a natural generalization of the sphere packing problem~\cite{conway-sloane-book}. 
Let $ P>0,N>0 $ and $ L\in\bZ_{\ge2} $. 
We say that a point set $\cC$ in\footnote{Here we use $ \cB^n(r) $ to denote an $n$-dimensional Euclidean ball of radius $r$ centered at the origin.} $ \cB^n(\sqrt{nP}) $ forms a \emph{$(P,N,L-1)$-multiple packing}\footnote{We choose to stick with $L-1$ rather than $L$ for notational convenience. 
This is because in the proof, we need to examine the violation of $(L-1)$-packing, i.e., the existence of an $L$-sized subset that lies in a ball of radius $\sqrt{nN}$. } if any point in $ \bR^n $ lies in the intersection of at most $L-1$ balls of radius $ \sqrt{nN} $ around points in $\cC$. 
Equivalently, the radius of the smallest ball containing any size-$L$ subset of $\cC$ is larger than $ \sqrt{nN} $. 
This radius is known as the \emph{Chebyshev radius} of the $L$-sized subset. 
If $ L = 2 $, then $\cC$ forms a \emph{sphere packing}, i.e., a point set such that balls of radius $ \sqrt{nN} $ around points in $\cC$ are disjoint, or equivalently, the pairwise distance of points in $\cC$ is larger than $ 2\sqrt{nN} $. 
The density of $\cC$ is measured by its \emph{rate} defined as  
\begin{align}
R(\cC) \coloneqq \frac{1}{n}\ln|\cC|. 
\label{eqn:intro-rate}
\end{align}
Denote by $ C_{L-1}(P,N) $ the largest rate of a $(P,N,L-1)$-multiple packing as $ n\to\infty $. We will also refer to this as the \emph{adversarial list-decoding capacity}, or simply the \emph{list-decoding capacity}.
Note that $ C_{L-1}(P,N) $ depends on $P$ and $N$ only through their ratio $ N/P $ which we call the \emph{noise-to-signal ratio}.
The goal of this paper is to derive lower bounds on $ C_{L-1}(P,N) $.

The problem of multiple packing is closely related to the \emph{list-decoding} problem \cite{elias-1957-listdec,wozencraft-1958-listdec} in coding theory. 
Indeed, a multiple packing can be seen exactly as the \emph{Euclidean} analog of a list-decodable code.
We will interchangeably use the terms ``packing'' and ``code'' to refer to the point set of interest. 
To see the connection, note that if any point/codeword in a multiple packing is transmitted through an \emph{adversarial} omniscient jamming\footnote{An \emph{omniscient} adversary is one who can choose the jamming/additive noise vector that must satisfy a power constraint but otherwise be any function of the codebook and the transmitted codeword (available noncausally to the jammer). This is more powerful than an \emph{oblivious} jammer, who can transmit a jamming vector that can only depend on the codebook but not the transmitted codeword.} channel that can inflict an arbitrary additive noise of length at most $ \sqrt{nN} $, then given the distorted transmission, one can decode to a list of the nearest $L-1$ points which is guaranteed to contain the transmitted one. 
The quantity $ C_{L-1}(P,N) $ can therefore be interpreted as the capacity of this channel. 
Moreover, it is well known that with a small amount of shared secret key between the transmitter and receiver, list-decodable codes can be turned into unique-decodable codes so that the receiver can \emph{uniquely decode} to the correct codeword with a vanishingly small probability of error~\cite{langberg-focs2004,sarwate-thesis,bhattacharya2019shared}. 
List-decoding also serves as a proof technique for deriving bounds on the (unique-decoding) capacity for various adversarial jamming channels; see, e.g., \cite{zhang2022quadratically,zhang-2020-twoway}. 

\subsection{Bounded packings}
Let us start with the $L=2$ case. 
The best known lower bound is due to Blachman in 1962 \cite{blachman-1962} using a simple volume packing argument. 
The best known upper bound is due to Kabatiansky and Levenshtein in 1978 \cite{kabatiansky-1978} using the seminal Delsarte's linear programming framework \cite{delsarte-1973} from coding theory. 
These bounds meet nowhere except at two points: $ N/P = 0 $ (where $ C_{L-1}(P,N) = \infty $), and $ N/P = 1/2 $ (where $ C_{L-1}(P,N) = 0 $).

For $ L>2 $, Blinovsky~\cite{blinovsky-1999-list-dec-real} claimed a lower bound (\Cref{eqn:compare-lb-ee}) on $ C_{L-1}(P,N) $, and in fact our results are closely related to this work. Unfortunately, there were some gaps in the proof of~\cite{blinovsky-1999-list-dec-real} that we were not able to resolve, and we therefore use an alternate approach to proving this result which could be of wider interest. Please see \Cref{sec:bli-mistake-ee} for an in-depth discussion of the connection to~\cite{blinovsky-1999-list-dec-real}. 
To the best of our knowledge, the bound that we derive in this paper is the best known lower bound on $ C_{L-1}(P,N) $. 
Our high-level ideas of connecting error exponents to the list-decoding radius is in fact inspired by \cite{blinovsky-1999-list-dec-real}. However, we use a different approach to achieving the same.
In the same paper, Blinovsky \cite{blinovsky-1999-list-dec-real} also derived an upper bound using the ideas of the Plotkin bound \cite{plotkin-1960} and the Elias--Bassalygo bound \cite{bassalygo-pit1965} in coding theory. 
The same upper bound was originally shown by Blachman and Few \cite{blachman-few-1963-multiple-packing} using a more involved approach. 
Blinovsky and Litsyn \cite{blinovsky-litsyn-2011} later improved this bound in the low-rate regime by a recursive application of a bound on the distance distribution by Ben-Haim and Litsyn \cite{benhaim-litsyn-2006-reliability-gaussian}. 
The latter bound in turn relies on the Kabatiansky--Levenshtein linear programming bound \cite{kabatiansky-1978}. 
Blinovsky and Litsyn \cite{blinovsky-litsyn-2011} numerically verified that their bounds improve previous ones when the rate is sufficiently low, but no explicit expression was provided.  More recently, Zhang and Vatedka~\cite{zhang-split-misc} various upper and lower bounds on the list-decoding capacity and a related notion known as the \emph{average-radius} list-decoding\footnotemark{} capacity. 
\footnotetext{A set $ \cC $ of $ \bR^n $-valued points is called an average-radius multiple packing if for any $(L-1)$-subset of $\cC$, the maximum distance from any point in the subset to the centroid of the subset is less than $ \sqrt{nN} $. 
Here the centroid of a subset is defined as the average of the points in the subset. }

\subsection{Unbounded packings}
The above notion of $ (P,N,L-1) $-multiple packing is  well defined even if we remove the restriction that all points lie in $ \cB^n(\sqrt{nP}) $ and allow the packing to contain points anywhere in $ \bR^n $. 
The codebook can now be countably infinite, and this leads to the notion of $ (N,L-1) $-multiple packing. 
The density of such an unbounded packing is measured by the (normalized) number of points per volume 
\begin{align}
R(\cC) \coloneqq \limsup_{K\to\infty}\frac{1}{n}\ln\frac{\card{\cC\cap\cB^n(K)}}{\card{\cB^n(K)}}. 
\label{eqn:intro-rate-unbdd}
\end{align}
With slight abuse of terminology, we call $ R(\cC) $ the \emph{rate} of the unbounded packing $\cC$, a.k.a.\ the \emph{normalized logarithmic density (NLD)}. 
The largest density of unbounded multiple packings as $ n\to\infty $ is denoted by $ C_{L-1}(N) $.

For $L=2$, the unbounded sphere packing problem 
has a long history since at least the Kepler conjecture \cite{kepler-1611} in 1611. 
The best known lower bound is given by a straightforward volume packing argument \cite{minkowski-sphere-pack}. 
The best known upper bound is obtained by reducing it to the bounded case for which we have the Kabatiansky--Levenshtein linear programming-type bound \cite{kabatiansky-1978}. 
For $ L>2 $, Blinovsky \cite{blinovsky-2005-random-packing} described a lower bound by analyzing an (expurgated) Poisson Point Process (PPP). Further results along similar lines can be found in Zhang and Vatedka~\cite{zhang-split-ppp}.


For $ L\to\infty $, Zhang and Vatedka \cite{zhang2022listtrans-it} determined the limiting value of $ C_{L-1}(N) $. 
The limit of $ C_{L-1}(P,N) $ as $ L\to\infty $ is a folklore in the literature and a proof can be found in \cite{zhang2022quadratically}. 

Very little is known about structured packings. 
Grigorescu and Peikert \cite{grigorescu-peikert-2012-list-dec-barnes-wall} initiated the study of list-decodability of lattices. 
Some recent work can be found in Mook and Peikert~\cite{mook-peikert-2020-lattice}, and Zhang and Vatedka \cite{zhang2022listtrans-it} on list-decodability of random lattices and infinite constellations. 


\subsection{Error exponents}
Our lower bounds on $C_{L-1}(P,N)$ and $C_{L-1}(N)$ are derived by making an interesting connection between list-decodable codes for adversarial (omnsicient jamming) channels and list-decodable codes for the additive white Gaussian noise (AWGN) channel. 

Loosely speaking, we show that any code that is $(L-1)$-list-decodable over the AWGN $\cN(0,\sigma^2)$ channel with exponentially decaying probability of error $e^{-nE+o(n)}$ for some $E>0$ can be expurgated without loss of rate to give a code with Chebyshev radius $\sqrt{2n\sigma^2 E + o(n)}$. We then derive bounds on the list-decoding random coding and expurgated error exponents for the AWGN channel, and use these to obtain lower bounds on the (adversarial) list-decoding capacity. A similar approach was used to derive lower bounds on the zero-rate threshold of \emph{binary} channels under (adversarial) list-decoding in~\cite{d2021new}.
However, no lower bounds on the list-decoding capacity were derived below the zero-rate threshold. 

List-decoding error exponents for discrete memoryless channels (DMCs) were originally studied by Gallager~\cite{gallager} and Viterbi and Omura~\cite{viterbi2013principles}. A more systematic study of list-decoding error exponents for DMCs was made by Merhav~\cite{merhav2014list}. Merhav~\cite{merhav2014list} gave bounds on the list-decoding random coding and expurgated error exponents for both constant and exponential (in $n$) list sizes. In this work, we derive expressions for the list-decoding error exponents for discrete memoryless channels and AWGN channels with constant list sizes. We also derive these bounds in the case where input constraints are imposed on the channel through an extension of the same ideas. The techniques used are standard, following~\cite{gallager} and in fact, our expressions for the DMC without input constraints numerically match those in Gallager~\cite{gallager} and Merhav~\cite{merhav2014list}. However, previous results obtain the error exponent in terms of an optimization problem or in a form which unfortunately does not allow us to derive explicit lower bounds on the achievable Chebyshev radius \cite[Eqn.\ (47) and (48)]{merhav2014list}. For the AWGN channel, we derive explicit expressions for the list-decoding random coding and expurgated exponents which could be of independent interest. We also solve the optimization problem in an alternate form that allows us to get a simple closed form expression for the achievable (adversarial) list-decoding rate.

\subsection{List-decoding}
\label{sec:related}
For $L=2$, the problem of (unbounded) sphere packing has a long history and has been extensively studied, especially for small dimensions. 
The largest packing density is open for almost every dimension, except for $ n = 1 $ (trivial), $2$ (\cite{thue1911-2dspherepacking,toth-1940-2dspherepacking}), $ 3 $ (the Kepler conjecture, \cite{hales1998kepler,hales2017formal}), $ 8 $ (\cite{viazovska-2017-8dspherepacking}) and $24$ (\cite{cohn-2017-24spherepacking}). 
For $ n\to\infty $, the best lower and upper bounds remain the trivial sphere packing bound and Kabatiansky--Levenshtein's linear programming bound \cite{kabatiansky-1978}. 
This paper is only concerned with (multiple) packings in high dimensions and we measure the density in the normalized way as mentioned in \Cref{sec:intro}. 

There is a parallel line of research in combinatorial coding theory. 
Specifically, a uniquely-decodable code (resp.\ list-decodable code) is nothing but a sphere packing (resp.\ multiple packing) which has been extensively studied for $ \bF_q^n $ equipped with the Hamming metric. 

We first list the best known results for sphere packing (i.e., $L=2$) in Hamming spaces. 
For $q=2$, the best lower and upper bounds are the Gilbert--Varshamov bound \cite{gilbert1952,varshamov1957} proved using a trivial volume packing argument and the second MRRW bound \cite{mrrw2} proved using the seminal Delsarte's linear programming framework \cite{delsarte-1973}, respectively. 
Surprisingly, the Gilbert--Varshamov bound can be improved using algebraic geometry codes \cite{goppa1977,tvz} for $ q\ge49 $. 
Note that such a phenomenon is absent in $ \bR^n $; as far as we know, no algebraic constructions of Euclidean sphere packings are known to beat the greedy/random constructions. 
For $ q\ge n $, the largest packing density is known to exactly equal the Singleton bound \cite{komamiya1953singleton,joshi1958singleton,singleton1964} which is met by, for instance, the Reed--Solomon code \cite{reed-solomon}. 

Less is known for multiple packing in Hamming spaces. 
We first discuss the binary case (i.e., $q=2$). 
For every $ L\in\bZ_{\ge2} $, the best lower bound appears to be Blinovsky's bound \cite[Theorem 2, Chapter 2]{blinovsky2012book} proved under the stronger notion of average-radius list-decoding. 
The best upper bound for $L=3$ is due to Ashikhmin, Barg and Litsyn \cite{abl-2000-list-size-2} who combined the MRRW bound \cite{mrrw2} and Litsyn's bound \cite{litsyn-1999} on distance distribution. 
For any $ L\ge4 $, the best upper bound is essentially due to Blinovsky again \cite{blinovsky-1986-ls-lb-binary}, \cite[Theorem 3, Chapter 2]{blinovsky2012book}, though there are some partial improvements. 
In particular, the idea in \cite{abl-2000-list-size-2} was recently generalized to larger $L$ by Polyanskiy \cite{polyanskiy-2016-list-dec} who improved Blinovsky's upper bound for \emph{even} $L$ (i.e., odd $L-1$) and sufficiently large $R$. 
Similar to \cite{abl-2000-list-size-2}, the proof also makes use of a bound on distance distribution due to Kalai and Linial \cite{kalai-linial-1995-distance-distribution} which in turn relies on Delsarte's linear programming bound. 
For larger $q$, Blinovsky's lower and upper bounds \cite{blinovsky-2005-ls-lb-qary,blinovsky-2008-ls-lb-qary-supplementary}, \cite[Chapter III, Lecture 9, \S 1 and 2]{ahlswede-blinovsky-2008-book} remain the best known.

As $L\to\infty$, the limiting value of the largest multiple packing density is a folklore in the literature known as the ``list-decoding capacity'' theorem\footnote{It is an abuse of terminology to use ``list-decoding capacity'' here to refer to the large $L$ limit of the $(L-1)$-list-decoding capacity.}.
Moreover, the limiting value remains the same under a more general notion of average-radius list-decoding. 

The problem of list-decoding was also studied for settings beyond the Hamming errors, e.g., list-decoding against erasures \cite{guruswami-it2003,ben-aroya-doron-ta-shma-2018-explicit-erasure-ld}, insertions/deletions \cite{guruswami-2020-listdec-insdel}, asymmetric errors \cite{polyanskii-zhang-2021-z}, etc. 
Zhang et al.\ considered list-decoding over general adversarial channels \cite{zhang-2019-list-dec-general}. 
List-decoding against other types of adversaries with \emph{limited} knowledge such as oblivious or myopic adversaries were also considered in the literature \cite{hughes-1997-list-avc,sarwate-gastpar-2012-listdec,zhang-2020-obli-list-dec,hosseinigoki2019list,zhang2022quadratically}. 

\subsection*{Relation to conference version}
This work was presented in part at the 2022 IEEE International Symposium on Information Theory~\cite{zhang2022lowerisit}. All proofs were omitted in the published 6-page conference paper. The current article contains complete proofs of all results, and also includes several novel results on error exponents and list-decoding for Euclidean codes without power constraints.

\section{Our results}
\label{sec:results}

In this paper, we derive  lower bounds on the largest multiple packing density for the bounded and the unbounded case.
Let $ C_{L-1}(P,N) $ and $ C_{L-1}(N) $ denote the largest possible density of bounded and unbounded multiple packings, respectively.

\subsection{{Bounded packings}}
\label{sec:results-bdd}
In \Cref{thm:lb-ee}, we derive the following lower bound on the $(P,N,L-1)$-list-decoding capacity: 
\begin{align}
C_{L-1}(P,N) &\ge \frac{1}{2}\sqrbrkt{\ln\frac{(L-1)P}{LN} + \frac{1}{L-1}\ln\frac{P}{L(P-N)}}. \label{eqn:compare-lb-ee}
\end{align}
The above bound was also claimed in~\cite{blinovsky-1999-list-dec-real} by connecting list-decoding for adversarial channels with the probability of error of list-decoding over AWGN channels. However, there were some gaps in the proof that we could not fully resolve. Our work uses similar high-level ideas, but we use a different approach in connecting the Chebyshev radius of a code with the list-decoding error exponent for communication over AWGN channels. A more detailed discussion of the connections between these two works can be found in \Cref{sec:bli-mistake-ee}.

It is a folklore (whose proof can be found in \cite{zhang2022quadratically}) that as $ L\to\infty $, $ C_{L-1}(P,N) $ converges to the following expression:
\begin{align}
C_{\mathrm{LD}}(P,N) &= \frac{1}{2}\ln\frac{P}{N}. \label{eqn:compare-ld-cap} 
\end{align}

This bound, and the bounds derived in~\cite{zhang-split-misc} for $ (P,N,L-1) $-multiple packing are plotted in \Cref{fig:ld-bdd} with $ L = 5 $. 
The horizontal axis is the noise-to-signal ratio $ N/P $ and the vertical axis is the value of various bounds. 
\Cref{eqn:compare-lb-ee} turns out to be the largest lower bound for all $N,P\ge0$ and $ L\in\bZ_{\ge2} $. 
Furthermore, it was shown in \cite{zhang-split-misc} via a completely different approach (Gallager's bounding trick and large deviation principle) that the same bound also holds for expurgated spherical codes under average-radius list-decoding. 
We also plot our lower bound together with an Elias-Bassalygo-type upper bound 
\begin{align}
C_{L-1}(P,N) &\le \frac{1}{2}\ln\frac{(L-1)P}{LN} \label{eqn:compare-eb} 
\end{align}
on the capacity from~\cite{zhang-split-misc} for $ L=3,4,5 $. 
They both converge from below to \Cref{eqn:compare-ld-cap} as $ L $ increases. 

\begin{figure}[htbp]
	\centering
		\includegraphics[width=0.95\textwidth]{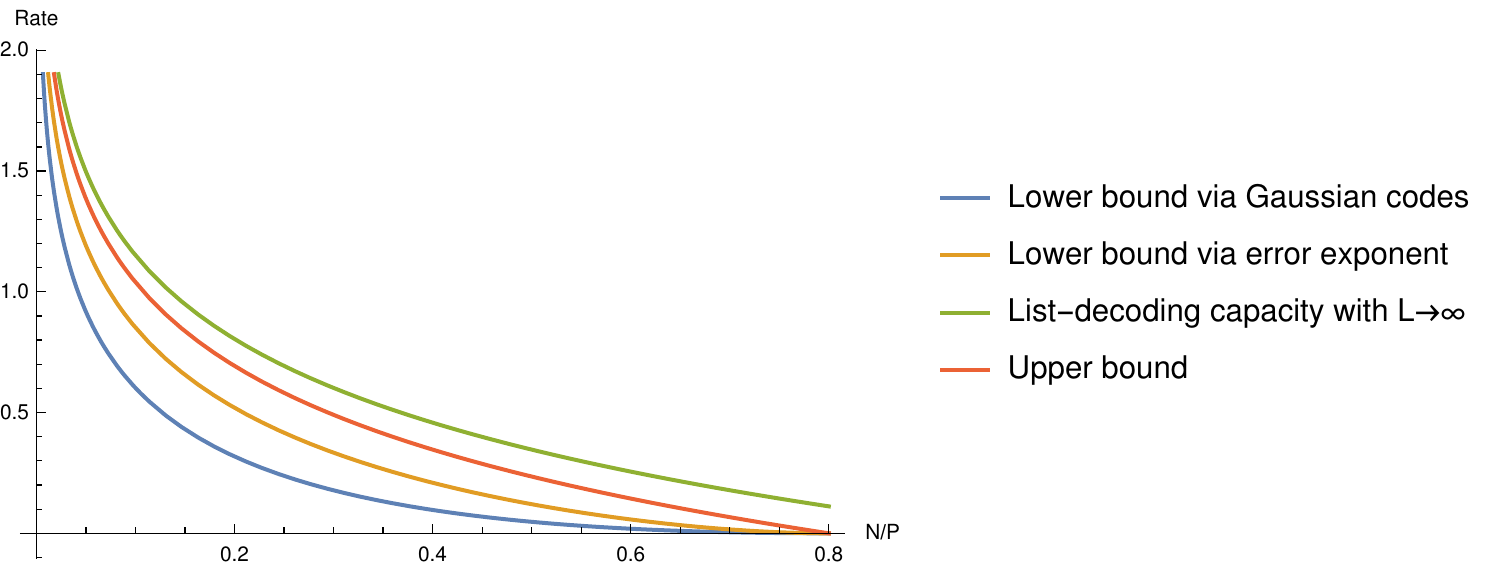}
	\caption{Comparison of different bounds for the $(P,N,L-1)$-list-decoding problem. The horizontal axis is $ N/P $ and the vertical axis is the value of various bounds. Recall that the rate (\Cref{eqn:intro-rate}) of a bounded packing is defined as its normalized cardinality. We plot bounds for $ L = 5 $. As can be seen from the plots, the results in this paper (\Cref{eqn:compare-lb-ee}) give the best known lower bounds on the capacity (Lower bound via error exponent). The lower bound using Gaussian codebooks and the upper bound (\Cref{eqn:compare-eb}) are derived in~\cite{zhang-split-misc}. 
	}
	\label{fig:ld-bdd}
\end{figure}

\begin{figure}[htbp]
	\centering
	\begin{subfigure}[b]{\linewidth}
		\centering
		\includegraphics[width=0.95\textwidth]{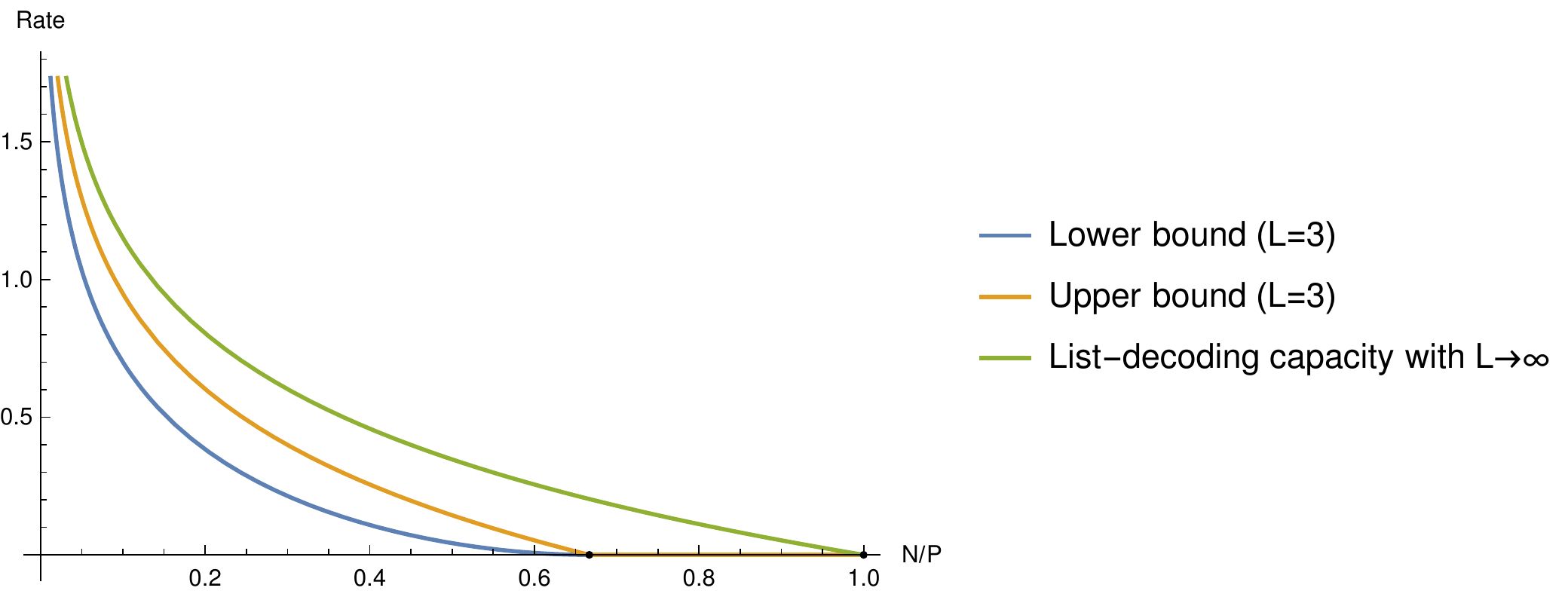}
		\caption{}
		\label{fig:EEEB3}
	\end{subfigure}
	\\
	\begin{subfigure}[b]{\linewidth}
		\centering
		\includegraphics[width=0.95\textwidth]{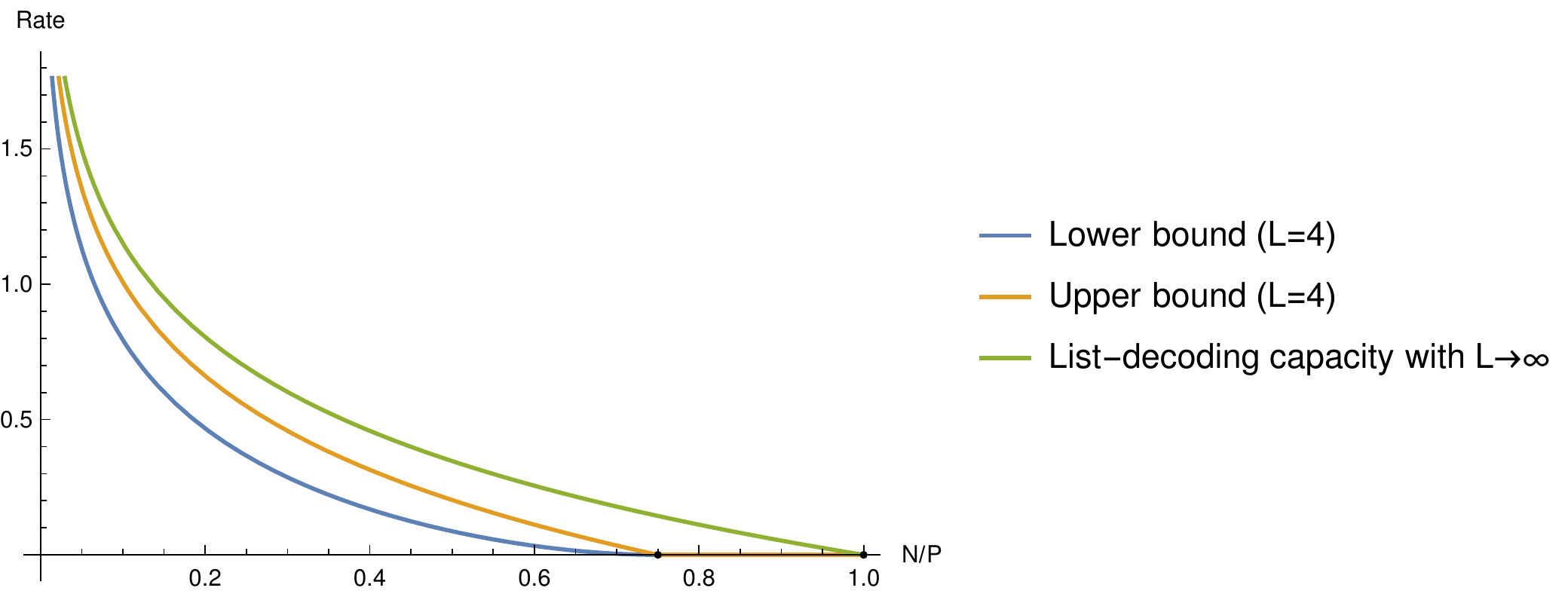}
		\caption{}
		\label{fig:EEEB4}
	\end{subfigure}
	\\
	\begin{subfigure}[b]{\linewidth}
		\centering
		\includegraphics[width=0.95\textwidth]{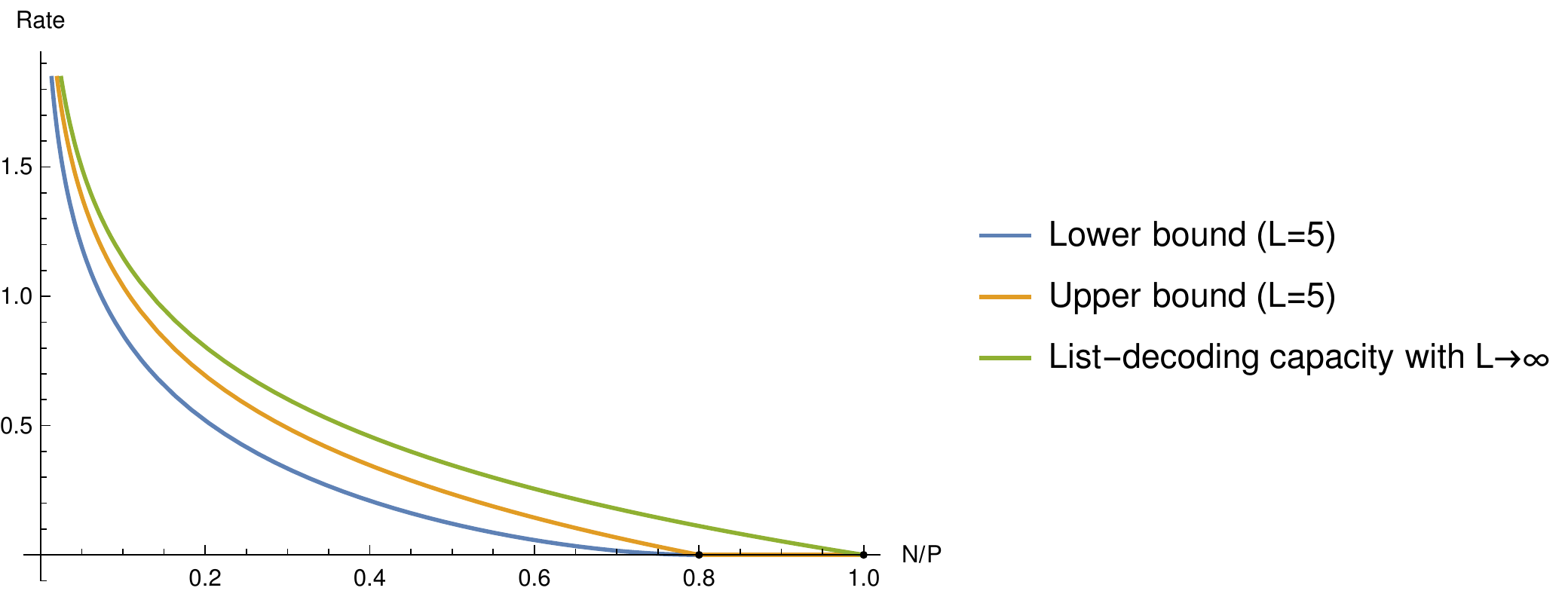}
		\caption{}
		\label{fig:EEEB5}
	\end{subfigure}
	\caption{Plots of the lower bounds in \Cref{eqn:compare-lb-ee} for $ C_{L-1}(P,N) $ derived in this paper and the Elias-Bassalygo-type upper bound (\Cref{eqn:compare-eb}) from~\cite{zhang-split-misc} for $ L=3,4,5 $. As $L$ increases, they both converge from below to $ C_{\mathrm{LD}}(P,N) $ (\Cref{eqn:compare-ld-cap}).  }
	\label{fig:EEEB}
\end{figure}

\subsection{{Unbounded packings}}
\label{sec:results-unbdd}
We then juxtapose various bounds for the $(N,L-1)$-multiple packing problem. 
In \Cref{thm:lb-ee-unbdd}, the following lower bound on $ C_{L-1}(N) $ 
\begin{align}
 C_{L-1}(N) &\ge \frac{1}{2}\ln\frac{L-1}{2\pi eNL} - \frac{\ln L}{2(L-1)} \label{eqn:compare-lb-ppp}
\end{align}
is obtained via the connection with error exponents for the AWGN channel using a codebook generated using Poisson Point Processes (PPPs).
In \cite{zhang-split-ppp} it is shown that the same bound is in fact the \emph{exact} asymptotics of a certain ensemble of infinite constellations under $(N,L-1)$-\emph{average-radius} list-decoding (which is stronger than $ (N,L-1) $-list-decoding). 

It is known (see, e.g., \cite{zhang2022listtrans-it}) that as $ L\to\infty $, $ C_{L-1}(N) $ converges to the following expression: 
\begin{align}
C_{\mathrm{LD}}(N) &= \frac{1}{2}\ln\frac{1}{2\pi eN}. \label{eqn:compare-ld-cap-unbdd}
\end{align}
Therefore, our bound converges to $C_{\mathrm{LD}}(N)$ as $L\to\infty$.

The bound in \Cref{eqn:compare-lb-ppp} together with the Elias-Bassalygo-type upper bound~\cite{zhang-split-misc} 
\begin{align}
C_{L-1}(N) &\le \frac{1}{2}\ln\frac{L-1}{2\pi eNL} \label{eqn:compare-eb-unbdd}
\end{align}
are plotted in \Cref{fig:PPPEB} for $ L=3,4,5 $. 
The horizontal axis is $N$ and the vertical axis is the value of various bounds. 
\Cref{eqn:compare-lb-ppp} turns out to be the largest known lower bound for all $N\ge0$ and $ L\in\bZ_{\ge2} $.
\Cref{eqn:compare-lb-ppp,eqn:compare-eb-unbdd} both converge from below to \Cref{eqn:compare-ld-cap-unbdd} as $ L $ increases.


\begin{figure}[htbp]
	\centering
	\includegraphics[width=0.95\textwidth]{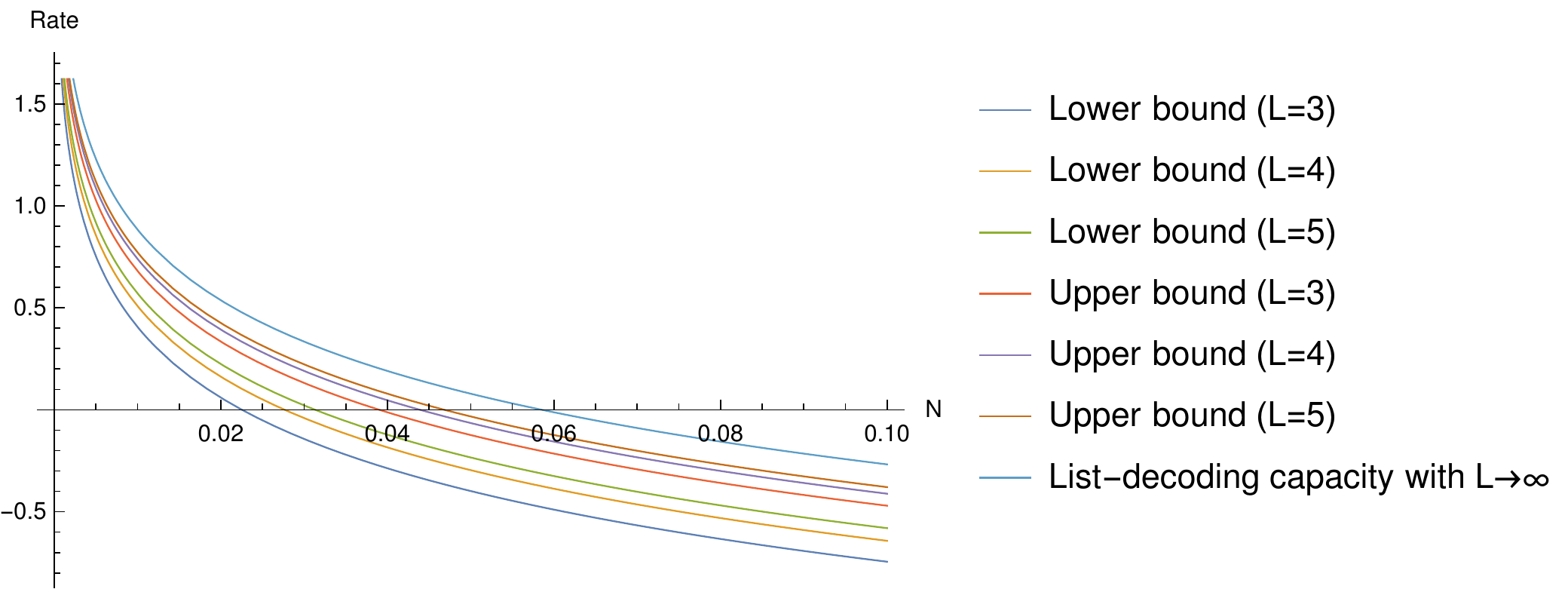}
	\caption{Plots of the best known lower bound (\Cref{eqn:compare-lb-ppp}) on $ C_{L-1}(N) $ and the Elias--Bassalygo-type upper bound (\Cref{eqn:compare-eb-unbdd} derived in~\cite{zhang-split-misc}) for $ L=3,4,5 $. The horizontal axis is $N$ and the vertical axis is the value of bounds. Recall that the rate (\Cref{eqn:intro-rate-unbdd}) of an unbounded packing is defined as the (normalized) number of points per volume which can be negative. As $ L $ increases, they both converge from below to $ C_{\mathrm{LD}}(N) $ (\Cref{eqn:compare-ld-cap-unbdd}). The lower bound \Cref{eqn:compare-lb-ppp} is obtained in this paper using the connection with error exponents. Moreover, it is actually the exact asymptotics of a certain ensemble of infinite constellations under the \emph{average-radius} notion of unbounded multiple packing (see~\cite{zhang-split-misc}). }
	\label{fig:PPPEB}
\end{figure}

\subsection{List-decoding error exponents}
\label{sec:results-ld-ee}

As alluded to above, our bounds on the  multiple packing density (\Cref{eqn:compare-lb-ee,eqn:compare-lb-ppp}) are obtained via a curious connection to list-decoding error exponents of Additive White Gaussian Noise (AWGN) channels.  
Informally, the error exponent of a code $\cC$ used over an AWGN channel is the asympototic value of $ -\frac{1}{n}\ln(P_{\e,\avg}(\cC)) $, where $P_{\e,\avg}(\cC)$ is the average probability of error when the code is used to communicate over an AWGN channel. 
See \Cref{sec:def-ld-ee} for formal definitions and \Cref{sec:ld-ee} for analogous definitions for more general channels. 
Deriving tight bounds on the best achievable list-decoding error exponents is of independent interest in information theory. 
Another part of the contribution of this paper consists in the derivation of explicit lower bounds on the maximal error exponents for AWGN channels under list-decoding. 
(We also have results on list-decoding error exponents for more general channels; see \Cref{sec:rce,sec:expurg-expo,sec:ee-input-constr,sec:ee-cts-alpha}.)

Let $ \sigma>0 $ and $ L\in\bZ_{\ge2} $. 
Consider a channel which takes as input an $ \bR^n $-valued vector and adds to it an $n$-dimensional independent Gaussian noise vector each entry i.i.d.\ with mean $0$ and variance $ \sigma^2 $. 
We prove the existence of codes for such a channel attaining certain error exponents under $ (L-1) $-list-decoding (i.e., the receiver decodes the channel output to the list of $L-1$ nearest codewords).

\subsubsection{Input constrained case}
\label{sec:results-ld-ee-bdd}

In the input constrained case, the channel input $\vx$ is subject to a power constraint $ \normtwo{\vx}\le\sqrt{nP} $ for some $P>0$. 
Let $ \snr \coloneqq P/\sigma^2 $ denote the signal-to-noise ratio (SNR). 
The capacity of an AWGN channel with SNR $\snr$ was shown by Shannon \cite{shannon1948mathematical} to be $ \frac{1}{2}\ln(1+\snr) $. 
In \Cref{thm:awgn-rce,thm:awgn-exe}, we prove that there exist codes of rate (as per \Cref{eqn:intro-rate}) $ 0\le R\le\frac{1}{2}\ln(1+\snr) $ that under maximum likelihood $(L-1)$-list-decoding attain an error exponent $ E_{L-1}(R,\snr) $ defined as follows:
\begin{align}
E_{L-1}(R,\snr) &\ge \begin{cases}
E_{\mathrm{r},L-1}(R,\snr), & R_{\crit,L-1}(\snr) \le R\le\frac{1}{2}\ln(1+\snr) \\
E_{\mathrm{sl},L-1}(R,\snr), & R_{\x,L-1}(\snr) \le R\le R_{\crit,L-1}(\snr) \\
E_{\ex,L-1}(R,\snr), & 0\le R\le R_{\x,L-1}(\snr)
\end{cases}, \notag 
\end{align}
where $ E_{\mathrm{r},L-1},E_{\mathrm{sl},L-1},E_{\ex,L-1} $ denote the random coding exponent, the straight line bound and the expurgated exponent, respectively. 
These bounds read as follows:
\begin{align}
E_{\mathrm{r},L-1}(R,\snr) &\coloneqq \frac{1}{2}\ln\sqrbrkt{e^{2R} - \frac{\snr(e^{2R} - 1)}{2}\paren{\sqrt{1+\frac{4e^{2R}}{\snr(e^{2R} - 1)}} - 1}} \notag \\
&\quad + \frac{\snr}{4e^{2R}}\paren{e^{2R} + 1 - (e^{2R} - 1)\sqrt{1+\frac{4e^{2R}}{\snr(e^{2R} - 1)}}}, \label{eqn:er-ip-constr} \\
E_{\mathrm{sl},L-1}(R,\snr) &\coloneqq -R(L-1) + \frac{L-1}{2}\ln\paren{L + \snr + \sqrt{(L-\snr)^2 + 4\snr}} + \frac{1}{2}\ln\paren{L-\snr + \sqrt{(L-\snr)^2 + 4\snr}} \notag \\
&\quad + \frac{1}{4}\paren{L + \snr - \sqrt{(L-\snr)^2 + 4\snr}} - \frac{L}{2}\ln(2L), \label{eqn:esl-ip-constr} \\
E_{\ex,L-1}(R,\snr) &\coloneqq \frac{\snr(Lt-1)}{2Lt}, \label{eqn:eex-ip-constr}
\end{align}
where $ t\in[1,1/L] $ is the unique solution to the equation $ (Lt - 1)e^{2R} = (L-1)t^{\frac{L}{L-1}} $. 
Moreover, 
\begin{align}
R_{\crit,L-1}(\snr) &\coloneqq \frac{1}{2}\ln\paren{\frac{1}{2} + \frac{\snr}{2L} + \frac{1}{2}\sqrt{1 - \frac{2(L-2)}{L^2}\snr + \frac{\snr^2}{L^2}}}, \label{eqn:rcrit-ip-constr} \\
R_{\x,L-1}(\snr) &\coloneqq \frac{1}{2}\paren{\ln\frac{\sqrt{L^2+\snr^2 - 2\snr(L-2)} + L + \snr}{2L} + \frac{1}{L-1}\ln\frac{\sqrt{L^2+\snr^2 - 2\snr(L-2)} + L - \snr}{2L}}. \label{eqn:rx-ip-constr}
\end{align}

When specialized to $L-1=1$, the above bounds recover the Gallager's exponents \cite{gallager-1965-simple-deriv}, \cite[Theorem 7.4.4]{gallager} for unique-decoding.
The above bounds are plotted in \Cref{fig:LDUDEEFig} for $ L-1=1 $ and $ L-1=2 $, both with $ \snr=1 $ fixed. 

\begin{figure}[htbp]
	\centering
	\begin{subfigure}[b]{\linewidth}
		\centering
		\includegraphics[width=0.95\textwidth]{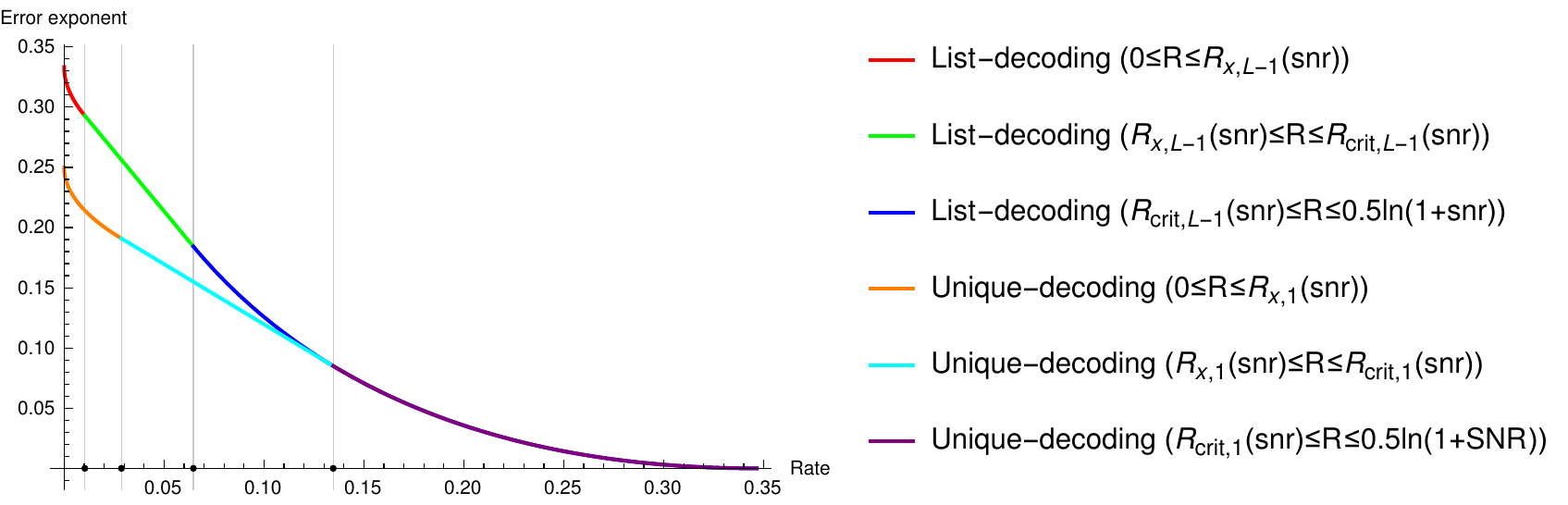}
		\label{fig:LDUDEE}
		\caption{}
	\end{subfigure}
	\\
	\begin{subfigure}[b]{\linewidth}
		\centering
		\includegraphics[width=0.95\textwidth]{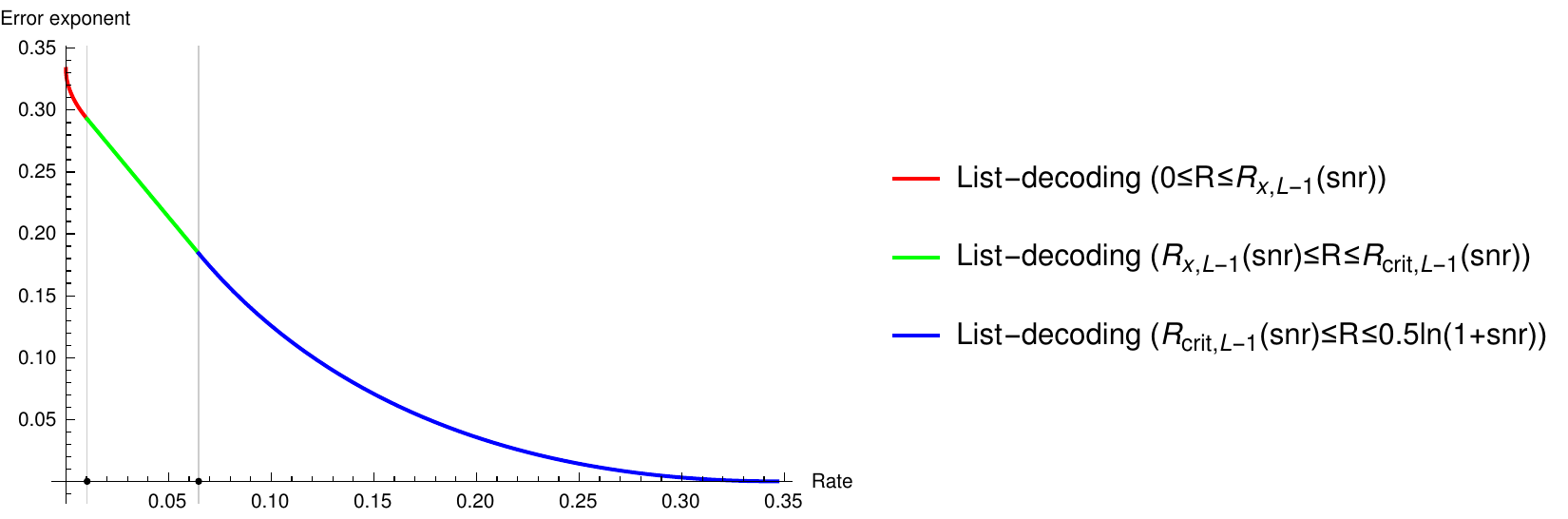}
		\caption{}
		\label{fig:LDEE}
	\end{subfigure}
	\\
	\begin{subfigure}[b]{\linewidth}
		\centering
		\includegraphics[width=0.95\textwidth]{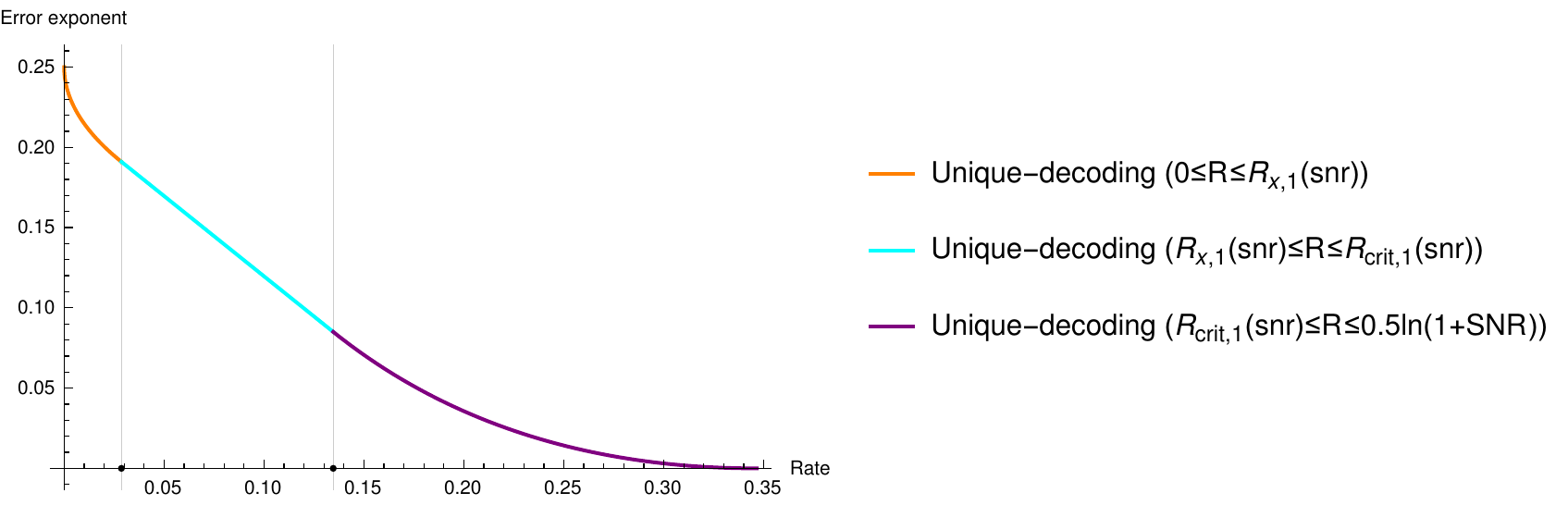}
		\caption{}
		\label{fig:UDEE}
	\end{subfigure}
	\caption{Comparison of Gallager's unique-decoding error exponents (\Cref{eqn:ee-low-ud,eqn:ee-mid-ud,eqn:ee-high-ud}) and our list-decoding error exponents (\Cref{{eqn:er-ip-constr,eqn:esl-ip-constr,eqn:eex-ip-constr}}) for AWGN channels with $ \snr = 1 $. We plot our bounds for $ L=3 $. The horizontal axis is the rate $ 0\le R\le\frac{1}{2}\ln(1+\snr) $ and the vertical axis is the values of the exponents. The list-decoding error exponents and the unique-decoding error exponents are plotted jointly in \Cref{fig:LDUDEE} and are plotted separately in \Cref{fig:LDEE,fig:UDEE}, respectively. Interestingly, the error exponent under list-decoding remains the same for sufficiently large rate, i.e., $ R\ge R_{\x,1}(\snr) $. However, for any rate less than $ R_{\x,1}(\snr) $, list-decoding does increase the error exponent. Moreover, the critical rates $ R_{\x,L-1}(\snr) $ and $ R_{\x,L-1}(\snr) $ (see \Cref{eqn:rcrit-ip-constr,eqn:rx-ip-constr}) become smaller than $ R_{\x,1}(\snr) $ and $ R_{\x,1}(\snr) $ (see \Cref{eqn:rx-ud,eqn:rcrit-ud}), respectively, under $(L-1)$-list-decoding. }
	\label{fig:LDUDEEFig}
\end{figure}

\subsubsection{Input unconstrained case}
\label{sec:results-ld-ee-unbdd}

In the input unconstrained case, the capacity of an AWGN channel with noise variance $ \sigma^2 $ was shown by Poltyrev \cite{poltyrev1994coding} to be $ \frac{1}{2}\ln\frac{1}{2\pi e\sigma^2} $. 
In \Cref{thm:rce-ppp,thm:exe-matern}, we prove that there exist codes of rate (as per \Cref{eqn:intro-rate-unbdd}) $ R=\frac{1}{2}\ln\frac{1}{2\pi e\sigma^2\alpha^2} $ for some $ \alpha\ge1 $ that under maximum likelihood $(L-1)$-list-decoding attain an error exponent $ E_{L-1}(\alpha) $ defined as follows:
\begin{align}
E_{L-1}(\alpha) &\ge \begin{cases}
E_{\mathrm{r},L-1}(\alpha), & 1\le\alpha\le\sqrt{L} \\
E_{\mathrm{sl},L-1}(\alpha), & \sqrt{L}\le\alpha\le \sqrt{2L}\\
E_{\ex,L-1}(\alpha), & \alpha\ge\sqrt{2L}
\end{cases}, \label{eqn:ld-ee-ip-unconstr} 
\end{align}
where $ E_{\mathrm{r},L-1},E_{\mathrm{sl},L-1},E_{\ex,L-1} $ denote the random coding exponent, the straight line bound and the expurgated exponent, respectively. 
These bounds read as follows:
\begin{align}
E_{\mathrm{r},L-1}(\alpha) &\coloneqq \frac{\alpha^2}{2} - \ln\alpha - \frac{1}{2}, \notag \\
E_{\mathrm{sl},L-1}(\alpha) &\coloneqq \frac{L-1}{2} - \frac{L}{2}\ln L + (L-1)\ln\alpha, \notag \\
E_{\ex,L-1}(\alpha) &\coloneqq \frac{\alpha^2}{16} + \frac{1}{16}\sqrt{\alpha^4 + 8\alpha^2(2L-3)+16} - \frac{L-1}{2}\ln\paren{\sqrt{\alpha^4 + 8\alpha^2(2L-3) + 16} - \alpha^2 + 4} \notag \\
&\quad + \frac{L-2}{2}\ln\paren{\sqrt{\alpha^4 + 8\alpha^2(2L-3) + 16} + \alpha^2 + 4} + \frac{3}{2}\ln2 - \frac{1}{4} . \notag 
\end{align}

When specialized to $L-1=1$, the above bounds recover the Poltyrev's exponents \cite[Theorem 3]{poltyrev1994coding} for unique-decoding.
The above bounds are plotted in \Cref{fig:LDUDEEunbddFig} for $ L-1=1 $ and $ L-1=2 $. 

\begin{figure}[htbp]
	\centering
	\begin{subfigure}[b]{\linewidth}
		\centering
		\includegraphics[width=0.9\textwidth]{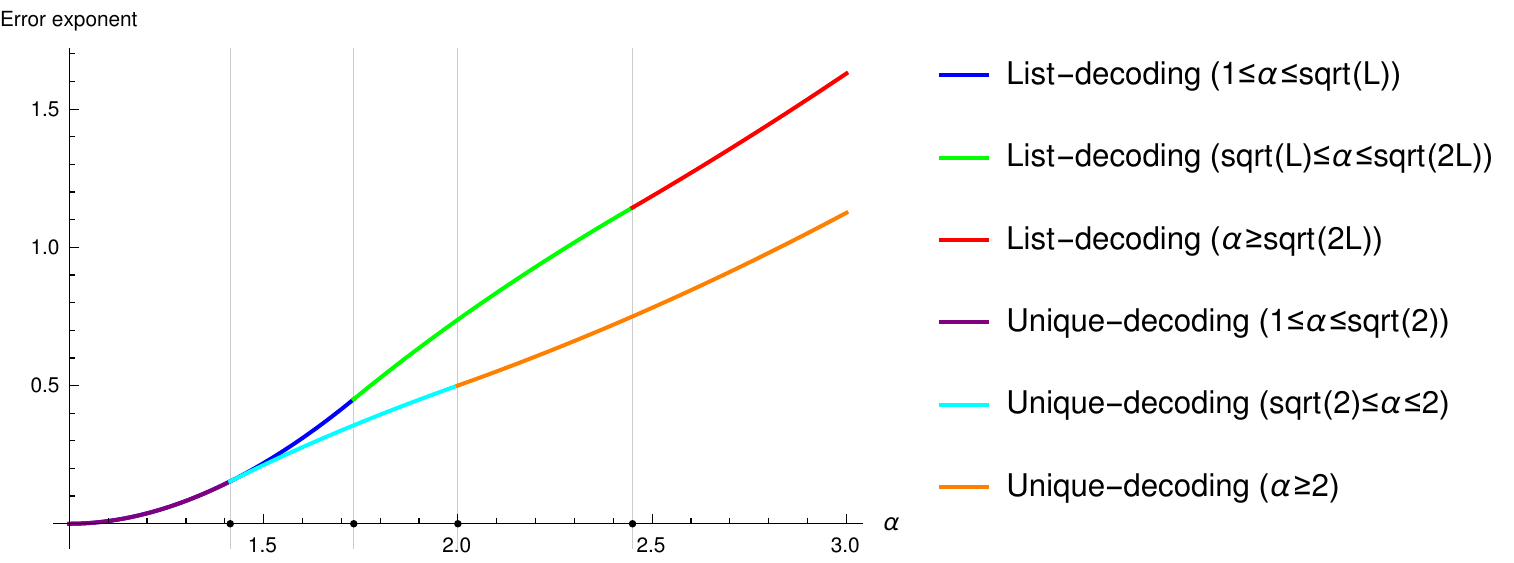}
		\label{fig:LDUDEEunbdd}
		\caption{}
	\end{subfigure}
	\\
	\begin{subfigure}[b]{\linewidth}
		\centering
		\includegraphics[width=0.9\textwidth]{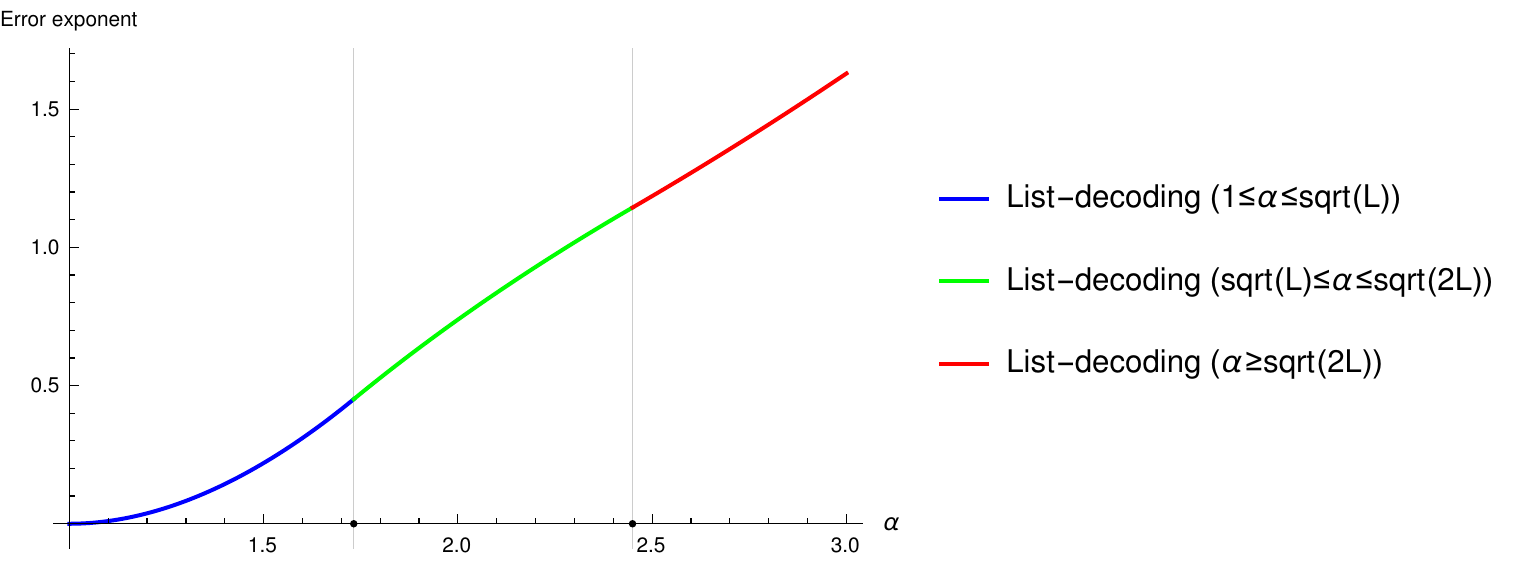}
		\caption{}
		\label{fig:LDEEunbdd}
	\end{subfigure}
	\\
	\begin{subfigure}[b]{\linewidth}
		\centering
		\includegraphics[width=0.9\textwidth]{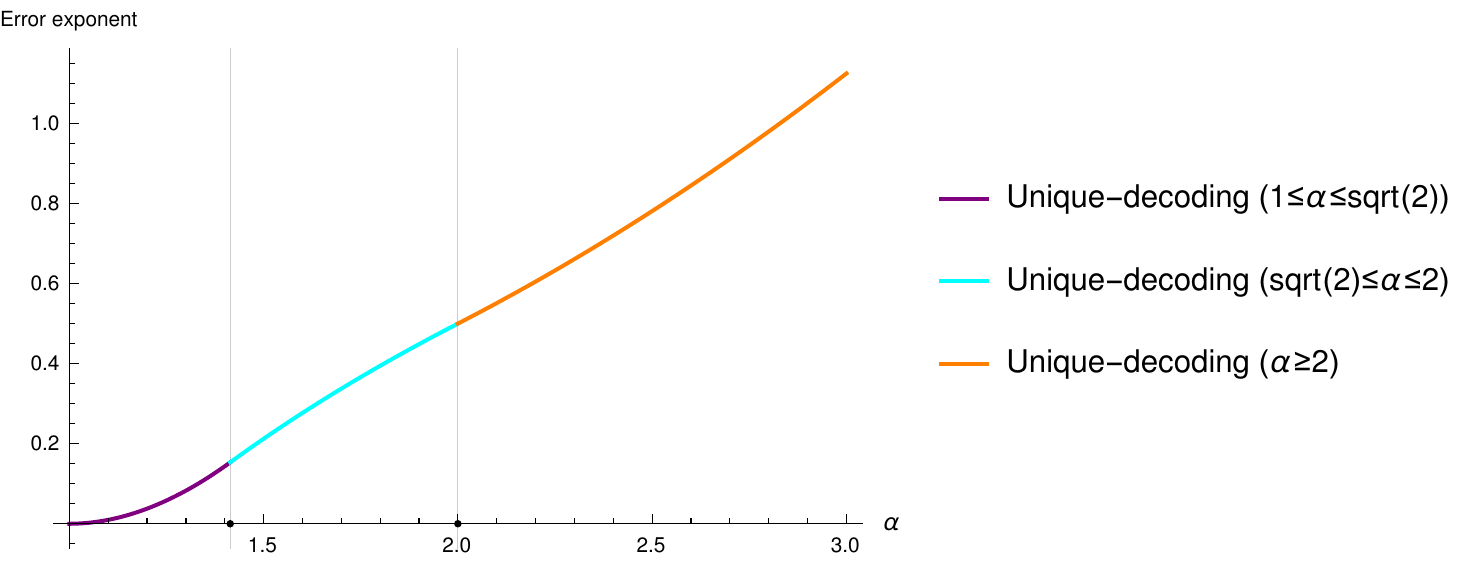}
		\caption{}
		\label{fig:UDEEunbdd}
	\end{subfigure}
	\caption{Comparison of Poltyrev's unique-decoding error exponents (\Cref{eqn:ee-ud-unbdd}) and our list-decoding error exponents (\Cref{eqn:ld-ee-ip-unconstr}) for input unconstrained AWGN channels with noise variance $ \sigma^2 $. We plot our bounds for $L=3$. The horizontal axis is $ \alpha\ge1 $ which parameterizes the rate $R$ through the relation $ R = \frac{1}{2}\ln\frac{1}{2\pi e\sigma^2\alpha^2} $. The vertical axis is the values of various exponents. The list-decoding error exponents and the unique-decoding error exponents are plotted jointly in \Cref{fig:LDEEunbdd} and are plotted separately in \Cref{fig:LDEEunbdd,fig:UDEEunbdd}, respectively. We observe that list-decoding does not increase the error exponent for any $ 1\le\alpha\le\sqrt{2} $. However, for any $\alpha>\sqrt{2}$, list-decoding does increase the error exponent. Moreover, the critical values of $ \alpha $ move from $ \sqrt{2} $ and $ 2 $ to $ \sqrt{L} $ and $ \sqrt{2L} $, respectively, under $(L-1)$-list-decoding.  
	}
	\label{fig:LDUDEEunbddFig}
\end{figure}

\section{List-decoding capacity for large $L$}
\label{sec:listdec-cap-large}

All bounds in this paper hold for any \emph{fixed} $L$. 
In this section, we discuss the impact of our finite-$L$ bounds on the understanding of the limiting values of the largest multiple packing density as $ L\to\infty $. 
Some of these results were known previously and others follow from the bounds in the current paper.

Characterizing $ C_{L-1}(P,N) $ or $ C_{L-1}(N) $ is a difficult task that is out of reach given the current techniques. 
However, if the list-size $L$ is allowed to grow, we can actually characterize 
\begin{align}
C_{\mathrm{LD}}(P,N) &\coloneqq \lim_{L\to\infty} C_{L-1}(P,N), \quad
C_{\mathrm{LD}}(N) \coloneqq \lim_{L\to\infty} C_{L-1}(N),  \notag 
\end{align}
where the subscript $ \mathrm{LD} $ denotes List-Decoding.

It is well-known that $ C_{\mathrm{LD}}(P,N) = \frac{1}{2}\ln\frac{P}{N} $. 
Specifically, the following theorem appears to be a folklore in the literature and a complete proof can be found in \cite{zhang2022quadratically}. 
\begin{theorem}[Folklore, \cite{zhang2022quadratically}]
\label{thm:listdec-cap-large-bdd}
Let $ 0<N\le P $. 
Then for any $ \eps>0 $, 
\begin{enumerate}
	\item There exist $ (P,N,L-1) $-multiple packings of rate $ \frac{1}{2}\ln\frac{P}{N} - \eps $ for some $ L = \cO\paren{\frac{1}{\eps}\ln\frac{1}{\eps}} $; 
	
	\item Any $ (P,N,L-1) $-multiple packing of rate $ \frac{1}{2}\ln\frac{P}{N} + \eps $ must satisfy $ L = e^{\Omega(n\eps)} $. 
\end{enumerate}
Therefore, $ C_{\mathrm{LD}}(P,N) = \frac{1}{2}\ln\frac{P}{N} $. 
\end{theorem}
A simple calculation reveals that \Cref{eqn:compare-lb-ee} equals $C_{\mathrm{LD}}(P,N)-\Theta(\frac{1}{L}\ln\frac{1}{L})$ for large $L$. This implies that we can construct $(P,N,L-1)$ multiple packings of rate $C_{\mathrm{LD}}(P,N)-\varepsilon$ and $L=\Theta(\frac{1}{\varepsilon}\ln\frac{1}{\varepsilon})$, thereby recovering the above result. 
It is an interesting open question to resolve whether this is indeed the right scaling.

The unbounded version $ C_{\mathrm{LD}}(N) $ is characterized in \cite{zhang2022listtrans-it} which equals $ \frac{1}{2}\ln\frac{1}{2\pi eN} $. 
\begin{theorem}[\cite{zhang2022listtrans-it}]
\label{thm:listdec-cap-large-unbdd}
Let $ N>0 $. 
Then for any $ \eps>0 $, 
\begin{enumerate}
	\item There exist $ (N,L-1) $-multiple packings of rate $ \frac{1}{2}\ln\frac{1}{2\pi eN} - \eps $ for some $ L = \cO\paren{\frac{1}{\eps}\ln\frac{1}{\eps}} $; 
	
	\item Any $ (N,L-1) $-multiple packing of rate $ \frac{1}{2}\ln\frac{1}{2\pi eN} + \eps $ must satisfy $ L = e^{\Omega(n\eps)} $. 
\end{enumerate}
Therefore, $ C_{\mathrm{LD}}(N) = \frac{1}{2}\ln\frac{1}{2\pi eN} $. 
\end{theorem}
For large $L$, our lower bound in \Cref{eqn:compare-lb-ppp} reduces to $C_{\mathrm{LD}}(N)-\Theta(\frac{1}{L}\ln\frac{1}{L})$. Once again, we get that for rates that are $\varepsilon$-close to capacity, the list size scales as $\Theta(\frac{1}{\varepsilon}\ln\frac{1}{\varepsilon})$ thereby recovering the above result.



\section{Our techniques}
\label{sec:techniques}

To derive lower bounds on list-decoding capacity, the most popular strategy is random coding with expurgation \cite{zhang-split-misc}, a standard tool from information theory. 
To show the existence of a list-decodable code of rate $R$, we can simply randomly sample $ e^{nR} $ points independently each according to a certain distribution.  
We then throw away (a.k.a.\ expurgate) one point from each of the bad lists. 
By carefully analyzing the error event and choosing a proper rate, we can guarantee that the remaining code has essentially the same rate after the removal process. 
We then get a list-decodable code of rate $R$ by noting that the remaining code contains no bad lists. 

The challenge is, however, that analyzing the error event involving the Chebyshev radius is a tricky task. 
In this paper, we take a different approach via a proxy known as the \emph{error exponent} for an AWGN channel. 
The latter quantity is the optimal exponent of the probability of list-decoding error of a code used over a Gaussian channel which inflicts an additive white Gaussian noise. 
We establish a curious inequality which relates the Chebyshev radius of lists in a code to the error exponent of the code. 
This inequality and connection originally appeared in \cite{blinovsky-1999-list-dec-real}, but some of the details were missing (see \Cref{sec:bli-mistake-ee}). 
We use different ideas to and provide a complete alternate proof in \Cref{sec:lb-ee}, which is a major contribution of this work. 
Towards this end, we provide geometric understanding of the higher-order Voronoi partition induced by $L$-lists which naturally arises as the error regions under maximum likelihood list-decoding. 
We obtain sharp estimates on the Gaussian measure of the higher-order Voronoi region associated with a list which relates the error probability to the Chebyshev radius of the list. 
This inequality bridges two quantities of fundamentally different natures. 
The Chebyshev radius is a combinatorial characteristic of a code against worst-case errors, whereas the error exponent is a probabilistic characteristic of a code against average-case errors. 
The multiple packing problem then reduces to bounding the error exponent. 

Our results on list-decoding error exponents of Gaussian channels are of independent interest beyond the study of multiple packing. 
We borrow standard techniques from information theory to prove bounds on list-decoding error exponents. 
Specifically, in the bounded case, we follow Gallager's approach \cite{gallager-1965-simple-deriv,gallager} and analyze random spherical codes; in the unbounded case, we mix the ideas in \cite{ingber-2012-finite-dim-infinite-const,anantharam-baccelli-2010-pp-long} and analyze PPPs and their expurgated versions (known as \matern processes) using tools from stochastic geometry, e.g., the Slivnyak's theorem and the Campbell's theorem. 
It has been long known that list-decoding with any subexponential (in $n$) list-sizes does not increase the capacity of any discrete memoryless channel (DMC) or Gaussian channel. 
Our results further show that list-decoding with constant list-sizes does not even improve the error exponent of capacity-achieving codes. 
In fact, for any $ L\in\bZ_{\ge2} $ and any rate $R$ above a certain critical rate $ R_{\crit,1} $ below the capacity, the $(L-1)$-list-decoding error exponent coincides with the unique-decoding error exponent (i.e., when $ L=2 $). 
However, the error exponent \emph{does} strictly increase under list-decoding when $ R $ is below $ R_{\crit,1} $. 
By carefully analyzing the aforementioned ensembles of random codes and solving delicate optimization problems coming out of the analysis, we obtain explicit bounds on the list-decoding error exponent of Gaussian channels with or without input constraints. 
These expressions, to the best of our knowledge, are not known before. 
Moreover, they recover prior results by Gallager \cite{gallager-1965-simple-deriv}, \cite[Theorem 7.4.4]{gallager} (in the bounded case) and Poltyrev \cite{poltyrev1994coding} (in the unbounded case) for $L=2$.

\section{Organization of the paper}
\label{sec:org}

This paper is a collection of lower and upper bounds on the largest multiple packing density. 
The rest of the paper is organized as follows. 
Notational conventions are listed in \Cref{sec:notation}, and some useful facts/lemmas are listed in \Cref{sec:prelim}. 
After that, we present in \Cref{sec:def} the formal definitions of multiple packing and pertaining notions. 
We also discuss different notions of density of codes used in the literature. 

In \Cref{sec:lb-ee}, we prove the inequality that relates the Chebyshev radius to error exponent and combine it with bounds on error exponent to obtain lower bounds on the largest multiple packing density. 
The bounds on error exponent used in this section are proved in \Cref{sec:ld-ee} for the bounded case and in \Cref{sec:ld-ee-unbdd} for the unbounded case. 
We end the paper with several open questions in \Cref{sec:open}.

\section{Notation}
\label{sec:notation}
\noindent\textbf{Conventions.}
Sets are denoted by capital letters in calligraphic typeface, e.g., $\cC,\cB$, etc. 
Random variables are denoted by lower case letters in boldface or capital letters in plain typeface, e.g., $\bfx,S$, etc. Their realizations are denoted by corresponding lower case letters in plain typeface, e.g., $x,s$, etc. Vectors (random or fixed) of length $n$, where $n$ is the blocklength without further specification, are denoted by lower case letters with  underlines, e.g., $\vbfx,\vbfg,\vx,\vg$, etc. 
Vectors of length different from $n$ are denoted by an arrow on top and the length will be specified whenever used, e.g., $ \vec t, \vec\alpha $, etc. 
The $i$-th entry of a vector $\vx\in\cX^n$ is denoted by $\vx(i)$ since we can alternatively think of $\vx$ as a function from $[n]$ to $\cX$. Same for a random vector $\vbfx$. Matrices are denoted by capital letters, e.g., $A,\Sigma$, etc. Similarly, the $(i,j)$-th entry of a matrix $G\in\bF^{n\times m}$ is denoted by $G(i,j)$. We sometimes write $G_{n\times m}$ to explicitly specify its dimension. For square matrices, we write $G_n$ for short. Letter $I$ is reserved for identity matrix.  

\noindent\textbf{Functions.}
We use the standard Bachmann--Landau (Big-Oh) notation for asymptotics of real-valued functions in positive integers. 

For two real-valued functions $f(n),g(n)$ of positive integers, we say that $f(n)$ \emph{asymptotically equals} $g(n)$, denoted $f(n)\asymp g(n)$, if 
\[\lim_{n\to\infty}\frac{f(n)}{g(n)} = 1.\]
For instance, $2^{n+\log n}\asymp2^{n+\log n}+2^n$, $2^{n+\log n}\not\asymp2^n$.
 We write $f(n)\doteq g(n)$ (read $f(n)$ dot equals $g(n)$) if the coefficients of the dominant terms in the exponents of $f(n)$ and $g(n)$ match,
\[\lim_{n\to\infty}\frac{\log f(n)}{\log g(n)} = 1.\]
For instance, $2^{3n}\doteq2^{3n+n^{1/4}}$, $2^{2^n}\not\doteq2^{2^{n+\log n}}$. Note that $f(n)\asymp g(n)$ implies $f(n)\doteq g(n)$, but the converse is not true.

For any $q\in\bR_{>0}$, we write $\log_q(\cdot)$ for the logarithm to the base $q$. In particular, let $\log(\cdot)$ and $\ln(\cdot)$ denote logarithms to the base $2$ and $e$, respectively.

For any $\cA\subseteq\Omega$, the indicator function of $\cA$ is defined as, for any   $x\in\Omega$, 
\[\one_{\cA}(x)\coloneqq\begin{cases}1,&x\in \cA\\0,&x\notin \cA\end{cases}.\]
At times, we will slightly abuse notation by saying that $\one_{\sfA}$ is $1$ when event $\sfA$ happens and 0 otherwise. Note that $\one_{\cA}(\cdot)=\indicator{\cdot\in\cA}$.

\noindent\textbf{Sets.}
For any two nonempty sets $\cA$ and $\cB$ with addition and multiplication by a real scalar, let $\cA+\cB$ 
denote the Minkowski sum 
of them which is defined as
$\cA+\cB\coloneqq\curbrkt{a+b\colon a\in\cA,b\in\cB}$.
If $\cA=\{x\}$ is a singleton set, we write $x+\cB$ and
 for $\{x\}+\cB$.
For any $ r\in\bR $, the $r$-dilation of $\cA$ is defined as $ r\cA \coloneqq \curbrkt{r\va:\va\in\cA} $. 
In particular, $ -\cA\coloneqq (-1)\cA $. 

For $M\in\bZ_{>0}$, we let $[M]$ denote the set of first $M$ positive integers $\{1,2,\cdots,M\}$.

\noindent\textbf{Geometry.}
Let $ \normtwo{\cdot} $ denote the Euclidean/$\ell_2$-norm. Specifically, for any $\vx\in\bR^n$,
\[ \normtwo{\vx} \coloneqq\paren{\sum_{i=1}^n\vx(i)^2}^{1/2}.\]

With slight abuse of notation, we let $|\cdot|$ denote the ``volume'' of a set w.r.t.\ a measure that is obvious from the context. 
If $ \cA $ is a finite set, then $ |\cA| $ denotes the cardinality of $\cA$ w.r.t.\ the counting measure. 
For a set $ \cA\subset\bR^n $, let
\begin{align}
\aff(\cA) &\coloneqq \curbrkt{\sum_{i = 1}^k\lambda_i\va_i:k\in\bZ_{\ge1};\;\forall i\in[k], \va_i\in\cA,\lambda_i\in\bR,\sum_{i = 1}^k\lambda_i = 1} \notag 
\end{align}
denote the \emph{affine hull} of $\cA$, i.e., the smallest affine subspace containing $\cA$. 
If $ \cA $ is a connected compact set in $ \bR^n $ with nonempty interior and $ \aff(\cA) = \bR^n $, then $ |\cA| $ denotes the volume of $\cA$ w.r.t.\ the $n$-dimensional Lebesgue measure. 
If $ \aff(\cA) $ is a $k$-dimensional affine subspace for $ 1\le k<n $, then $ |\cA| $ denotes the $k$-dimensional Lebesgue volume of $\cA$. 

The closed $n$-dimensional Euclidean unit ball is defined as
\[\cB^n \coloneqq \curbrkt{\vy\in\bR^n\colon \normtwo{\vy} \le 1}.\]
The $(n-1)$-dimensional Euclidean unit sphere is defined as
\[\cS^{n-1} \coloneqq \curbrkt{\vy\in\bR^n\colon \normtwo{\vy} = 1}.\]
For any $ \vx\in\bR^n $ and $ r\in\bR_{>0} $, let $ \cB^n(r) \coloneqq r\cB^n, \cS^{n-1}(r) \coloneqq r\cS^{n-1} $ and $ \cB^n(\vx,r) \coloneqq \vx + r\cB^n, \cS^{n-1}(\vx,r) \coloneqq \vx + r\cS^{n-1} $. 

Let $ V_n \coloneqq |\cB^n| $.

\section{Basic definitions and facts}
\label{sec:def}

Given the intimate connection between packing and error-correcting codes, we will interchangeably use the terms ``multiple packing'' and ``list-decodable code''. 
The parameter $ L\in\bZ_{\ge2} $ is called the \emph{multiplicity of overlap} or the \emph{list-size}. 
The parameters $N$ and $P$ (in the case of bounded packing) are called the \emph{input and noise power constraints}, respectively. 
Elements of a packing are called either \emph{points} or \emph{codewords}. 
We will call a size-$L$ subset of a packing an \emph{$L$-list}. 
This paper is only concerned with the fundamental limits of multiple packing for asymptotically large dimension $n\to\infty$. 
When we say ``a'' code $ \cC $, we always mean an infinite sequence of codes $ \curbrkt{\cC_i}_{i\ge1} $ where $ \cC_i\subset\bR^{n_i} $ and $ \curbrkt{n_i}_{i\ge1} $ is an increasing sequence of positive integers. 
We call $\cC$ a \emph{spherical code} if $ \cC\subset\cS^{n-1}(\sqrt{nP}) $ and we call it a \emph{ball code} if $ \cC\subset\cB^n(\sqrt{nP}) $. 

In the rest of this section, we list a sequence of formal definitions and some facts associated with these definitions. 

\begin{definition}[Bounded multiple packing]
\label{def:packing-ball}
Let $ N,P>0 $ and $ L\in\bZ_{\ge2} $. 
A subset $ \cC \subseteq \cB^n( \sqrt{nP}) $ is called a \emph{$ (P,N,L-1) $-list-decodable code} (a.k.a.\ a \emph{$(P,N,L-1)$-multiple packing}) if for every $ \vy\in\bR^n $,
\begin{align}
\card{\cC\cap\cB^n(\vy, \sqrt{nN})} \le L-1 .
\label{eqn:packing-ball}
\end{align}
The \emph{rate} (a.k.a.\ \emph{density}) of $ \cC $ is defined as
\begin{align}
R(\cC) \coloneqq \frac{1}{n}{\ln\cardC} .
\label{eqn:density-bounded}
\end{align}
\end{definition}

\begin{definition}[Unbounded multiple packing]
\label{def:packing-euclidean}
Let $ N>0 $ and $ L\in\bZ_{\ge2} $. 
A subset $ \cC \subseteq \bR^n $ is called a \emph{$ (N,L-1) $-list-decodable code} (a.k.a.\ an \emph{$(N,L-1)$-multiple packing}) if for every $ \vy\in\bR^n $,
\begin{align}
\card{\cC\cap\cB^n(\vy, \sqrt{nN})} \le L-1 .
\label{eqn:packing-ball-repeat}
\end{align}
The \emph{rate} (a.k.a.\ \emph{density}) of $ \cC $ is defined as 
\begin{align}
R(\cC) \coloneqq& \limsup_{K\to\infty} \frac{1}{n}\ln\frac{\card{\cC\cap (K\cB)}}{\card{K\cB}}, 
\label{eqn:density-unbounded}
\end{align}
where $ \cB $ is an arbitrary centrally symmetric connected compact set in $ \bR^n $ with nonempty interior. 
\end{definition}

\begin{remark}
Common choices of $ \cB $ include the unit ball $ \cB^n $, the unit cube $ [-1,1]^n $, the fundamental Voronoi region $ \cV_\Lambda $ of a (full-rank) lattice $ \Lambda\subset\bR^n $, etc. 
Some choices of $\cB$ may be more convenient than the others for analyzing certain ensembles of packings. 
Therefore, we do not fix the choice of $\cB$ in \Cref{def:packing-euclidean}. 
\end{remark}

\begin{remark}
It is a slight abuse of notation to write $ R(\cC) $ to refer to the rate of either a bounded packing or an unbounded packing. 
However, the meaning of $ R(\cC) $ will be clear from the context. 
The rate of an unbounded packing (as per \Cref{eqn:density-unbounded}) is also called the \emph{normalized logarithmic density} in the literature. 
It measures the rate (w.r.t.\ \Cref{eqn:density-bounded}) per unit volume. 
\end{remark}

Note that the condition given by \Cref{eqn:packing-ball,eqn:packing-ball-repeat} is equivalent to that for any $ (\vx_1,\cdots,\vx_L)\in\binom{\cC}{L} $, 
\begin{align}
\bigcap_{i = 1}^L\cB^n(\vx_i,\sqrt{nN}) = \emptyset. \label{eqn:packing-ball-alternative} 
\end{align}

\begin{definition}[Chebyshev radius of a list]
\label{def:cheb-rad-avg-rad}
Let $ \vx_1,\cdots,\vx_L $ be $L$ points in $ \bR^n $. 
Then the \emph{squared Chebyshev radius} $ \rad^2(\vx_1,\cdots,\vx_L) $ of $ \vx_1,\cdots,\vx_L $ is defined as the (squared) radius of the smallest ball containing $ \vx_1,\cdots,\vx_L $, i.e., 
\begin{align}
\rad^2(\vx_1,\cdots,\vx_L) \coloneqq& \min_{\vy\in\bR^n} \max_{i\in[L]} \normtwo{\vx_i - \vy}^2. \label{eqn:cheb-rad} 
\end{align}
\end{definition}

\begin{remark}
One should note that for an $L$-list $\cL$ of points, the smallest ball containing $\cL$ is not necessarily the same as the \emph{circumscribed ball}, i.e., the ball such that all points in $\cL$ live on the boundary of the ball. 
The circumscribed ball of the polytope $ \conv\curbrkt{\cL} $ spanned by the points in $\cL$ may not exist. 
If it does exist, it is not necessarily the smallest one containing $ \cL $. 
However, whenever it exists, the smallest ball containing $\cL$ must be the circumscribed ball of a certain \emph{subset} of $\cL$. 
\end{remark}

\begin{definition}[Chebyshev radius of a code]
\label{def:cheb-rad-avg-rad-code}
Given a code $ \cC\subset\bR^n $ of rate $R$, the \emph{squared $(L-1)$-list-decoding radius} of $\cC$ is defined as
\begin{align}
\rad^2_L(\cC) \coloneqq&  \min_{\cL\in\binom{\cC}{L}} \rad^2(\cL). \label{eqn:rad-code}
\end{align}
\end{definition}

Note that $ (L-1) $-list-decodability defined by \Cref{eqn:packing-ball} or \Cref{eqn:packing-ball-alternative} is equivalent to $ \rad^2_L(\cC) > nN$. 
We also define the \emph{$(P,N,L-1)$-list-decoding capacity} (a.k.a.\ \emph{$(P,N,L-1)$-multiple packing density})
$$ C_{L-1}(P,N) \coloneqq \limsup_{n\to\infty} \limsup_{\cC \subseteq \cB^n( \sqrt{nP})\colon \rad^2_L(\cC)>nN} R(\cC) ,$$ 
and the \emph{squared $(L-1)$-list-decoding radius} at rate $R$ with input constraint $P$
\begin{align}
\rad^2_L(P,R) \coloneqq& \limsup_{n\to\infty} \limsup_{\cC \subseteq \cB^n( \sqrt{nP}) \colon R(\cC)\ge R} \rad^2_L(\cC), \notag 
\end{align}
and their unbounded analogues \emph{$(N,L-1)$-list-decoding capacity} (a.k.a.\ \emph{$(N,L-1)$-multiple packing density}) $ C_{L-1}(N) $ and the \emph{squared $(L-1)$-list-decoding radius} $ \rad^2_L(R) $ at rate $R$: 
\begin{align}
C_{L-1}(N) &\coloneqq \limsup_{n\to\infty} \limsup_{\cC \subseteq \bR^n\colon \rad^2_L(\cC)>nN} R(\cC) , \notag \\
\rad^2_L(R) &\coloneqq \limsup_{n\to\infty} \limsup_{\cC \subseteq \bR^n \colon R(\cC)\ge R} \rad^2_L(\cC). \notag 
\end{align}

\section{Lower bounds on list-decoding capacity via error exponents}
\label{sec:lb-ee}
In this section, we will show the following lower bound on $ C_{L-1}(P,N) $. 
\begin{theorem}
\label{thm:lb-ee}
For any $ P,N>0 $ such that $ N\le\frac{L-1}{L}P $ and any $ L\in\bZ_{\ge2} $, the $(P,N,L-1)$-list-decoding capacity $ C_{L-1}(P,N) $ is at least 
\begin{align}
C_{L-1}(P,N) &\ge \frac{1}{2}\sqrbrkt{\ln\frac{(L-1)P}{LN} + \frac{1}{L-1}\ln\frac{P}{L(P-N)}}. \label{eqn:lb-ee} 
\end{align}
\end{theorem}

\begin{remark}
When $ L\to\infty $, the above bound (\Cref{eqn:lb-ee}) converges to the list-decoding capacity $ \frac{1}{2}\ln\frac{P}{N} $ for $L\to\infty$ (see \Cref{sec:listdec-cap-large}). 
For $ L=2 $, it recovers the best known bound $ \frac{1}{2}\ln\frac{P^2}{4N(P-N)} $ (see, e.g.,~\cite{zhang-split-misc}). 
Furthermore, it is tight at $ N/P=0 $ where the optimal density is $\infty$ and $ N/P=\frac{L-1}{L} $ where the optimal density is $ 0 $ (see~\cite{zhang-split-misc} for the \emph{Plotkin point}). 
\end{remark}

To handle the Chebyshev radius, we follow an indirect approach which relates the Chebyshev radius to a quantity called \emph{error exponent}. 
To this end, we take a detour by first introducing the notion of error exponent and then presenting bounds on it. 
We find it curious that the $(P,N,L-1)$-list-decodability against \emph{worst-case} errors can be related to the error exponent of a \emph{Gaussian} channel that only inflicts \emph{average-case} errors.

\subsection{Basic definitions regarding list-decoding error exponents}
\label{sec:def-ld-ee}

We first introduce maximum likelihood list-decoding and error exponents in the context of transmission over AWGN channels. 
Relevant definitions for more general channels can be found in \Cref{sec:ld-ee}. 

Consider a Gaussian channel $ \vbfy = \vbfx + \vbfg $ where the input $ \vbfx $ satisfies $ \normtwo{\vbfx}\le\sqrt{nP} $ and $ \vbfg\sim\cN(\vzero,\sigma^2I_n) $ is an additive white Gaussian noise with mean zero and variance $\sigma^2$. 
Let $ \cC = \curbrkt{\vx_i}_{i = 1}^M $ be a codebook for the above Gaussian channel, that is, $ \normtwo{\vx_i}\le\sqrt{nP} $ for all $ 1\le i\le M $.

We are interested in the \emph{probability of $(L-1)$-list-decoding error} of $ \cC $ under the \emph{maximum likelihood (ML) $(L-1)$-list-decoder}. 
Formally, let $ \dec^{\mathrm{ML}}_{L-1,\cC}\colon\bR^n\to\binom{\cC}{L-1} $ denote the ML $(L-1)$-list-decoder. 
Given $\vbfy$, the ML list-decoder outputs the list of the nearest $L-1$ codewords in $ \cC $ to $ \vbfy $. 
We say that an \emph{$(L-1)$-list-decoding error} occurs if the transmitted codeword $ \vx_i $ does not lie within the list $ \dec^{\mathrm{ML}}_{L-1,\cC}(\vx_i + \vbfg) $. 
Let us define $P_{\e,L-1}^{\mathrm{ML}}(i,\cC)$ to be the conditional probability of a decoding error when the $i$-th codeword is transmitted, i.e., the probability that the decoder outputs a list of codewords that does not contain $\vx_i$, conditioned on the event that $\vx_i$ was sent:
\begin{align}
P_{\e,L-1}^{\mathrm{ML}}(i,\cC) 
&\coloneqq \prob{\dec^{\mathrm{ML}}_{L-1,\cC}(\vx_i + \vbfg) \not\ni\vx_i} \notag \\
&= \prob{\exists \curbrkt{i_1,\cdots,i_{L-1}}\in\binom{[M]\setminus\curbrkt{i}}{L-1},\;\forall j\in[L-1],\;\normtwo{\vx_{i_j} - (\vx_i + \vbfg)} < \normtwo{\vbfg}}. \notag 
\end{align}
Occasionally, we also write $ P_{\e,L-1}^{\mathrm{ML}}(\vx_i,\cC) $ to denote the same quantity above.
Then, the \emph{average (over codewords) probability of $(L-1)$-list-decoding error} of $\cC$ under $ \dec^{\mathrm{ML}}_{L-1,\cC} $ is defined as 
\begin{align}
P_{\e,\avg,L-1}^{\mathrm{ML}}(\cC) &\coloneqq \frac{1}{M} \sum_{i = 1}^M P_{\e,L-1}^{\mathrm{ML}}(i,\cC). \notag 
\end{align}

\subsection{Connection between list-decoding error exponents and Chebyshev radius}
\label{sec:connection-ld-ee-cheb-rad} 

In this subsection, we present a connection between list-decoding error exponents of a code used over an AWGN channel to the Chebyshev radius of the same code. We show that the Chebyshev radius of a code can be bounded by a quantity that depends on the probability of error of the code for transmission over a suitable AWGN channel.

\begin{lemma}
\label{thm:ee-lb}
For any code $ \cC = \curbrkt{\vx_i}_{i = 1}^M $, there exists a subcode $ \cC'\subset\cC $ of size $ M'\coloneqq|\cC'|\ge M/2 $ such that for all $ \cL\in\binom{\cC'}{L} $, 
\begin{align}
P_{\e,\avg,L-1}^{\mathrm{ML}}(\cL) &\le 2 P_{\e,\avg,L-1}^{\mathrm{ML}}(\cC), \notag 
\end{align}
where 
\begin{align}
P_{\e,\avg,L-1}^{\mathrm{ML}}(\cL) &\coloneqq \frac{1}{L} \sum_{\vx\in\cL} P_{\e,L-1}^{\mathrm{ML}}(\vx,\cL), \notag
\end{align}
and 
\begin{align}
P_{\e,L-1}^{\mathrm{ML}}(\vx,\cL) &\coloneqq 
\prob{\dec^{\mathrm{ML}}_{L-1,\cL}(\vx + \vbfg) \not\ni\vx}
= \prob{\forall \vx'\in\cL\setminus\curbrkt{\vx},\;\normtwo{\vx' - (\vx + \vbfg)} < \normtwo{\vbfg}}. \notag 
\end{align}
\end{lemma}

\begin{proof}
Without loss of generality, assume that the codewords in $ \cC $ are listed according to ascending order of $ P_{\e,L-1}^{\mathrm{ML}}(i,\cC) $, that is, 
\begin{align}
P_{\e,L-1}^{\mathrm{ML}}(1,\cC) \le P_{\e,L-1}^{\mathrm{ML}}(2,\cC) \le \cdots \le P_{\e,L-1}^{\mathrm{ML}}(M,\cC). \notag 
\end{align} 
By Markov's inequality (\Cref{lem:markov}), each of the first (at least) $ M/2 $ codewords has probability of error at most $ 2 P_{\e,\avg,L-1}^{\mathrm{ML}}(\cC) $. 
Let $ \cC'\coloneqq \curbrkt{\vx_i}_{i = 1}^{M/2} \subset \cC $. 
Take any $ \cL\in\binom{\cC'}{L} $ and any $ \vx\in\cL $. 
\begin{align}
P_{\e,L-1}^{\mathrm{ML}}(\vx,\cL) &= 
\prob{\forall \vx'\in\cL\setminus\curbrkt{\vx},\;\normtwo{\vx' - (\vx + \vbfg)} < \normtwo{\vbfg}} \notag \\
&\le \prob{\bigcup_{\cL'\in\binom{\cC'\setminus\curbrkt{\vx}}{L-1}}\curbrkt{\forall\vx'\in\cL',\;\normtwo{\vx' - (\vx + \vbfg)} < \normtwo{\vbfg}}} \notag \\
&= P_{\e,L-1}^{\mathrm{ML}}(\vx,\cC') \notag \\
&\le 2P_{\e,\avg,L-1}^{\mathrm{ML}}(\cC). \notag
\end{align}
Therefore
\begin{align}
P_{\e,\avg,L-1}^{\mathrm{ML}}(\cL) &\le 2 P_{\e,\avg,L-1}^{\mathrm{ML}}(\cC), \notag 
\end{align}
which finishes the proof. 
\end{proof}

\begin{theorem}
\label{thm:blinovsky-identity}
Let $ \cL = \curbrkt{\vx_1,\cdots,\vx_L}\subset\bR^n $ be an arbitrary set of $L$ (where $L\ge2$) points in $ \bR^n $ satisfying $ (i) $ there exists a constant $ C>0 $ independent of $n$ such that $ \normtwo{\vx_i}\le\sqrt{nC} $ for all $ 1\le i\le L $; $ (ii) $ there exists a constant $ c>0 $ independent of $n$ such that $ \normtwo{\vx_i - \vx_j}\ge\sqrt{nc} $ for all $ 1\le i\ne j\le L $. 
Then 
\begin{align}
P_{\e,\avg,L-1}^{\mathrm{ML}}(\cL) &\ge \exp\paren{-\frac{\rad^2(\cL)}{2\sigma^2}-o(n)}. \label{eqn:bli-identity} 
\end{align}
\end{theorem}

Note that the case where $ L-1=1 $ is trivial which corresponds to unique-decoding. 
Indeed, suppose $ \cL = \curbrkt{\vx_1,\vx_2} $. 
Without loss of generality, assume $ \vx_1 = \vzero\in\bR^n $ and $ \vx_2 = [a,0,\cdots,0]\in\bR^n $ for some $ a \ge \sqrt{nc} $. 
It is not hard to see that
\begin{align}
P_{\e,1}^{\mathrm{ML}}(\vx_1,\cL) &= 
\prob{\normtwo{\vx_2 - (\vx_1+\vbfg)}<\normtwo{\vbfg}} \notag \\
&= \prob{\normtwo{\vx_2 - \vbfg}^2 < \normtwo{\vbfg}^2} \notag \\
&= \prob{(a - \vbfg(1))^2<\vbfg(1)^2} \notag \\
&= \prob{\vbfg(1) > a/2} \notag \\
&= \exp\paren{- \frac{(a/2)^2}{2\sigma^2} - o(n)}. \notag 
\end{align}
The last equality is by \Cref{lem:q-fn-bd}. 
By symmetry, $ P_{\e,1}^{\mathrm{ML}}(\vx_1,\cL) = P_{\e,1}^{\mathrm{ML}}(\vx_2,\cL) $ both of which are equal to $ P_{\e,\avg,1}^{\mathrm{ML}}(\cL) $. 
Since $ \sqrt{\rad^2(\curbrkt{\vx_1,\vx_2})} = \frac{1}{2}\normtwo{\vx_1 - \vx_2} = a/2 $, we see that \Cref{thm:blinovsky-identity} holds for $L-1=1$. 

We prove the above theorem in two subsequent subsections. The special case of $ L-1=2 $ is easier to handle as it exhibits a simpler geometric structure and admits more explicit calculations. We give a proof of \Cref{thm:blinovsky-identity} for this special case in \Cref{sec:bli-identity-three}. 
In fact we will prove a stronger statement:
\begin{align}
P_{\e,\avg,2}^{\mathrm{ML}}(\{\vx_1,\vx_2,\vx_3\}) &= \exp\paren{-\frac{\rad^2(\vx_1,\vx_2,\vx_3)}{2\sigma^2} - o(n)}. \notag 
\end{align}
We then prove \Cref{thm:blinovsky-identity} in \Cref{sec:bli-identity-general} for general $ L-1\ge2 $ using the Laplace's method (\Cref{thm:lap-fine}). 

\subsection{Proof of \Cref{thm:blinovsky-identity}  when $ L-1 = 2 $}
\label{sec:bli-identity-three}

\subsubsection{Voronoi partition and higher-order Voronoi partition}
\label{sec:voronoi-partition}
We first introduce the notion of a \emph{Voronoi partition} induced by a point set and its \emph{higher-order} generalization. 

Let $ \cC\subset\bR^n $ be a discrete set of points. 
The \emph{Voronoi region} $ \cV_\cC(\vx) $ associated with $ \vx\in\cC $ is defined as the region in which any point is closer to $ \vx $ than to any other points in $\cC$, i.e., 
\begin{align}
\cV_\cC(\vx) &\coloneqq \curbrkt{\vy\in\bR^n: \forall \vx'\in\cC\setminus\curbrkt{\vx},\;\normtwo{\vy - \vx'} > \normtwo{\vy - \vx} }. \notag 
\end{align}
When the underlying point set $ \cC $ is clear from the context, we write $ \cV(\vx) $ for $ \cV_\cC(\vx) $. 
Clearly, $ \cV_\cC(\vx)\cap\cV_\cC(\vx') = \emptyset $ for $ \vx\ne\vx'\in\cC $ and $ \bigcup\limits_{\vx\in\cC}\cV_\cC(\vx)$ is different from $\bR^n $ by a set of zero Lebesgue measure. 
The collection of Voronoi regions induced by $\cC$ is called the \emph{Voronoi partition} induced by $\cC$. 
It is not hard to see that for any $\cC\subset\bR^n$ and any $ \vx\in\bR^n $, the Voronoi region $ \cV_\cC(\vx) $ contains exactly one point  from $\cC$, which is $ \vx $ itself. 

Every Voronoi region can be written as an intersection of halfspaces.
To compute $ \cV(\vx) $ for any $ \vx\in\cC $, one can draw a hyperplane bisecting and perpendicular to the segment connecting $ \vx $ and $ \vx' $ for each $ \vx'\in\cC\setminus\curbrkt{\vx} $. 
Let $ \cH_{\vx'}(\vx) $ be the halfspace induced by the hyperplane that contains $\vx$, i.e., 
\begin{align}
\cH_{\vx'}(\vx) &\coloneqq \curbrkt{\vy\in\bR^n:\inprod{\vy}{\vx - \vx'} \ge \frac{\normtwo{\vx}^2 - \normtwo{\vx'}^2}{2}}. \notag 
\end{align} 
Then $ \cV(\vx) $ is nothing but the intersection of all such halfspaces, i.e., 
\begin{align}
\cV_\cC(\vx) &= \bigcap_{\vx'\in\cC\setminus\curbrkt{\vx}} \cH_{\vx'}(\vx). \notag 
\end{align}

More generally, one can define Voronoi regions associated with \emph{subsets} of points in $ \cC $. 
Let $ L\in\bZ_{\ge1} $. 
The \emph{order-$L$ Voronoi region} $ \cV_{\cC,L}(\cL) $ associated with $ \cL\in\binom{\cC}{L} $ is defined as the region such that the set of the nearest $L$ points from $\cC$ to any point in the region is $\cL$, i.e., 
\begin{align}
\cV_{\cC,L}(\cL) &\coloneqq \curbrkt{\vy\in\bR^n:\forall \vx'\in\cC\setminus\cL,\;\normtwo{\vy - \vx'} > \max_{\vx\in\cL}\normtwo{\vy - \vx} }. \label{eqn:def-high-ord-voronoi} 
\end{align}
Again, we will ignore the subscripts if they are clear. 
If $ \cL = \curbrkt{\vx} $ is a singleton set, $ \cV_{\cC,1}(\{\vx\}) = \cV_\cC(\vx) $. 
Clearly, $ \cV_{\cC,L}(\cL)\cap\cV_{\cC,L}(\cL') = \emptyset $ for $ \cL\ne\cL'\in\binom{\cC}{L} $ and $ \bigcup\limits_{\cL\in\binom{\cC}{L}}\cV_\cC(\cL) = \bR^n $ (up to a set of measure zero). 
The collection of order-$L$ Voronoi regions induced by all $L$-subsets of $\cC$ is called the \emph{order-$L$ Voronoi partition} induced by $\cC$. 

Computing the order-$L$ Voronoi partition of a point set $\cC\subset\bR^n $ is in general not easy for $ L>1 $. 
Even when $ n = 2 $, i.e., all points in $\cC$ are on a plane, the problem is not trivial and the resulting order-$L$ Voronoi partition may exhibit significantly different behaviours from the $L=1$ case \cite[Fig.\ 2-5]{lee-1982-high-order-voronoi}. 

However, if one is given the order-$(L-1)$ Voronoi partition of $\cC$ and the (first order) Voronoi partition for all sets $ \cC\setminus\cL' $ (where $ \cL'\in\binom{\cC}{L-1} $), then the order-$L$ Voronoi partition of $ \cC $ can be computed in the following way. 
For $ \cL\in\binom{\cC}{L} $, to compute $ \cV_{\cC,L}(\cL) $, for each $ \vx\in\cL $, compute the following set $ \cV_{\cC,L-1}(\cL\setminus\curbrkt{\vx})\cap\cV_{\cC\setminus\paren{\cL\setminus\curbrkt{\vx}}}(\vx) $. 
Then $ \cV_{\cC,L}(\cL) $ is nothing but their unions, i.e., 
\begin{align}
\cV_{\cC,L}(\cL) &= \bigcup_{\vx\in\cL} \cV_{\cC,L-1}(\cL\setminus\curbrkt{\vx})\cap\cV_{\cC\setminus\paren{\cL\setminus\curbrkt{\vx}}}(\vx). \notag
\end{align}

\subsubsection{Connection to list-decoding error probability for AWGN channels}
Let us return to the task of estimating the probability of $(L-1)$-list-decoding error of an $L$-list $\cL\subset\bR^n $. Given the order-$(L-1)$ Voronoi partition of $\cL$, the error probability of any $\vx\in\cL$ can be written as 
\begin{align}
P_{\e,L-1}^{\mathrm{ML}}(\vx,\cL) = \prob{\vx + \vbfg \in \cV_{\cL,L-1}(\cL\setminus\{\vx\})}, \label{eqn:voronoi-as-dec-region} 
\end{align}
i.e., the probability that $ \vx $ is the furthest point to $ \vx+\vbfg $ among $ \cC $.

Let $ \vx_1,\vx_2,\vx_2 $ be three distinct points in $ \bR^n $. 
In the proceeding two subsections, we divide the analysis of \Cref{eqn:voronoi-as-dec-region} into two cases according to the largest angle of the triangle spanned by $ \vx_1,\vx_2,\vx_3 $.

\subsubsection{Case 1: The largest angle of the triangle spanned by $ \vx_1,\vx_2,\vx_3 $ is acute or right}
\label{sec:three-points-acute-right}
As shown in \Cref{fig:voronoi-acute}, in this case, the smallest ball containing $ \vx_1,\vx_2,\vx_3 $ coincides with the circumscribed ball. 
As explained in \Cref{sec:voronoi-partition}, the Voronoi partition induced by $ \curbrkt{\vx_1,\vx_2,\vx_3} $ can be easily computed and is depicted in the first figure of \Cref{fig:voronoi-acute}. 
The second order Voronoi partition can be computed given the (first order) Voronoi partition. 
For example, $ \cV(\vx_1,\vx_2) $ is comprised of the subregion in $ \cV(\vx_1) $ whose points are closer to $ \vx_2 $ (such a subregion can be computed by computing the Voronoi partition with $ \vx_1 $ removed) and the subregion in $ \cV(\vx_2)$ whose points are closer to $ \vx_1 $ (such a subregion can be computed by computing the Voronoi partition with $ \vx_2 $ removed). 
One observes that each of the resulting second order Voronoi regions may contain no (see $ \cV(\vx_2,\vx_3) $), one (see $ \cV(\vx_1,\vx_3) $) or two points (see $ \cV(\vx_1,\vx_2) $) from the point set. 
This is in contrast with the (first order) Voronoi regions which only contain one point from the point set. 
In general, points can also be on the boundary of the higher-order Voronoi regions. 
This happens when, e.g., $ \vx_1,\vx_2,\vx_3 $ span an equilateral triangle. 
\begin{figure}[htbp]
	\centering
	\includegraphics[width=0.95\textwidth]{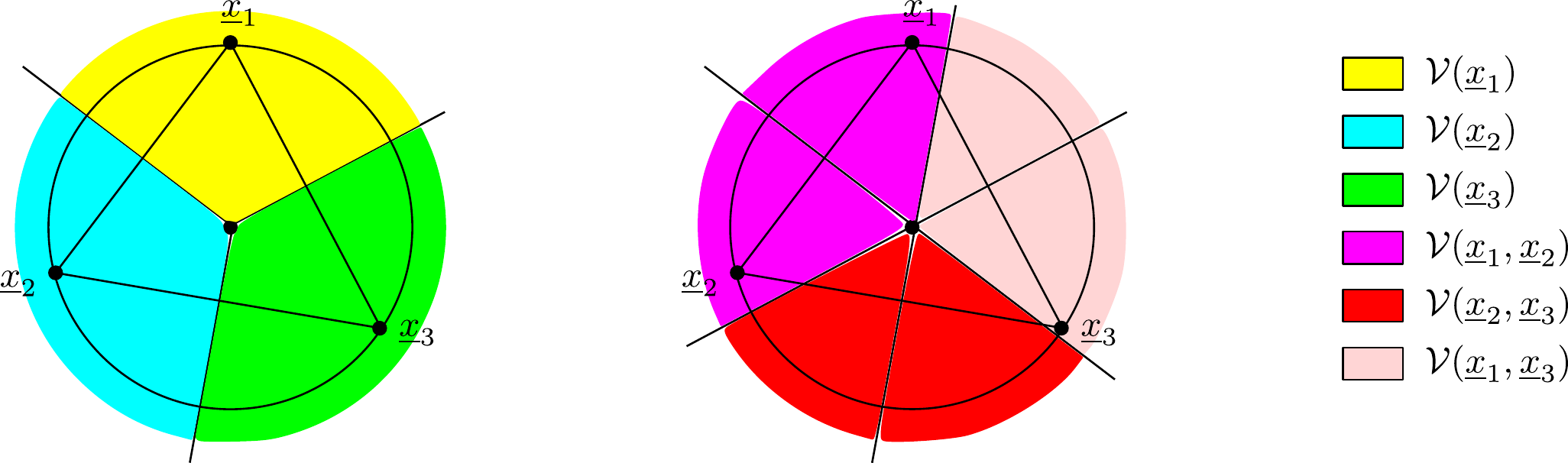}
	\caption{The Voronoi partition (left) and the second order Voronoi partition (right) of $ \{\vx_1,\vx_2,\vx_3\} $ when $ \vx_1,\vx_2,\vx_3 $ span an acute/right triangle. Note that in this case, the smallest ball containing $ \vx_1,\vx_2,\vx_3 $ coincides with the circumscribed ball. That is, all points lie on the boundary of the ball. We use the shorthand notation $ \cV(\vx_i) = \cV_{\{\vx_1,\vx_2,\vx_3\}}(\vx_i) $ and $ \cV(\vx_i,\vx_j) = \cV_{\{\vx_1,\vx_2,\vx_3\}, 2}(\{\vx_i,\vx_j\}) $. }
	\label{fig:voronoi-acute}
\end{figure}

To show \Cref{thm:blinovsky-identity} in this case, we need to estimate $ P_{\e,\avg,2}^{\mathrm{ML}}(\{\vx_1,\vx_2,\vx_3\}) $. 
Consider the plane containing $ \vx_1,\vx_2,\vx_3 $. 
As depicted in \Cref{fig:voronoi-acute-1}, let the center of the smallest ball containing $ \vx_1,\vx_2,\vx_3 $ be the origin, denoted by $ O $. 
Let the ray going from $ \vx_1 $ to $O$ be the $ x_1 $ axis and the line perpendicular to it be the $ x_2 $ axis. 
Under this parameterization, $ \normtwo{\vx_1}^2 = \normtwo{\vx_2}^2 = \normtwo{\vx_3}^2 = \rad^2(\vx_1,\vx_2,\vx_3) $ and $ \vx_1(i) = \vx_2(i) = \vx_3(i) = 0 $ for every $ 3\le i\le n $.

\begin{figure}[htbp]
	\centering
	\includegraphics[width=0.35\textwidth]{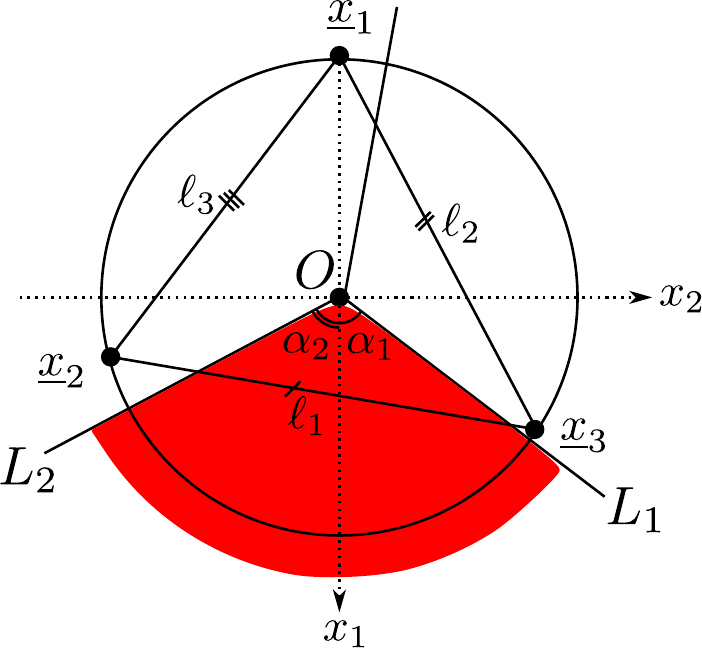}
	\caption{Suppose that the largest angle of the triangle spanned by $ \vx_1,\vx_2,\vx_3 $ is acute or right. The origin $O$ is set to be the center of the smallest ball containing $ \vx_1,\vx_2,\vx_3 $. The $ x_1 $ axis is set to be the ray going from $ \vx_1 $ to $O$ and the $ x_2 $ axis is the ray perpendicular to the $ x_1 $ axis. The circumradius coincides with the Chebyshev radius which equals $ r = \sqrt{\rad^2(\vx_1,\vx_2,\vx_3)} = \normtwo{\vx_1} = \normtwo{\vx_2} = \normtwo{\vx_3} $. The (second order) Voronoi region $ \cV(\vx_2,\vx_3) $ has two boundaries, denoted by the rays $ L_1 $ and $ L_2 $. The angle between the $ x_1 $ axis and the rays $ L_1,L_2 $ are denoted by $ \alpha_1,\alpha_2 $, respectively. The pairwise distances of $ \vx_1,\vx_2,\vx_3 $ are denoted by $ \ell_1,\ell_2,\ell_3 $. }
	\label{fig:voronoi-acute-1}
\end{figure}

Let us first estimate $ P_{\e,2}^{\mathrm{ML}}(\vx_1,\{\vx_1,\vx_2,\vx_3\}) $. 
Suppose that in the plane spanned by $ \vx_1,\vx_2,\vx_3 $, the boundaries of $ \cV(\vx_2,\vx_3) $ are given by two rays $ L_1 $ and $ L_2 $ as depicted in \Cref{fig:voronoi-acute-1}. 
It is not hard to check that if the largest angle of the triangle spanned by the three points is acute or right, then $ \cV(\vx_2,\vx_2) $ belongs to the halfspace $ \{x_2\ge0\} $ whereas $ \vx_1 $ belongs to the other halfspace $ \{x_2\le0\} $. 
Suppose $ L_1 $ and $ L_2 $ are parameterized by $ x_2 = a_1x_1 $ and $ x_2 = -a_2x_1 $ for some constants\footnote{ We explain below why the slopes $ a_1>0,a_2>0 $ must be lower bounded by some constant independent of $n$. Let $ \ell_1 \coloneqq \normtwo{\vx_2 - \vx_3},\ell_2 \coloneqq \normtwo{\vx_1 - \vx_3},\ell_3 \coloneqq\normtwo{\vx_1 - \vx_2} $. Under the assumptions in \Cref{thm:blinovsky-identity}, it is guaranteed that $ \ell_1,\ell_2,\ell_3=\Theta(\sqrt{n}) $. It is a well-known fact that the circumradius of a triangle with side lengths $ \ell_1,\ell_2,\ell_3 $ is equal to $ r=\frac{\ell_1\ell_2\ell_3}{4\sqrt{s(s - \ell_1)(s - \ell_2)(s - \ell_3)}} $ where $ s = \frac{\ell_1 + \ell_2 + \ell_3}{2} $. Under the assumptions in \Cref{thm:blinovsky-identity}, $ r = \Theta(\sqrt{n}) $. Let $ \alpha_1,\alpha_2 $ denote the angles between the $ x_1 $ axis and the rays $ L_1,L_2 $, respectively. Then $ \sin\alpha_i = \frac{a_i}{\sqrt{1+a_i^2}} $ for $ i=1,2 $. On the other hand, $ \sin\alpha_1 = \frac{\ell_3/2}{r},\sin\alpha_2 = \frac{\ell_2/2}{r} $. We therefore get the relations $ \frac{a_1}{\sqrt{1+a_1^2}} = \frac{\ell_3/2}{r},\frac{a_2}{\sqrt{1+a_2^2}} = \frac{\ell_2/2}{r} $, the RHSs of which are on the order of $ \Theta(1) $. Hence $ a_1 = \frac{\ell_3}{\sqrt{4r^2 - \ell_3^2}} = \Theta(1),a_2 = \frac{\ell_2}{\sqrt{4r^2 - \ell_2^2}} = \Theta(1) $. }
$ a_1>0,a_2>0 $ respectively. 
Let $ \cV \coloneqq \curbrkt{[x_1,x_2]\in\bR^2:x_1\ge0, -a_2x_1\le x_2\le a_1x_1} $, $ r \coloneqq \sqrt{\rad^2(\vx_1,\vx_2,\vx_3)} $ and $ a\coloneqq\max\{a_1,a_2\}>0 $. 
We are now ready to estimate $ P_{\e,2}^{\mathrm{ML}}(\vx_1,\{\vx_1,\vx_2,\vx_3\}) $. 
\begin{align}
P_{\e,2}^{\mathrm{ML}}(\vx_1,\{\vx_1,\vx_2,\vx_3\}) &= \prob{\vx_1 + \vbfg \in\cV(\vx_2,\vx_3)} \notag \\
&= \prob{[-r,0] + [\vbfg(1), \vbfg(2)]\in\cV} \notag \\
&= \prob{[\bfg_1, \bfg_2] \in [r,0] + \cV} \label{eqn:def-g1-g2} \\
&= \prob{\bfg_1\ge r, -a_2(\bfg_1 - r) \le \bfg_2 \le a_1(\bfg_1 - r)} \notag \\
&= \int_{r}^\infty \int_{-a_2(x_1 - r)}^{a_1(x_1 - r)} \frac{1}{2\pi\sigma^2}\exp\paren{-\frac{x_1^2 + x_2^2}{2\sigma^2}} \diff x_2\diff x_1 \notag \\
&\ge \int_{r}^\infty \int_{0}^{a(x_1 - r)} \frac{1}{2\pi\sigma^2}\exp\paren{-\frac{x_1^2 + x_2^2}{2\sigma^2}} \diff x_2\diff x_1 \notag \\
&= \int_{r}^\infty \frac{1}{\sqrt{2\pi\sigma^2}}\exp\paren{-\frac{x_1^2}{2\sigma^2}}\int_{0}^{a(x_1 - r)} \frac{1}{\sqrt{2\pi\sigma^2}}\exp\paren{-\frac{x_2^2}{2\sigma^2}} \diff x_2\diff x_1 \notag \\
&= \int_r^\infty \frac{1}{\sqrt{2\pi\sigma^2}}\exp\paren{-\frac{x_1^2}{2\sigma^2}}\sqrbrkt{\frac{1}{2} - \int_{a(x_1-r)}^\infty \frac{1}{\sqrt{2\pi\sigma^2}}\exp\paren{-\frac{x_2^2}{2\sigma^2}}\diff x_2} \diff x_1 \notag \\
&\approx \int_r^\infty \frac{1}{\sqrt{2\pi\sigma^2}}\exp\paren{-\frac{x_1^2}{2\sigma^2}}\sqrbrkt{\frac{1}{2} - \frac{1}{12}\exp\paren{-\frac{a^2(x_1-r)^2}{2\sigma^2}}} \diff x_1 \label{eqn:use-gauss-tail-bd} \\
&\approx \frac{1}{2}\cdot\frac{1}{12}\exp\paren{-\frac{r^2}{2\sigma^2}} - \frac{1}{12} \int_r^\infty \frac{1}{\sqrt{2\pi\sigma^2}} \exp\paren{-\frac{x_1^2 + a^2(x_1-r)^2}{2\sigma^2}}\diff x_1. \label{eqn:eqn:use-gauss-tail-bd-2}
\end{align}
In \Cref{eqn:def-g1-g2}, $ \bfg_1 $ and $ \bfg_2 $ are two independent Gaussians with mean zero and variance $ \sigma^2 $. 
In \Cref{eqn:use-gauss-tail-bd,eqn:eqn:use-gauss-tail-bd-2}, we use (twice) the bound on the $Q$-function (\Cref{lem:q-fn-bd}). 

We then proceed to estimate the integral in \Cref{eqn:eqn:use-gauss-tail-bd-2}. 
\begin{align}
& \int_r^\infty \frac{1}{\sqrt{2\pi\sigma^2}} \exp\paren{-\frac{x_1^2 + a^2(x_1-r)^2}{2\sigma^2}}\diff x_1 \notag \\
&= \int_r^\infty \frac{1}{\sqrt{2\pi\sigma^2}} \exp\paren{-\frac{1}{2\sigma^2}\paren{(1+a^2)x_1^2 - 2a^2rx_1 + a^2r^2}}\diff x_1 \notag \\
&= \int_r^\infty \frac{1}{\sqrt{2\pi\sigma^2}} \exp\paren{-\frac{1}{2\sigma^2}\sqrbrkt{\paren{\sqrt{1+a^2}x_1 - \frac{a^2r}{\sqrt{1+a^2}}}^2 + a^2r^2 - \frac{a^4r^2}{1+a^2} }}\diff x_1 \notag \\
&= \exp\paren{-\frac{1}{2\sigma^2}\sqrbrkt{a^2r^2 - \frac{a^4r^2}{1+a^2}}} \int_{\sqrt{1+a^2}r - \frac{a^2r}{\sqrt{1+a^2}}}^\infty \frac{1}{\sqrt{2\pi\sigma^2}}\exp\paren{-\frac{s^2}{2\sigma^2} }\frac{1}{\sqrt{1+a^2}}\diff s \notag \\
&\approx \frac{1}{12\sqrt{1+a^2}}\exp\paren{-\frac{1}{2\sigma^2}\sqrbrkt{a^2r^2 - \frac{a^4r^2}{1+a^2}}} \exp\paren{-\frac{1}{2\sigma^2}\sqrbrkt{\sqrt{1+a^2}r - \frac{a^2r}{\sqrt{1+a^2}}}^2} \label{eqn:eqn:use-gauss-tail-bd-3} \\
&= \frac{1}{12\sqrt{1+a^2}}\exp\paren{-\frac{r^2}{2\sigma^2}} . \notag
\end{align}
\Cref{eqn:eqn:use-gauss-tail-bd-3} follows again from \Cref{lem:q-fn-bd}.

Continuing with \Cref{eqn:eqn:use-gauss-tail-bd-2}, we have 
\begin{align}
P_{\e,2}^{\mathrm{ML}}(\vx_1,\{\vx_1,\vx_2,\vx_3\})
&\gtrsim \frac{1}{2}\cdot\frac{1}{12}\exp\paren{-\frac{r^2}{2\sigma^2}} - \frac{1}{12}\cdot\frac{1}{12\sqrt{1+a^2}}\exp\paren{-\frac{r^2}{2\sigma^2}} 
= \frac{1}{24}\paren{1 - \frac{1}{6\sqrt{1+a^2}}} \exp\paren{-\frac{r^2}{2\sigma^2}}. \notag 
\end{align}
By the geometry of the second order Voronoi partition in \Cref{fig:voronoi-acute}, the same bound also holds for $ P_{\e,2}^{\mathrm{ML}}(\vx_2,\{\vx_1,\vx_2,\vx_3\}) $ and $ P_{\e,2}^{\mathrm{ML}}(\vx_3,\{\vx_1,\vx_2,\vx_3\}) $. 
Therefore \Cref{thm:blinovsky-identity} holds in this case.

\subsubsection{Case 2: The largest angle of the triangle spanned by $ \vx_1,\vx_2,\vx_3 $ is obtuse or flat}
\label{sec:three-points-obtuse}
In this case, the largest angle of the triangle spanned by $ \vx_1,\vx_2,\vx_3 $ is obtuse or flat. 
One can similarly compute the (first order) Voronoi partition and the second order Voronoi partition induced by $ \vx_1,\vx_2,\vx_3 $, as depicted in the first and second figures of \Cref{fig:voronoi-obtuse}, respectively. 
Note that in this case the smallest ball containing all three points is \emph{different} from the circumscribed ball. 
In fact, the former one only touches \emph{two} points among three whereas the latter one by definition touches all three points and is larger than the former one. 
Note that the Chebyshev radius of the triangle is now equal to half of the length of the longest edge. 
In the example depicted in \Cref{fig:voronoi-obtuse}, $ \rad^2(\vx_1,\vx_2,\vx_3) = \paren{\frac{1}{2}\normtwo{\vx_2 - \vx_3}}^2 $. 

\begin{figure}[htbp]
	\centering
	\includegraphics[width=0.95\textwidth]{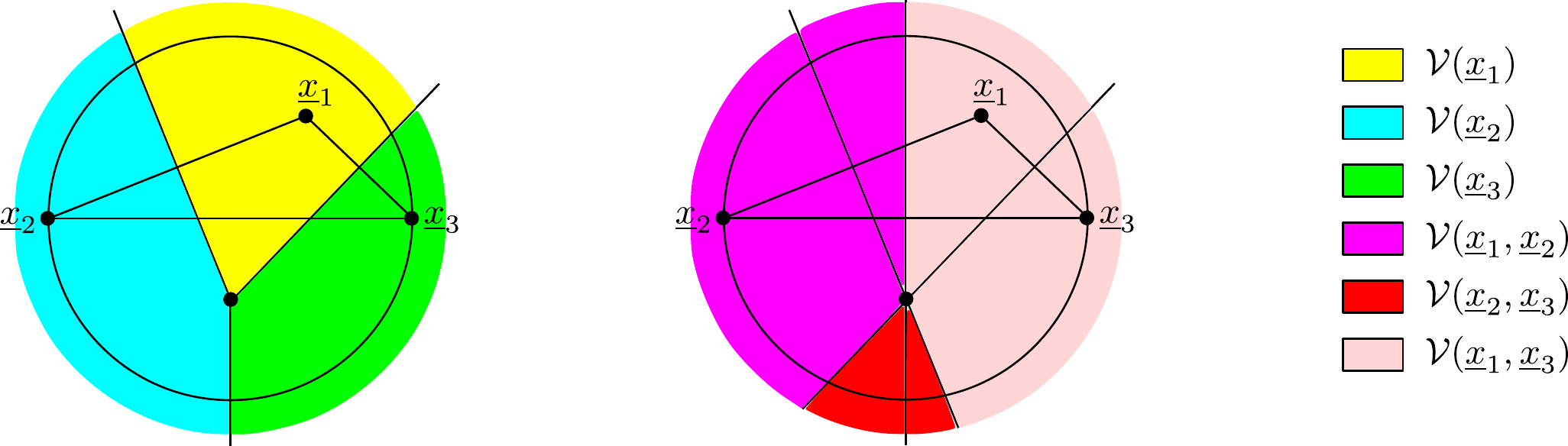}
	\caption{The Voronoi partition (left) and the second order Voronoi partition (right) of $ \{\vx_1,\vx_2,\vx_3\} $ when $ \vx_1,\vx_2,\vx_3 $ span an obtuse/flat triangle. Note that in this case, the smallest ball containing $ \vx_1,\vx_2,\vx_3 $ is strictly smaller than the circumscribed ball. In particular, the former ball only touches two points which are $ \vx_2,\vx_3 $ in the first subfigure. We use the shorthand notation $ \cV(\vx_i) = \cV_{\{\vx_1,\vx_2,\vx_3\}}(\vx_i) $ and $ \cV(\vx_i,\vx_j) = \cV_{\{\vx_1,\vx_2,\vx_3\}, 2}(\{\vx_i,\vx_j\}) $. }
	\label{fig:voronoi-obtuse}
\end{figure}

Following similar calculations as done in \Cref{sec:three-points-acute-right}, we can estimate $ P_{\e,2}^{\mathrm{ML}}(\vx_i,\{\vx_1,\vx_2,\vx_3\}) $ for each $ i = 1,2,3 $. 
Note that, as depicted in \Cref{fig:voronoi-obtuse-1}, the distance from $ \vx_2 $ to $ \cV(\vx_1,\vx_3) $ and the distance from $ \vx_3 $ to $ \cV(\vx_1,\vx_2) $ are both equal to $ \sqrt{\rad^2(\vx_1,\vx_2,\vx_3)} $, and both $ \cV(\vx_1,\vx_3) $ and $ \cV(\vx_1,\vx_2) $ contain a full quadrant. 
Therefore the same calculations as those in \Cref{sec:three-points-acute-right} yield 
\begin{align}
P_{\e,2}^{\mathrm{ML}}(\vx_2,\{\vx_1,\vx_2,\vx_3\}) &= \exp\paren{-\frac{r^2}{2\sigma^2}-o(n)}, \quad
P_{\e,2}^{\mathrm{ML}}(\vx_3,\{\vx_1,\vx_2,\vx_3\}) = \exp\paren{-\frac{r^2}{2\sigma^2}-o(n)}, \notag
\end{align}
where $ r = \sqrt{\rad^2(\vx_1,\vx_2,\vx_3)} $. 
However, the distance from $ \vx_1 $ to $ \cV(\vx_2,\vx_3) $ is strictly \emph{larger} than $ r $. 
To see this, we note that in the first subfigure of \Cref{fig:voronoi-obtuse-1}, the distance equals $ \normtwo{\vx_1} $ and $ \normtwo{\vx_1} = \normtwo{\vx_2} = \normtwo{\vx_3} $, the later two quantities of which are obviously larger than the radius of the ball. 
Hence 
\begin{align}
P_{\e,2}^{\mathrm{ML}}(\vx_1,\{\vx_1,\vx_2,\vx_3\}) &= \exp\paren{-\frac{d^2}{2\sigma^2}-o(n)}
\ll\exp\paren{-\frac{r^2}{2\sigma^2}-o(n)}, \notag
\end{align}
where $ d\coloneqq d_{\ell_2}(\vx_1,\cV(\vx_2,\vx_3)) = \normtwo{\vx_1} >\sqrt{\rad^2(\vx_1,\vx_2,\vx_3)} $. 
Overall, \Cref{thm:blinovsky-identity} still holds in this case. 

\begin{figure}[htbp]
	\centering
	\includegraphics[width=0.95\textwidth]{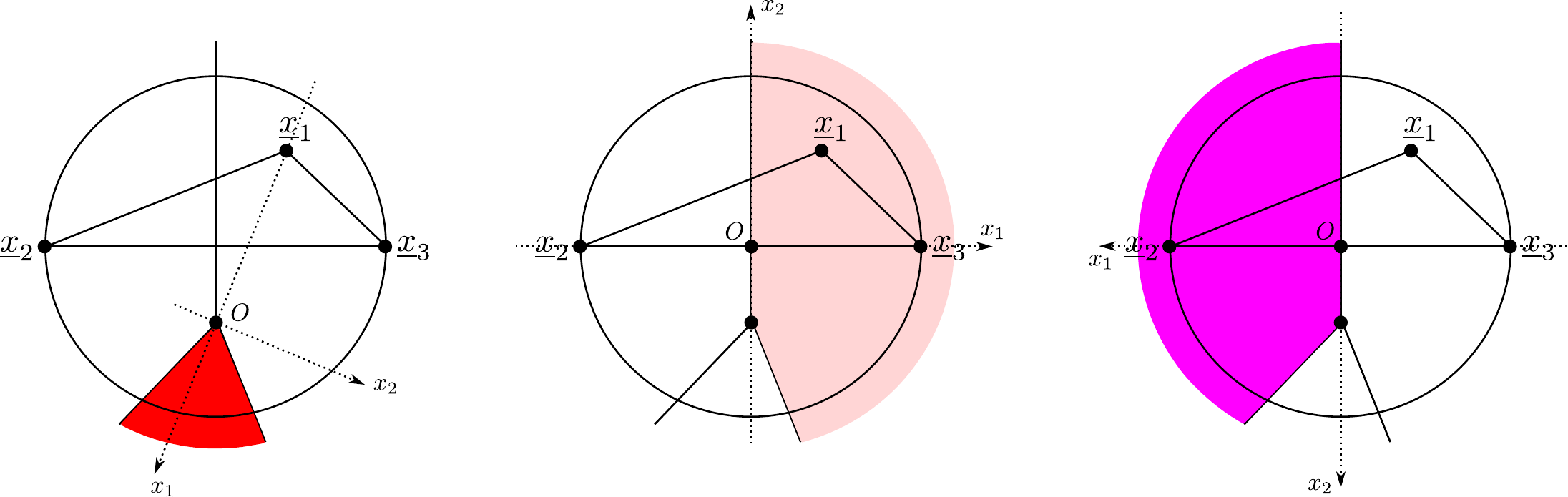}
	\caption{Suppose that $ \vx_1,\vx_2,\vx_3 $ span an obtuse/flat triangle and the length of the longest edge is given by $ \normtwo{\vx_2 - \vx_3} $. The radius of the smallest ball containing $ \vx_1,\vx_2,\vx_3 $ is equal to $ r = \frac{1}{2}\normtwo{\vx_2 - \vx_3} = \sqrt{\rad^2(\vx_1,\vx_2,\vx_3)} $. Then the distance from $ \vx_2 $ to $ \cV(\vx_1,\vx_3) $ and the distance from $ \vx_3 $ to $ \cV(\vx_1,\vx_2) $ are both equal to $ r $. However, the distance $d$ from $ \vx_1 $ to $ \cV(\vx_2,\vx_3) $ is strictly larger than $r$. }
	\label{fig:voronoi-obtuse-1}
\end{figure}

\subsection{Proof of \Cref{thm:blinovsky-identity} for general $ L-1\ge2 $}
\label{sec:bli-identity-general}
We now prove \Cref{thm:blinovsky-identity} in the general case where $ L-1\ge2 $. 
Let $ \cL\subset\bR^n $ be an arbitrary set of distinct $L$ points in $ \bR^n $. 
We assume that $\cL$ satisfies $ (i) $ a mild minimum distance condition: there exists a constant $ c>0 $ such that $ \normtwo{\vx - \vx'}\ge\sqrt{nc} $ for every distinct pair $ \vx\ne\vx' $ in $\cL$; $ (ii) $ a mild maximum norm condition: $ \cL\subset\cB^n(\sqrt{nC}) $ for some constant $ C>0 $. 
Let $ \cB_\cL $ be the smallest ball containing $\cL$. 
It is clear that there must be a point in $\cL$ that lies on the boundary of $ \cB_\cL $, otherwise $ \cB_\cL $ can be shrunk yet still contains $\cL$, which violates the minimality of $ \cB_\cL $. 
Let $ \vx_0 $ denote a point on the boundary of $ \cB_\cL $, as depicted in the first subfigure of \Cref{fig:voronoi-general}. 
\begin{figure}[htbp]
	\centering
	\includegraphics[width=0.95\textwidth]{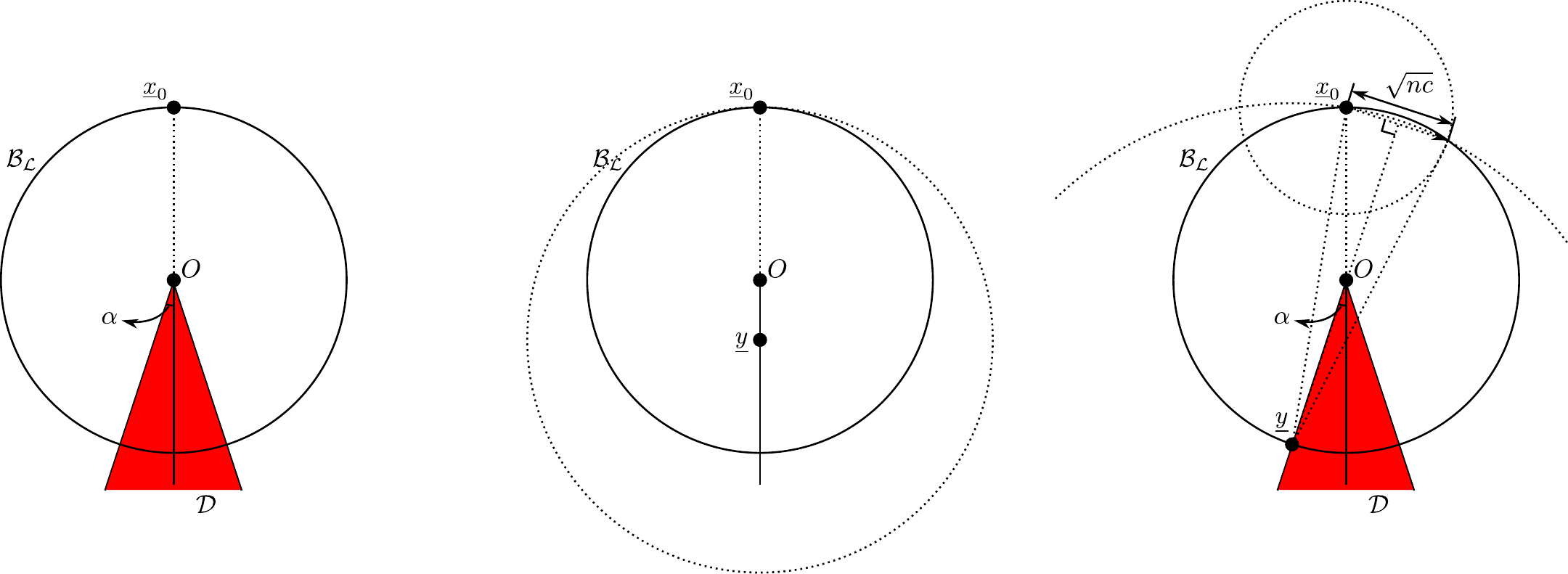}
	\caption{Suppose $ \cL\subset\bR^n $ is a set of $L$ points each of length $ \cO(\sqrt{n}) $ and the minimum pairwise distance is on the order of $ \Theta(\sqrt{n}) $. Let $ \cB_\cL $ be the smallest ball containing $ \cL $. Then there must exist a point $ \vx_0\in\cL $ on the boundary of $ \cB_\cL $. We show that the $(L-1)$-list-decoding error of $ \vx_0 $ under ML list-decoder is large. We do so by lower bounding the Gaussian measure of the ML $(L-1)$-list-decoding error region of $ \vx_0 $ by that of a cone $ \cD $ of angular radius $ \alpha $ for some constant $ \alpha>0 $. Indeed, from the geometry of the second and third subfigures, we show in \Cref{lem:cone-in-voronoi} that any received vector $\vy$ in $ \cD $ will result in a list-decoding error under ML $(L-1)$-list-decoder. }
	\label{fig:voronoi-general}
\end{figure}

Since there are only $L$ points in $ \cL $, $ \dim(\aff\curbrkt{\cL})\le L-1 $. 
By translating $\cL$ such that $ \aff(\cL) $ becomes a subspace, we can therefore parameterize $ \bR^n $ using the orthonormal basis of $ \aff\curbrkt{\cL} $ (with its extension to $ \bR^n $). 
Under this parameterization, for any $ \vx\in\cL $, we have $ \vx(i) = 0 $ for all $ L\le i\le n $. 
In the analysis we will only work with vectors in $ \bR^{L-1} $ which are obtained by restricting vectors in $ \bR^n $ to the first $L-1$ coordinates and stick with the same notation. 

As mentioned in \Cref{eqn:voronoi-as-dec-region}, for an $L$-list $\cL$, the complement of the ML $(L-1)$-list-decoding region of $ \vx_0 $ is given by the order-$(L-1)$ Voronoi region $ \cV_{\cL,L-1}(\cL\setminus\{\vx_0\}) $ of $ \cL\setminus\{\vx_0\} $. 
For $ L-1>2 $, the shape of $ \cV_{\cL,L-1}(\cL\setminus\{\vx_0\}) $ seems delicate. 
However, we manage to prove the following lemma (\Cref{lem:cone-in-voronoi}) which helps us estimate the probability that the a Gaussian noise brings $ \vx_0 $ to the ML $(L-1)$-list-decoding error region $ \cV_{\cL,L-1}(\cL\setminus\{\vx_0\}) $. 

To state the lemma, we need the following set of definitions. 
Let $ \vx_0 $ be a point in $\cL$ that lies on the boundary of $ \cB_\cL $. 
As argued above, such an $ \vx_0 $ must exist. 
Let $ O $ be the center of $ \cB_\cL $. 
We also set $O$ to be the origin of our coordinate system. 
Let $ \alpha $ be such that $ \sin\alpha = \frac{\sqrt{nc}/2}{\sqrt{\rad^2(\cL)}} $ (see the third subfigure of \Cref{fig:voronoi-general}). 
Note that under the assumptions in \Cref{thm:blinovsky-identity}, it is guaranteed that $ \alpha $ is a constant (independent of $n$).\footnote{To see this, it suffices to show $ \sqrt{\rad^2(\cL)} = \Theta(\sqrt{n}) $. Apparently, $ \sqrt{\rad^2(\cL)}\le\sqrt{nC} $ since $ \cL\subset\cB^n(\sqrt{nC}) $. Also, $ \sqrt{\rad^2(\cL)}\ge\frac{1}{2}\sqrt{nc} $ which is tight for $ L = 2 $. Therefore $ \sqrt{\rad^2(\cL)} = \Theta(\sqrt{n}) $ and $ \alpha = \sin^{-1}\frac{\sqrt{nc}/2}{\sqrt{\rad^2(\cL)}} = \Theta(1) $. } 
Let $ \cD\subset\bR^{L-1} $ be the cone of angular radius $ \alpha $ with apex at $O$ and axis along the direction of $ -\vx_0 $. 
The cone $\cD$ is depicted in \Cref{fig:voronoi-general}. 
With these parameters/objects at hands, we claim that $ \cD $ is a subset of $ \cV_{\cL,L-1}(\cL\setminus\{\vx_0\}) $ (the latter of which, by the notational convention of this section, is also a subset of $ \bR^{L-1} $ obtained by projecting the original $n$-dimensional (order-$(L-1)$) Voronoi region to its first $L-1$ coordinates).

\begin{lemma}
\label{lem:cone-in-voronoi}
Let $ C>c>0 $ be constants. 
Let $ \cL\subset\cB^n(\sqrt{nC}) $ be a set of $L$ points with minimum pairwise distance at least $ \sqrt{nc} $. 
Let $ \cB_\cL $ be the smallest ball containing $ \cL $. 
Let $ \cD\subset\bR^{L-1} $ be the $(L-1)$-dimensional cone of angular radius $ \alpha = \sin^{-1}\frac{\sqrt{nc}/2}{\sqrt{\rad^2(\cL)}} $ depicted in \Cref{fig:voronoi-general}. 
Let $ \vx_0\in\cL $ be on the boundary of $ \cB_\cL $. 
Then $ \cD\subset\cV_{\cL,L-1}(\cL\setminus\{\vx_0\}) $. 
\end{lemma}

\begin{proof}
We first note that all points on the ray shooting from $O$ along the direction of $ -\vx_0 $ are in $ \cV_{\cL,L-1}(\cL\setminus\{\vx_0\}) $. 
To see this, take any point $ \vy $ on that ray and draw a ball of radius $ \normtwo{\vx_0 - \vy} $ around $ \vy $ (see the second subfigure of \Cref{fig:voronoi-general}). 
Then $ \cB^{L-1}(O,\sqrt{\rad^2(\cL)})\subset\cB^{L-1}(\vy,\normtwo{\vx_0 - \vy}) $ and they are tangent at $ \vx_0 $. 
Therefore $ \vx_0 $ is the unique furthest point to $ \vy $ in $ \cB_\cL $. 
That is, given $ \vy $ on the ray, the ML $(L-1)$-list-decoder will \emph{not} output $ \vx_0 $. 

The above argument for the ray can be extended to hold for the cone $ \cD $ given the $\sqrt{nc}$-minimum distance guarantee. 
Clearly, to show that $\cD$ is a subset of $ \cV_{\cL,L-1}(\cL\setminus\{\vx_0\}) $, it suffices to consider points on the boundary of $ \cD $. 
Now take any point $\vy\ne O$ on the boundary of $ \cD $. 
(The case $ \vy = O $ was already handled in the above paragraph.)
Again, draw the ball $ \cB^{L-1}(\vy, \normtwo{\vx_0 - \vy}) $ (see the third subfigure of \Cref{fig:voronoi-general}). 
It is not hard to see that there is no point from $ \cL $ other than $ \vx_0 $ that is in $ \cB_\cL\setminus\cB^{L-1}(\vy,\normtwo{\vx_0 - \vy}) $, since by the $\sqrt{nc}$-minimum distance guarantee, $ \cB^{L-1}(\vx_0,\sqrt{nc})\cap\cL = \{\vx_0\} $. 
Therefore, $ \vx_0 $ is the furthest point in $ \cL $ from $\vy$, and given $\vy$, the ML $(L-1)$-list-decoder will not output $ \vx_0 $. 
This finishes the proof of the lemma. 
\end{proof}

Provided \Cref{lem:cone-in-voronoi}, we are finally ready to estimate the probability of ML $(L-1)$-list-decoding error (\Cref{eqn:voronoi-as-dec-region}). 
As before, let $ r\coloneqq\sqrt{\rad^2(\cL)} $. 
We work with polar coordinates. 
Let the apex of the cone $ \cD $ be the origin $ O $. 
Parameterize $ \vx_0 $ as $ [-r,\vu_0]\in\bR^{L-1} $ for some $ \vu_0\in\cS^{L-2} $. 
\begin{align}
P_{\e,L-1}^{\mathrm{ML}}(\vx_0, \cL)
&= \prob{\vx_0 + \vbfg \in \cV_{\cL,L-1}(\cL\setminus\{\vx_0\})} \notag \\
&\ge \prob{\vx_0 + \vbfg\in\cD} \notag \\
&= \int_{-\vx_0 + \cD} \frac{1}{(2\pi\sigma^2)^{(L-1)/2}} \exp\paren{-\frac{\normtwo{\vg}^2}{2\sigma^2}} \diff\vg \notag \\
&= \int_{\cS^{L-2}} \int_0^\infty \frac{1}{(2\pi\sigma^2)^{(L-1)/2}} \exp\paren{-\frac{\normtwo{\rho\vu}^2}{2\sigma^2}} \rho^{L-2}\cdot |\cS^{L-2}|\cdot \one_{-\vx_0+\cD}(\rho\vu) \diff\rho\diff\mu(\vu) \label{eqn:use-int-polar} \\
&= \int_{\cS^{L-2}} \int_r^\infty \frac{1}{(2\pi\sigma^2)^{(L-1)/2}} \exp\paren{-\frac{\rho^2}{2\sigma^2}} \rho^{L-2} \cdot|\cS^{L-2}|\cdot\one_{\rho^{-1}(-\vx_0+\cD)}(\vu) \diff\rho\diff\mu(\vu) \label{eqn:unit-u} \\
&= \int_r^\infty\frac{1}{(2\pi\sigma^2)^{(L-1)/2}} \exp\paren{-\frac{\rho^2}{2\sigma^2}} \rho^{L-2}\cdot|\cS^{L-2}|\cdot \paren{\int_{\cS^{L-2}}\one_{\rho^{-1}(-\vx_0+\cD)}(\vu) \diff\mu(\vu)} \diff\rho \label{eqn:interchange-int} \\
&= \int_r^\infty \frac{1}{(2\pi\sigma^2)^{(L-1)/2}} \exp\paren{-\frac{\rho^2}{2\sigma^2}} \rho^{L-2}\cdot|\cS^{L-2}|\cdot \frac{|\cS^{L-2}\cap\rho^{-1}(-\vx_0+\cD)|}{|\cS^{L-2}|}\diff\rho \label{eqn:cap-area} \\
&= \int_r^\infty\frac{1}{(2\pi\sigma^2)^{(L-1)/2}} \exp\paren{-\frac{\rho^2}{2\sigma^2}}\cdot| \cS^{L-2}(\rho)\cap(-\vx_0+\cD) | \diff\rho. \label{eqn:vol-scale} 
\end{align}
In \Cref{eqn:use-int-polar}, we switch to polar coordinates using \Cref{lem:int-polar} where $ \mu(\cdot) $ denotes the uniform probability measure on $ \cS^{L-2} $. 
\Cref{eqn:unit-u} follows since $ \normtwo{\vu}^2 = 1 $ for $ \vu\in\cS^{L-2} $ and the inner integral vanishes for any $ \rho $ such that $ \rho = \normtwo{\rho\vu}\le\normtwo{-\vx_0} = r $. 
In \Cref{eqn:interchange-int}, we interchange the inner and outer integrations. 
\Cref{eqn:cap-area} follows by noting that the inner integral is nothing but the normalized surface area of the cap obtained by taking the intersection of $ \cS^{L-2} $ and the (shifted and rescaled) cone $ \rho^{-1}(-\vx_0+\cD) $. 
\Cref{eqn:vol-scale} follows from the fact that the $(L-2)$-dimensional volume scales like $ |\rho\cS^{L-2}| = \rho^{L-2}|\cS^{L-2}| $.

To proceed, we bound the volume of the cap by first computing its radius $s = s(\rho,\alpha,r)$ as a function of $\rho$ (and $ \alpha,r $ as well). 
The geometry is depicted in \Cref{fig:cap-rad}. 
\begin{figure}[htbp]
	\centering
	\includegraphics[width=0.4\textwidth]{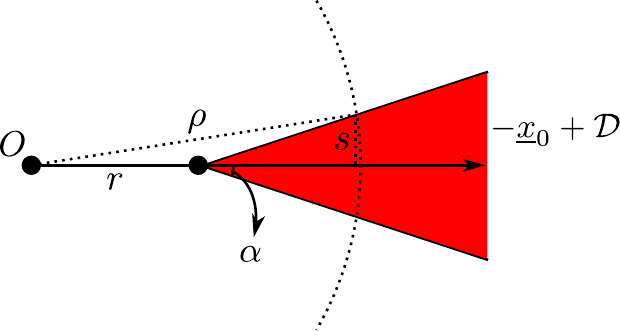}
	\caption{In the above figure, $ -\vx_0+\cD $ is a cone of angular radius $\alpha$, the apex of which is $r$ away from the origin $O$. To integrate using polar coordinates, for each radius $\rho\ge r$, we need to compute the surface measure of the cap obtained by taking the intersection of $ \cS^{L-2}(\rho) $ and $ -\vx_0+\cD $. It suffices to compute the radius $s$ of the cap. This can be done by examining the elementary geometry depicted above. }
	\label{fig:cap-rad}
\end{figure}

By Pythagorean theorem, it is not hard to see that
\begin{align}
\paren{\frac{s}{\tan\alpha} + r}^2 + s^2 &= \rho^2. \notag 
\end{align}
Solving $s$, we get
\begin{align}
s &= s(\rho,\alpha,r) = \frac{(\tan\alpha)\paren{\sqrt{(1 + \tan^2\alpha)\rho^2 - (\tan^2\alpha)r^2} - r}}{\tan^2\alpha + 1}
= (\sin\alpha)\paren{\sqrt{\rho^2 - r^2\sin^2\alpha} - r\cos\alpha}
\end{align}
Since the volume of an $(L-2)$-dimensional cap is lower bounded by that of an $(L-2)$-dimensional ball of the same radius, continuing with \Cref{eqn:vol-scale}, we have
\begin{align}
P_{\e,L-1}^{\mathrm{ML}}(\vx_0, \cL)
&\ge \int_r^\infty\frac{1}{(2\pi\sigma^2)^{(L-1)/2}} \exp\paren{-\frac{\rho^2}{2\sigma^2}}\cdot| \cB^{L-2}(s(\rho,\alpha,r)) | \diff\rho \notag \\
&= (2\pi\sigma^2)^{-(L-1)/2} V_{L-2} (\sin^{L-2}\alpha) \int_r^\infty \exp\paren{-\frac{\rho^2}{2\sigma^2}} \paren{\sqrt{\rho^2 - r^2\sin^2\alpha} - r\cos\alpha}^{L-2} \diff\rho \notag \\
&= (2\pi\sigma^2)^{-(L-1)/2} V_{L-2} (\sin^{L-2}\alpha) \int_1^\infty \exp\paren{-\frac{r^2t^2}{2\sigma^2}} \paren{\sqrt{r^2t^2 - r^2\sin^2\alpha} - r\cos\alpha}^{L-2} r\diff t \notag \\
&= (2\pi\sigma^2)^{-(L-1)/2} V_{L-2} (\sin^{L-2}\alpha) r^{L-1} \int_1^\infty \exp\paren{-\frac{r^2t^2}{2\sigma^2}} \paren{\sqrt{t^2 - \sin^2\alpha} - \cos\alpha}^{L-2} \diff t. \label{eqn:int-to-eval}
\end{align}

Define the following two functions
\begin{align}
f(t) &\coloneqq \frac{t^2}{2\sigma^2}, \quad 
g(t) \coloneqq \paren{\sqrt{t^2 - \sin^2\alpha} - \cos\alpha}^{L-2}. \notag 
\end{align}
We note that $ f'(t) = t/\sigma^2 $ and in the domain $ [1,\infty) $, $ f(t) $ attains its unique minimum $ \frac{1}{2\sigma^2} $ at $ t^* = 1 $. 
Furthermore, $ g(t^*) = g^{(1)}(t^*) = g^{(2)}(t^*) = \cdots = g^{(L-3)}(t^*) = 0 $ where $ g^{(k)}(t) $ denotes the $k$-th derivative of $g(t)$. 
However, the $(L-2)$-st derivative of $ g(t) $ does not vanish at $ t^* $ and in fact one can check that it equals
\begin{align}
g^{(L-2)}(t^*) &= \left.\frac{(L-2)!t^{L-2}}{(t^2 - \sin^2\alpha)^{(L-2)/2}}\right|_{t = t^*}
= \frac{(L-2)!}{\cos^{L-2}\alpha}. \notag 
\end{align}

Now, we apply Laplace's method (\Cref{thm:lap-fine}) to compute the integral above (\Cref{eqn:int-to-eval}). 
As $n\to\infty$, we have $ r = \Theta(\sqrt{n})\to\infty $ and therefore
\begin{align}
\int_1^\infty \exp\paren{-\frac{r^2t^2}{2\sigma^2}} \paren{\sqrt{t^2 - \sin^2\alpha} - \cos\alpha}^{L-2} \diff t
&= \int_1^\infty \exp\paren{-r^2f(t)} g(t) \diff t \notag \\
& \xrightarrow{n\to\infty} \exp\paren{-r^2f(t^*)}\cdot\frac{g^{(L-2)}(t^*)}{(r^2f^{(1)}(t^*))^{L-1}} \notag \\
&= \frac{\sigma^{2(L-1)}(L-2)!}{(\cos^{L-2}\alpha)r^{2(L-1)}}\exp\paren{-\frac{r^2}{2\sigma^2}}. \notag 
\end{align}
Putting this back to \Cref{eqn:int-to-eval}, we have, as $ n\to\infty $
\begin{align}
P_{\e,L-1}^{\mathrm{ML}}(\vx_0, \cL) 
&\gtrsim \paren{\frac{\sigma^2}{2\pi}}^{(L-1)/2}V_{L-2}(\tan^{L-2}\alpha)\frac{(L-2)!}{r^{L-1}}\exp\paren{-\frac{r^2}{2\sigma^2}}. \notag 
\end{align}
Since $ \sigma,L,\alpha $ are all constants independent of $n$, we have shown
\begin{align}
P_{\e,\avg,L-1}^{\mathrm{ML}}(\cL) 
&\ge \frac{1}{L}P_{\e,L-1}^{\mathrm{ML}}(\vx_0,\cL)
= \exp\paren{-\frac{r^2}{2\sigma^2}-o(n)}, \notag 
\end{align}
as desired.

\subsection{Putting things together}
\label{sec:lb-ee-together}

\Cref{thm:ee-lb,thm:blinovsky-identity} imply the following corollary which gives a lower bound on the error probability of a code $\cC$ in terms of the Chebyshev radius. 
\begin{corollary}
\label{cor:ee-lb}
Let $ P,\sigma>0 $ and $ L\in\bZ_{\ge2} $. 
For any code $ \cC\subset\cB^n(\sqrt{nP}) $ of size $ M $ and minimum pairwise distance at least $ \sqrt{nc} $ for some constant $ c>0 $, there exists a subcode $ \cC'\subset\cC $ of size at least $ M'\ge M/2 $ such that for all $ \cL\in\binom{\cC'}{L} $, 
\begin{align}
P_{\e,\avg,L-1}^{\mathrm{ML}}(\cC) &\ge \exp\paren{-\frac{\rad^2(\cL)}{2\sigma^2} - o(n)}. \notag 
\end{align}
\end{corollary}

On the other hand, one can construct codes whose error probability is small. 
By carefully analyzing a random code (with expurgation) in \Cref{sec:awgn-exe}, we have the following upper bound on the $(L-1)$-list-decoding error probability under ML list-decoder. 
(Many other related results on list-decoding error exponents will also be proved in \Cref{sec:awgn-exe}.)

\begin{theorem}
\label{thm:ee-ub}
Let $ P,\sigma>0 $ and $ L\in\bZ_{\ge2} $. 
There exist codes $ \cC\subset\cS^{n-1}(\sqrt{nP}) $ of rate $R$ such that when used over an AWGN channel with input constraint $P$ and noise variance $ \sigma^2 $, it attains the following expurgated error exponent under ML $(L-1)$-list-decoding. 
\begin{align}
P_{\e,\avg,L-1}^{\mathrm{ML}}(\cC) &\le 
\exp\paren{-nE_{\ex,L-1}(R)+o(n)}, \notag 
\end{align}
where 
\begin{align}
E_{\ex,L-1}(R) 
&\coloneqq -\min_{s\ge0,\rho\ge1} R (L-1)\rho - \rho\sqrbrkt{sLP + \frac{1}{2}\ln(1-2sP) + \frac{L-1}{2}\ln\paren{1-2sP + \frac{P}{\sigma^2L\rho}}}. \label{eqn:ee-ub-min} 
\end{align}
\end{theorem}

\begin{remark}
The above theorem follows from the intermediate result given by \Cref{eqn:exprg-ee-to-opt} in \Cref{sec:awgn-exe}. 
We did not take the eventual explicit expression (without the minimization) in \Cref{thm:awgn-exe} since for the purpose of this section, the minimization in \Cref{eqn:ee-ub-min} can be solved in a simpler manner when combined with \Cref{cor:ee-lb}. 
\end{remark}

\Cref{cor:ee-lb} requires the minimum distance of the code to be at least $ \sqrt{nc} $ for an arbitrarily small constant $ c>0 $. 
This turns out to be a mild condition and can be met without sacrificing the rate by taking a sufficiently small $c>0$. 
Indeed, it was shown by Shannon \cite{shannon-1959-ee-gaussian} (see also Eqn.\ (45) in \cite{swannack-erez-wornell-2013-geometric-relation}) that even under unique-decoding, no rate loss is incurred if the code is expurgated so that the minimum distance is at least $ \sqrt{nc} $ for any $ 0\le c\le c(R) $ where $ c(R) \coloneqq \sqrt{2-2\sqrt{1-e^{-2R}}} $. 
Therefore, \Cref{thm:ee-ub} continues to hold even under the $\sqrt{nc}$-minimum distance condition for any $ 0\le c\le c(R) $.

Now, combining \Cref{cor:ee-lb} and \Cref{thm:ee-ub}, we get a code $ \cC\subset\cS^{n-1}(\sqrt{nP}) $ of size $M = e^{nR} $ which contains a subcode $ \cC'\subset\cC $ of size at least $ M/2 $ satisfying: for every $ \cL\in\binom{\cC'}{L} $,
\begin{align}
\exp\paren{-\frac{\rad^2(\cL)}{2\sigma^2} - o(n)} &\le P_{\e,\avg,L-1}^{\mathrm{ML}}(\cC)
\le \exp\paren{-nE_{\ex,L-1}(R) + o(n)}. \label[ineq]{eqn:ee-ub-lb} 
\end{align}

For the subcode $\cC'$ to be $(P,N,L-1)$-list-decodable, we have $ \rad^2_L(\cC') = \min\limits_{\cL\in\binom{\cC'}{L}}\rad^2(\cL) > nN $. 
Therefore, by \Cref{eqn:ee-ub-lb}, 
\begin{align}
\frac{N}{2\sigma^2} \ge E_{\ex,L-1}(R) - o(1). 
\label[ineq]{eqn:n-vs-ee}
\end{align}
We then ignore the $o(1)$ factor and optimize out the ancillary parameters $ s $ and $\rho$ to get an explicit bound on $R$ in terms of $P,N,L$. 
To this end, let 
\begin{align}
F(s,\rho) &\coloneqq R (L-1)\rho - \rho\sqrbrkt{sLP + \frac{1}{2}\ln(1-2sP) + \frac{L-1}{2}\ln\paren{1-2sP + \frac{P}{\sigma^2L\rho}}}. \label{eqn:def-f-s-rho} 
\end{align}
The critical point $s$ and $\rho$ in the minimization of $ E_{\ex,L-1}(R) $ satisfies:
\begin{align}
\frac{\partial}{\partial s}F(s,\rho) &= -P\rho\paren{L-\frac{1}{1-2Ps} - \frac{L-1}{1-2Ps + P/(\sigma^2L\rho)}} = 0, \label{eqn:crit-s} \\
\frac{\partial}{\partial \rho}F(s,\rho) &= R(L-1) - sLP - \frac{1}{2}\ln(1-2Ps)-\frac{L-1}{2}\ln\paren{1-2Ps+\frac{P}{\sigma^2L\rho}} - \frac{(L-1)P}{2(P+L(1-2Ps)\rho\sigma^2)} = 0. \label{eqn:crit-rho}
\end{align}
From \Cref{eqn:crit-s}, we have
\begin{align}
\rho&= \frac{(L-1) - 2LPs}{2L^2s(1-2Ps)\sigma^2}. \label{eqn:rho-expr} 
\end{align}
Substitute $\rho$ into \Cref{eqn:crit-rho}, we have
\begin{align}
-\ln(1-2Ps) + (L-1)\sqrbrkt{2R - \ln\frac{(L-1)(1-2Ps)}{L(1-2Ps) - 1}} &= 0. \label{eqn:eqn-r}
\end{align}
Solving $R$ from \Cref{eqn:eqn-r}, we get 
\begin{align}
R &= \frac{1}{2}\sqrbrkt{ \ln\frac{(L-1)(1-2Ps)}{L(1-2Ps) - 1} + \frac{1}{L-1}\ln(1-2Ps)}. \label{eqn:r-expr}
\end{align}
Note that for \Cref{eqn:r-expr} to be valid, we need one additional constraint on $s$, i.e., $ s<\frac{1-1/L}{2P} $ which implies $ L(1-2Ps)-1>0 $. 
Now, putting the expressions of the critical $\rho$ (\Cref{eqn:rho-expr}) and the critical $R$ (\Cref{eqn:r-expr}) into $ F(s,\rho) $ (\Cref{eqn:def-f-s-rho}), we have
\begin{align}
F(s,\rho) &= -\frac{P(L(1-2Ps) - 1)}{2L\sigma^2(1-2Ps)}. \label{eqn:crit-f-s-rho} 
\end{align}
Substitute \Cref{eqn:crit-f-s-rho} back to the relation between $N$ and $ E_{\ex,L-1}(R) $ (\Cref{eqn:n-vs-ee}), we have 
\begin{align}
N &\ge \frac{P(L(1-2Ps) - 1)}{L(1-2Ps)}. \label[ineq]{eqn:ineq-s}
\end{align}
Note that there is no $ \sigma^2 $ in the above relation as it is cancelled out. 
Since the RHS increases as $s$ decreases, to maximize the list-decoding radius $N$, we need to take the minimum $s$. 
Therefore, we take $s$ that saturates \Cref{eqn:ineq-s}: 
\begin{align}
s &= \frac{(L-1)P - LN}{2L(P-N)P}. \label{eqn:final-s}
\end{align}
Finally, putting $s$ (\Cref{eqn:final-s}) to the expression of $R$ (\Cref{eqn:r-expr}), we get the desired bound 
\begin{align}
R &= \frac{1}{2}\sqrbrkt{\ln\frac{(L-1)P}{LN} + \frac{1}{L-1}\ln\frac{P}{L(P-N)}}. \label{eqn:ld-ee-final-rate-bd} 
\end{align}

As a sanity check, the critical value $s$ given by \Cref{eqn:final-s} is indeed nonnegative since $ N $ is less than the Plotkin point $ \frac{L-1}{L}P $. 
Also, it is not hard to check that $ s<\frac{1-1/L}{2P} $. 
Putting the critical value of $s$ (\Cref{eqn:final-s}) into the expression of $\rho$ (\Cref{eqn:rho-expr}), we get
\begin{align}
\rho &= \frac{N(P-N)}{(L(P-N) - P)\sigma^2}. \notag 
\end{align}
We note that $\rho$ is nonnegative for the same reason. 
Moreover, since $ \sigma^2 $ does not show up in the final bound on $R$, one can take a sufficiently small $\sigma^2$ to make $ \rho\ge1 $. 
In particular, it suffices to take $ 0<\sigma \le \sqrt{\frac{N(P-N)}{L(P-N) - P}} $. 
Finally, to double check, we note that the expurgated exponent given by \Cref{eqn:ee-ub-min} is achievable if $ R\le R_{\x,L-1}(\snr) $ where $ R_{\x,L-1}(\snr) $ is defined by \Cref{eqn:def-rx}. 
Since $ R_{\x,L-1}(\snr) $ is increasing as $ \snr = P/\sigma^2 $ increases, that is, as $ \sigma^2 $ decreases, the condition $ R_{\x,L-1}(\snr)\ge R $ can be satisfied if we take $ \sigma^2 $ to be sufficiently small so that $ R_{\x,L-1}(\snr) $ becomes larger than \Cref{eqn:ld-ee-final-rate-bd}. 
The exact threshold is given by
\begin{align}
\snr &\ge \frac{P(L(P-N) - P)}{N(P-N)}, \notag 
\end{align}
or
\begin{align}
\sigma &\le \sqrt{\frac{N(P-N)}{L(P-N) - P}}, \label{eqn:bound-on-awgn-var}
\end{align}
which is the same as what we obtained before. 

At a first glance, it may appear that the rate in \Cref{eqn:ld-ee-final-rate-bd} is achieved by any $\sigma$ satisfying \Cref{eqn:bound-on-awgn-var} above. It turns out that this is not true. The reason why $\sigma^2$ does not appear in the final expression is because we chose $\rho$ to maximize $R$. In this process, $\sigma^2$
was conveniently canceled out. However, a numerical evaluation of $2\sigma^2E_{\ex,L-1}$ reveals that this is in fact decreasing in $\sigma^2$, and the maximum is in fact achieved by taking $\sigma^2\to 0$.

\subsection{Connections to \cite{blinovsky-1999-list-dec-real}}
\label{sec:bli-mistake-ee}
The paper \cite{blinovsky-1999-list-dec-real} originally tried to build the connection between list-decoding radius and error exponent (\Cref{eqn:bli-identity}) and used it to obtain the same bound (\Cref{eqn:lb-ee}) as ours. 
However, there were some gaps in the proof. The proof presented in the current paper uses the same high-level idea as that presented in~\cite{blinovsky-1999-list-dec-real}, but we deviate in our approach towards characterizing the order-$(L-1)$ Voronoi regions.  

To the best of our understanding, the main idea in \cite{blinovsky-1999-list-dec-real} is to lower bound a higher-order Voronoi region (which arises as the list-decoding error region) by a (first order) Voronoi region whose Gaussian measure is then estimated. 
Therefore, \cite{blinovsky-1999-list-dec-real} takes a different perspective than ours on a higher-order Voronoi region. 
Let $ \cC\subset\bR^n $ and $ \cL\in\binom{\cC}{L-1} $. 
In \cite{blinovsky-1999-list-dec-real}, it was claimed that the order-$(L-1)$ Voronoi region $ \cV_{\cC,L-1}(\cL) $ associated with $\cL$ can be written as $ \cV_{\cC,L-1}(\cL) = \bigcap\limits_{\vx\in\cL}\cU(\vx) $ for a collection of $ \cU(\vx) $ each associated with a point $ \vx $. 
It was then claimed that $ \cU(\vx) = \cV_\cC(\vx) $, the RHS of which is the (first order) Voronoi region associated with $\vx$. 
However, it is not clear why this should be the case since different $ \cV_\cC(\vx) $'s are disjoint and their intersection is always empty. 
On the other hand, $ \cV_{\cC,L-1}(\cL) $ is never empty. 
In fact, $ \cU(\vx) $ also depends on $ \cL $ and had better be denoted by $ \cU_{\cL}(\vx) $. 
To see this, note that the original definition of $ \cV_{\cC,L-1}(\cL) $ (\Cref{eqn:def-high-ord-voronoi}) can be rewritten as 
\begin{align}
\cV_{\cC,L-1}(\cL) &= \curbrkt{\vy\in\bR^n:\forall\vx'\in\cC\setminus\cL,\forall\vx\in\cL,\;\normtwo{\vy - \vx'}>\normtwo{\vy - \vx}} \notag \\
&= \bigcap_{\vx'\in\cC\setminus\cL}\bigcap_{\vx\in\cL}\curbrkt{\vy\in\bR^n:\normtwo{\vy - \vx'}>\normtwo{\vy - \vx}} \notag \\
&= \bigcap_{\vx\in\cL}\bigcap_{\vx'\in\cC\setminus\cL}\curbrkt{\vy\in\bR^n:\normtwo{\vy - \vx'}>\normtwo{\vy - \vx}}.  \notag
\end{align}
Therefore, one can take 
\begin{align}
\cU_{\cL}(\vx) &\coloneqq \bigcap_{\vx'\in\cC\setminus\cL}\curbrkt{\vy\in\bR^n:\normtwo{\vy - \vx'}>\normtwo{\vy - \vx}},  \notag
\end{align}
and it holds that $ \cV_{\cC,L-1}(\cL) = \bigcap\limits_{\vx\in\cL}\cU_{\cL}(\vx) $.

Secondly, it was also claimed that $ \cV_{\cC,L-1}(\cL)\subset\cV_\cC(\vx) $ for any $ \vx\in\cL $. 
This seems to be inconsistent with the geometry even in the case of $L=3$. 
Indeed, in \Cref{fig:voronoi-acute,fig:voronoi-obtuse}, neither $ \cV(\vx_1,\vx_2) $ nor $ \cV(\vx_1) $ is a subset of each other. 

It is then claimed that that
\begin{align}
\prob{\vx + \vbfg\notin\cV_{\cL\cup\{\vx\},L-1}(\vx)}
\ge \prob{\vx + \vbfg\notin\cV_{\cL\cup\{\vx\}}(\vx)} 
\doteq\exp\paren{-\frac{\rad^2(\cL)}{2\sigma^2}}. \label{eqn:bli-wrong-ineq}
\end{align}
However, if we consider the example in \Cref{fig:voronoi-acute}, there seems to be an issue with the above. 
From the geometry therein, if we take $ \vx = \vx_1 $ and $ \cL = \{\vx_2,\vx_3\} $, the second probability in \Cref{eqn:bli-wrong-ineq} should be larger than the RHS since the distance from $ \vx_1 $ to $ \bR^n\setminus\cV(\vx_1) $ is strictly less than $ \sqrt{\rad^2(\vx_1,\vx_2,\vx_3)} $. 
Moreover, the reason for \cite{blinovsky-1999-list-dec-real} to look at this probability is solely a result of the preceding arguments. 
We instead study \Cref{eqn:voronoi-as-dec-region} in which is really the higher-order Voronoi region.

Finally, \cite{blinovsky-1999-list-dec-real} takes $\sigma\to 0$ in order to obtain \Cref{eqn:bli-identity}. 
In our alternate approach, this step seems to be avoided since the $\sigma^2$ term conveniently cancels out. However, as pointed out earlier, it so happens that  $2\sigma^2nE_{\ex,L-1}$ is decreasing in $\sigma^2$ in the parameter regime of interest although explicit maximization is bypassed because we chose the $\rho$ to maximize $R$. 


The fundamental difference between \cite{blinovsky-1999-list-dec-real} and the results presented above is in handling the higher-order Voronoi region.  In \cite{blinovsky-1999-list-dec-real}, an attempt is made to write the higher-order Voronoi region as the intersection of several conventional Voronoi regions. 
To the best of our understanding, there is no simple relation (even inclusions) between the conventional Voronoi partition and the higher-order Voronoi partition.

\subsection{Unbounded packings} 
\label{sec:lb-ee-unbdd}
We now adapt the techniques developed above for \emph{unbounded} packings. 
The two key ingredients are: $(i)$ a lower bound on the list-decoding error probability in terms of the Chebyshev radius; $(ii)$ an upper bound on the list-decoding error probability. 
For $ (ii) $, we have bounds in \Cref{thm:exe-matern} on the list-decoding error exponent of AWGN channels \emph{without} input constraints. Unfortunately, $(i)$ which was proved for finite codebooks cannot directly be generalized to the the setting of infinite codebooks. While \Cref{thm:blinovsky-identity} is valid for arbitrary countable codebooks, \Cref{thm:ee-lb} is true only for finite codebooks. One approach is to derive list decoding error exponents for infinite constellations under maximum probability of error. 

An easier approach is to consider a \emph{finite} codebook $\cC_a$ of sufficiently large size but restricted to lie within $ [-a/2,a/2]^n $ for a sufficiently large $a$. We construct an infinite constellation by tiling the codebook
\[
  \cC = \cC_a+a(1+o(1))\bZ^n.
\]
We then lower bound the Chebyshev radius of this infinite constellation $\cC$ with the list decoding error exponent of $\cC_a$ under maximum probability of error.

\subsection*{From infinite constellations to finite codebooks and back}

Consider any infinite constellation $\cC_\infty$ of rate $R$. Recall that
\[
  R=\limsup_{a\to\infty}\frac{1}{n}\ln\frac{|\cC_\infty\cap [-a/2,a/2]^n|}{a^n}.
\]
Fix $ a = n^2 $. Then, there exists $\vx\in\bR^n$ such that $|(\cC_\infty+\vx)\cap [-a/2,a/2]^n|\geq a^ne^{nR}$. Let us define the finite codebook 
\begin{align}
	\cC_a \coloneq (\cC_\infty+\vx)\cap [-a/2,a/2]^n , \label{eqn:def-C-a}
\end{align}
and the infinite constellation 
\begin{align}
	\cC \coloneq \cC_a + a(1+n^{-1.4})\bZ^n. \label{eqn:def-Ca-ic}
\end{align}
The above infinite constellation has rate $R(1-\delta_n)$ where $\lim_{n\to\infty}\delta_n=0$. Any two distinct shifts $\cC_a+\vz_1$ and $\cC_a+\vz_2,$ where $\vz_1\ne\vz_2\in a(1+n^{-1.4})\bZ^n$, are separated by a distance of at least $n^{0.6}$. This immediately implies the following. 
\begin{lemma}\label{lemma:chebrad_cf_c}
Let $\cC_a$ and $\cC$ be as defined in \Cref{eqn:def-C-a,eqn:def-Ca-ic}, respectively. If $ \rad^2_L(\cC_a)=\Theta(\sqrt{n}) $ (as per \Cref{def:cheb-rad-avg-rad-code}), then 
\[
  \rad^2_L(\cC_a) = \inf_{\cL\in\binom{\cC}{L}}\rad^2(\cL) . 
\]
\end{lemma}
\begin{proof}
Clearly, 
\[
  \rad^2_L(\cC_a) = \min_{\cL\in\binom{\cC_a}{L}}\rad^2(\cL) \geq \inf_{\cL\in\binom{\cC}{L}}\rad^2(\cL) . 
\]
Consider any $\cL\subset \cC$. If $\cL\subset \cC_a+\vz$ for some $\vz\in a(1+n^{-1.4})\bZ^n$, then $\rad^2(\cL)\geq \min\limits_{\cL\in\binom{\cC_a}{L}}\rad^2(\cL) = \rad^2_L(\cC_a) $. If not, then there are at least two points $\vx_1,\vx_2$ in $\cL$ such that $\vx_1\in\cC_a+\vz_1$ and $\vx_2\in\cC_a+\vz_2$ where $\vz_2\neq\vz_1\in a(1+n^{-1.4})\bZ^n$. But this implies that $\Vert\vx_1-\vx_2\Vert\geq n^{0.6}$ and $\rad^2(\cL)\geq n^{0.6}/2$. This completes the proof.
\end{proof}

Let $ \alpha\ge1 $ and $ R = \frac{1}{2}\ln\frac{1}{2\pi e\sigma^2\alpha^2} $. 
Or equivalently, $ \alpha = \sqrt{\frac{e^{-2R}}{2\pi e\sigma^2}} $. 
In \Cref{thm:exe-matern}, we prove lower bounds on the achievable expurgated list decoding error exponents of infinite constellations. This is obtained by choosing the codebook to be a Mat\'{e}rn point process derived from a Poisson point process. This means that the average probability of error is upper bounded by $\exp\paren{-nE_{\ex,L-1}(\alpha)+o(n)}$.

Let us take $\cC_\infty$ to be the Mat\'{e}rn point process above, $\cC_a=\cC_\infty\cap [-a/2,a/2]^n$ for $ a=n^2 $. Using standard tail bounds for PPPs,
\[
  \Pr[|\cC_a|<a^n(e^{nR}-n^3)]\leq \exp(-\Theta(n^2)),
\]
or
\[
  \Pr[|\cC_a|<a^ne^{nR(1-o(1))}]\leq \exp(-\Theta(n^2)).
\]
Therefore, with probability $1-e^{-\Theta(n^2)}$, the rate of $\cC$ is
\begin{align}
	R(\cC)\geq \frac{a^ne^{nR(1-o(1))}}{a^n(1+n^{-1.4})^n}=e^{nR(1-o(1))}. \label{eqn:rate-tiling}
\end{align}
Combining \Cref{eqn:rate-tiling} above with \Cref{lemma:chebrad_cf_c}, we get that for every $ \cL\in\binom{\cC}{L} $, 
\begin{align}
\exp\paren{-\frac{nN}{2\sigma^2}} \le \exp\paren{-\frac{\rad^2(\cL)}{2\sigma^2} - o(n)} &\le \exp\paren{-nE_{\ex,L-1}(\alpha)+o(n)}. \notag 
\end{align}
Hence, 
\begin{align}
N &\ge 2\sigma^2 E_{\ex,L-1}(\alpha) = 2\sigma^2 E_{\ex,L-1}\paren{\sqrt{\frac{e^{-2R}}{2\pi e\sigma^2}}}. \label[ineq]{eqn:unbdd-relation} 
\end{align}
It can be verified that the RHS as a function of $ \sigma $ is maximized at 
\begin{align}
\sigma &= \sqrt{\exp\paren{ -\frac{L}{L-1}\ln L - \ln(2\pi e) - 2R}}, \label{eqn:sigma-crit}
\end{align}
which corresponds to 
\begin{align}
\alpha &= \sqrt{\frac{e^{-2R}}{2\pi e\sigma^2}} = \sqrt{L^{\frac{L}{L-1}}}\in\sqrbrkt{\sqrt{L},\sqrt{2L}}. \notag
\end{align}
Substituting the critical $\sigma$ (\Cref{eqn:sigma-crit}) into \Cref{eqn:unbdd-relation}, we get the following inequality relating $N$ to $R$:
\begin{align}
N &\ge 2\exp\paren{ -\frac{L}{L-1}\ln L - \ln(2\pi e) - 2R}\cdot E_{\ex,L-1}\paren{\sqrt{L^{\frac{L}{L-1}}}} \notag \\
&= 2\exp\paren{ -\frac{L}{L-1}\ln L - \ln(2\pi e) - 2R}\cdot\frac{L-1}{2} \notag \\
&= (L-1)\exp\paren{-\sqrbrkt{\frac{\ln L}{L-1} + \ln(2\pi eL) + 2R}}. \notag 
\end{align}
Solving $R$, we get the following lower bound on the $(N,L-1)$-list-decoding capacity:
\begin{align}
R &= \frac{1}{2}\ln\frac{L-1}{2\pi eNL} - \frac{\ln L}{2(L-1)}, \notag 
\end{align}
We summarize our finding in the following theorem. 
\begin{theorem}
\label{thm:lb-ee-unbdd}
Let $ N>0 $ and $ L\in\bZ_{\ge2} $. 
The $ (N,L-1) $-list-decoding capacity $ C_{L-1}(N) $ is at least 
\begin{align}
C_{L-1}(N) &\ge \frac{1}{2}\ln\frac{L-1}{2\pi eNL} - \frac{\ln L}{2(L-1)}. \notag 
\end{align}
\end{theorem}

\subsection{Remark on the $\sigma^2$ that maximizes the Chebyshev radius}
\label{rk:ee-regime}
To prove \Cref{thm:lb-ee,thm:lb-ee-unbdd} for the bounded and unbounded cases, respectively, we combine \Cref{thm:blinovsky-identity} with bounds on error exponents. 
This combination then gives rise to an inequality relating $N$ to $R$. 
See \Cref{eqn:n-vs-ee,eqn:unbdd-relation} for the bounded and unbounded cases.
In \Cref{eqn:n-vs-ee}, the variance $ \sigma^2 $ of the Gaussian noise happens to cancel on both sides. 
To maximize $R$, one then needs to take the largest possible error exponent which occurs in the expurgated regime $ R\le R_{\x,L-1} $ (the latter quantity is defined in \Cref{eqn:def-rx}). 
However, in \Cref{eqn:unbdd-relation}, the Gaussian variance $ \sigma^2 $ does not cancel and one should optimize it out. 
It turns out that the optimal $ \sigma^2 $ does \emph{not} lie in the expurgated regime. 
Instead, one should use the error exponent in the ``straight line'' regime $ \sqrt{L}<\alpha\le\sqrt{2L} $ (under the parameterization of \Cref{thm:exe-matern}). 
Unfortunately we do not have intuition of this phenomenon.

\section{List-decoding error exponents}
\label{sec:ld-ee}
\subsection{DMCs with input constraints}
Consider a \emph{discrete memoryless channel (DMC)} $ W_{\bfy|\bfx}\in\Delta(\cY|\cX) $ with discrete input alphabet $ \cX $ and discrete output alphabet $ \cY $. 
The probability of the reciever seeing $\vy\in\cY^n$ at the output of the channel  when $\vx\in\cX^n$ is sent by the transmitter is equal to 
\begin{align}
\prob{\vbfy = \vy\condon\vbfx = \vx} &= \prod_{i = 1}^n W_{\bfy|\bfx}(\vy(i)|\vx(i)), \notag 
\end{align} 
for every $\vx\in\cX^n$ and $ \vy\in\cY^n $. We also impose input constraints at the transmitter.
This is specified by a set $\cP \subset\Delta(\cX) $. 
The constraints require that the empirical distribution of any codeword $ \vx\in\cX^n $ sent by the transmitter to lie within $ \cP $. 
Specifically, for $ \vx\in\cX^n $, let $ \tau_\vx\in\Delta(\cX) $ denote its \emph{empirical distribution} (a.k.a.\ \emph{histogram} or \emph{type}) defined as
\begin{align}
\tau_\vx(x) &\coloneqq \frac{1}{n}\sum_{i = 1}^n \indicator{\vx(i) = x}, \notag
\end{align}
for any $ x\in\cX $. 
Clearly, $(\tau_\vx(x) :x\in\cX)$ is a valid probability mass function on $ \cX $. 
An input sequence $ \vx\in\cX^n $ is said to satisfy the constraints $ \cP $ if $ \tau_\vx\in\cP $. 

Recall that the capacity of a DMC $ W_{\bfy|\bfx} $ with input constraints $ \cP\in\Delta(\cX) $ under unique decoding is~\cite{shannon1948mathematical}
\begin{align}
C(W_{\bfy|\bfx},\cP) &= \max_{P_\bfx\in\cP} I(\bfx;\bfy), \notag 
\end{align}
where the mutual information  
\begin{align}
I(\bfx;\bfy) \coloneqq \sum_{(x,y)\in\cX\times\cY} P_{\bfx,\bfy}(x,y)\log\frac{P_{\bfx,\bfy}(x,y)}{P_\bfx(x) P_\bfy(y)} \notag 
\end{align}
is evaluated w.r.t.\ the joint distribution $ P_{\bfx,\bfy} \coloneqq P_\bfx W_{\bfy|\bfx} $ whose marginals are denoted by $ P_\bfx $ and $ P_\bfy $. 

\subsection*{List-decoding for DMCs}
Let $ \cC = \curbrkt{\vx_i}_{i = 1}^M\subset\cX^n $ be a code satisfying the input constraints $\cP$ and equipped with an $(L-1)$-list-decoder $ \dec_{L-1,\cC}\colon\cY^n\to\binom{\cC}{L-1} $. 
The rate $ R(\cC) $ of $\cC$ is defined as $ R(\cC)\coloneqq\frac{1}{n}\log|\cC| $. 
 
We are interested in deriving upper bounds on the average probability of error when the code $\cC$ and decoder $ \dec_{L-1,\cC} $ are used for the DMC $W_{\bfy|\bfx}$. The average probability of error is defined as
\begin{align}
P_{\e,\avg,L-1}(\cC) &\coloneqq \frac{1}{M}\sum_{i = 1}^M P_{\e,L-1}(i,\cC) , \notag 
\end{align} 
where
\begin{align}
P_{\e,L-1}(i,\cC) &\coloneqq \probover{\vbfy\sim \prod_{j = 1}^n W_{\bfy|\bfx = \vx_i(j)}}{\dec_{L-1,\cC}(\vbfy)\not\ni\vx_i}. \notag 
\end{align}
Besides being of independent interest, results for this problem are used in \Cref{sec:lb-ee} to obtain bounds on the list-decoding capacity against \emph{worst-case} errors. 

We use fairly standard techniques by following Gallager's approach \cite{gallager-1965-simple-deriv}, \cite[Theorem 7.4.4]{gallager}. 
For a DMC $ W_{\bfy|\bfx} $ with input constraints $ \cP\subset\Delta(\cX) $, 
we construct a random code $\cC$ of rate $ R<C(W_{\bfy|\bfx},\cP) $ and analyze its average probability of error $ P_{\e,\avg,L-1}^{\mathrm{ML}}(\cC) $ under the ML $(L-1)$-list-decoder. 
Results obtained using this approach can be generalized to memoryless channels with \emph{continuous} alphabets, e.g., additive white Gaussian noise (AWGN) channels with power constraints.

\subsection{Random coding exponent}
\label{sec:rce}

\begin{theorem}
\label{thm:rce-dmc}
Let $ L\in\bZ_{\ge2} $. 
For any DMC $ W_{\bfy|\bfx}\in\Delta(\cY|\cX) $, there exists a sequence of codes $ \cC_n\subset\cX^n $ of increasing blocklengths, each of rate at least $R$ and satisfying
\begin{align}
\liminf_{n\to\infty}-\frac{1}{n}\log P_{\e,\avg,L-1}(\cC_n)
&\ge E_{\mathrm{r},L-1}(R), \notag 
\end{align}
where 
\begin{align}
E_{\mathrm{r}, L-1}(R) \coloneqq& \max_{P_\bfx\in\Delta(\cX)}\max_{0\le\rho\le1}\curbrkt{-(L-1)\rho R + E_{0}((L-1)\rho,P_\bfx)}, \notag
\end{align}
and
\begin{align}
E_{0}(\rho,P_\bfx) \coloneqq& -\log\sqrbrkt{\sum_{y\in\cY} \paren{\sum_{x\in\cX} P_{\bfx}(x) W_{\bfy|\bfx}(y|x)^{\frac{1}{1+\rho}}}^{1+\rho}}. \notag 
\end{align}
\end{theorem}

\begin{proof}
We use the well-known approach of~\cite{gallager}. Although these results are fairly standard, we give the proof for completeness. This will also help in generalizing the results to the input-constrained case, as well as for channels with continuous alphabet.

Let $ \cC\in\cX^{2^{nR}\times n} $ where every component of each codeword are drawn i.i.d.\ according to some $ P_{\bfx}\in\Delta(\cX) $. 
Let $ M\coloneqq 2^{nR} $. 
We use the ML $(L-1)$-list-decoder $ \dec^{\mathrm{ML}}_{L-1,\cC}\colon\cY^n\to\binom{[M]}{L-1} $. 
That is, receiving $ \vy\in\cY^n $, the decoder outputs a list $ \curbrkt{m_1,\cdots,m_{L-1}}\in\binom{[M]}{L} $ such that for any other $ m'\in[M]\setminus\curbrkt{m_1,\cdots,m_{L-1}} $,
\begin{align}
W_{\bfy|\bfx}^{\tn}(\vy|\vx_{m'})\le\min\curbrkt{W_{\bfy|\bfx}^{\tn}(\vy|\vx_{m_1}),\cdots,W_{\bfy|\bfx}^{\tn}(\vy|\vx_{m_{L-1}})}. \notag 
\end{align}
An error occurs if $ \vx_m $ was transmitted but $ m\notin\curbrkt{m_1,\cdots,m_{L_1}} $. 
For every $\rho>0$ and $s>0$, the error indicator function can be bounded as follows
\begin{align}
\indicator{\dec^{\mathrm{ML}}_{L-1,\cC}(\vbfy)\not\ni m}
&= \paren{\indicator{\exists \cL\in\binom{[M]\setminus\curbrkt{m}}{L-1},\forall i\in\cL, W_{\bfy|\bfx}^{\tn}(\vy|\vx_{i}) > W_{\bfy|\bfx}^{\tn}(\vy|\vx_m)}}^\rho \notag \\
&\leq \paren{\sum_{\cL\in\binom{[M]\setminus\curbrkt{m}}{L-1}} \prod_{i\in\cL} \indicator{ \paren{\frac{W_{\bfy|\bfx}^{\tn}(\vy|\vx_{i})}{W_{\bfy|\bfx}^{\tn}(\vy|\vx_m)}}^s > 1^s}}^\rho \notag \\
&\le \paren{\sum_{\cL\in\binom{[M]\setminus\curbrkt{m}}{L-1}} \prod_{i\in\cL} \paren{\frac{W_{\bfy|\bfx}^{\tn}(\vy|\vx_{i})}{W_{\bfy|\bfx}^{\tn}(\vy|\vx_m)}}^s}^\rho. \label[ineq]{eqn:indicator-bd}
\end{align}
\Cref{eqn:indicator-bd} follows from $
\indicator{a>1}<a, \label[ineq]{eqn:elem-ineq-indicator} $
for any $ a>0 $.

For any message $m$, the probability that $m$ is incorrectly list-decoded is 
\begin{align}
P^{\mathrm{ML}}_{\e,L-1}(m,\cC) &= \sum_{\vy\in\cY^n} W_{\bfy|\bfx}^{\tn}(\vy|\vx_m) \indicator{\dec^{\mathrm{ML}}_{L-1,\cC}(\vy)\notni m} \notag \\
&\le \sum_{\vy\in\cY^n} W_{\bfy|\bfx}^{\tn}(\vy|\vx_m) \paren{\sum_{\cL\in\binom{[M]\setminus\curbrkt{m}}{L-1}} \prod_{i\in\cL} \paren{\frac{W_{\bfy|\bfx}^{\tn}(\vy|\vx_{i})}{W_{\bfy|\bfx}^{\tn}(\vy|\vx_m)}}^s}^\rho \notag \\
&= \sum_{\vy\in\cY^n} W_{\bfy|\bfx}^{\tn}(\vy|\vx_m)^{1-s(L-1)\rho} \paren{\sum_{\cL\in\binom{[M]\setminus\curbrkt{m}}{L-1}} \prod_{i\in\cL} W_{\bfy|\bfx}^{\tn}(\vy|\vx_{i})^s}^\rho. \label{eqn:pem-unaveraged} 
\end{align}
Averaged over the random generation of $ \cC $, the error probability is 
\begin{align}
\exptover{\cC}{P^{\mathrm{ML}}_{\e,L-1}(m,\cC)} 
&\le \sum_{\vy\in\cY^n} \exptover{\cC}{W_{\bfy|\bfx}^{\tn}(\vy|\vbfx_m)^{1-s(L-1)\rho}} \paren{\sum_{\cL\in\binom{[M]\setminus\curbrkt{m}}{L-1}} \exptover{\cC}{\prod_{i\in\cL} W_{\bfy|\bfx}^{\tn}(\vy|\vbfx_{i})^s}}^\rho \label[ineq]{eqn:concave} \\
&= \sum_{\vy\in\cY^n} \exptover{\cC}{W_{\bfy|\bfx}^{\tn}(\vy|\vbfx_m)^{1-s(L-1)\rho}} \paren{\sum_{\cL\in\binom{[M]\setminus\curbrkt{m}}{L-1}} \prod_{i\in\cL} \exptover{\cC}{W_{\bfy|\bfx}^{\tn}(\vy|\vbfx_{i})^s}}^\rho, \label{eqn:ee-to-be-cont} 
\end{align}
where we have used the fact that each codeword in $\cC$ is independently generated. 
\Cref{eqn:concave} is valid when $ 0<\rho\le1 $ since $ (\cdot)^\rho $ is concave. 
We also use linearity of expectation here. 
For any $ i\in[M] $ and $ t>0 $, 
\begin{align}
\exptover{\cC}{W_{\bfy|\bfx}^{\tn}(\vy|\vbfx_{i})^t} &= 
\sum_{\vx\in\cX^n} P_{\bfx}^{\tn}(\vx) W_{\bfy|\bfx}^{\tn}(\vy|\vx)^t, \notag 
\end{align}
which is independent of $i$. 
Therefore, \Cref{eqn:ee-to-be-cont} equals
\begin{align}
& \sum_{\vy\in\cY^n} \paren{\sum_{\vx\in\cX^n} P_{\bfx}^{\tn}(\vx) W_{\bfy|\bfx}^{\tn}(\vy|\vx)^{1-s(L-1)\rho}} 
\paren{\sum_{\cL\in\binom{[M]\setminus\curbrkt{m}}{L-1}} \paren{\sum_{\vx\in\cX^n} P_{\bfx}^{\tn}(\vx) W_{\bfy|\bfx}^{\tn}(\vy|\vx)^{s}}^{L-1}}^\rho \notag \\
&= \binom{M-1}{L-1}^\rho \sum_{\vy\in\cY^n} \paren{\sum_{\vx\in\cX^n} P_{\bfx}^{\tn}(\vx) W_{\bfy|\bfx}^{\tn}(\vy|\vx)^{1-s(L-1)\rho}} 
\paren{\sum_{\vx\in\cX^n} P_{\bfx}^{\tn}(\vx) W_{\bfy|\bfx}^{\tn}(\vy|\vx)^{s}}^{(L-1)\rho}. \notag 
\end{align}
Letting $ s = \frac{1}{1+(L-1)\rho} $ and using $\binom{n}{k}\leq n^k$, we get
\begin{align}
& \binom{M-1}{L-1}^\rho \sum_{\vy\in\cY^n} \paren{\sum_{\vx\in\cX^n} P_{\bfx}^{\tn}(\vx) W_{\bfy|\bfx}^{\tn}(\vy|\vx)^{\frac{1}{1+(L-1)\rho}}}^{1+(L-1)\rho} \notag \\
&\le M^{(L-1)\rho} \sum_{\vy\in\cY^n} \paren{\sum_{(x_1,\cdots,x_n)\in\cX^n} \prod_{i = 1}^n P_{\bfx}(x_i) W_{\bfy|\bfx}(\vy(i)|x_i)^{\frac{1}{1+(L-1)\rho}}}^{1+(L-1)\rho} \notag \\
&= M^{(L-1)\rho} \prod_{i = 1}^n \sum_{y_i\in\cY} \paren{\sum_{x_i\in\cX} P_{\bfx}(x_i) W_{\bfy|\bfx}(y_i|x_i)^{\frac{1}{1+(L-1)\rho}}}^{1+(L-1)\rho} \notag \\
&= M^{(L-1)\rho} \sqrbrkt{\sum_{y\in\cY} \paren{\sum_{x\in\cX} P_{\bfx}(x) W_{\bfy|\bfx}(y|x)^{\frac{1}{1+(L-1)\rho}}}^{1+(L-1)\rho}}^n \label{eqn:rce-to-incorp-constr} \\
&= \exp_2\paren{-n\sqrbrkt{-(L-1)\rho R + E_{0}((L-1)\rho,P_\bfx)}}, \notag 
\end{align}
where 
\begin{align}
E_{0}(\rho,P_\bfx) \coloneqq& -\log\sqrbrkt{\sum_{y\in\cY} \paren{\sum_{x\in\cX} P_{\bfx}(x) W_{\bfy|\bfx}(y|x)^{\frac{1}{1+\rho}}}^{1+\rho}}. \notag 
\end{align}
Optimizing over $ 0\le\rho\le1 $, we get the random coding exponent
\begin{align}
E_{\mathrm{r}, L-1}(R,P_\bfx) \coloneqq& \max_{0\le\rho\le1}\curbrkt{-(L-1)\rho R + E_{0}((L-1)\rho,P_\bfx)}. \qedhere
\end{align}
\end{proof}

\subsection{Expurgated exponent}
\label{sec:expurg-expo}

\begin{theorem}
\label{thm:exe-dmc}
Let $ L\in\bZ_{\ge2} $. 
For any DMC $ W_{\bfy|\bfx}\in\Delta(\cY|\cX) $, there exists a sequence of codes $ \cC_n\subset\cX^n $ of increasing blocklengths, each of rate at least $R$ and satisfying
\begin{align}
\liminf_{n\to\infty}\left[-\frac{1}{n}\log P_{\e,\avg,L-1}(\cC_n) \right]
&\ge E_{\ex,L-1}(R), \notag 
\end{align}
where 
\begin{align}
E_{\ex,L-1}(R) \coloneqq& \max_{P_\bfx\in\Delta(\cX)} \max_{\rho\ge1} \curbrkt{-(L-1)\rho R + E_{\x,L-1}(\rho,P_\bfx)}, \notag
\end{align}
and 
\begin{align}
E_{\x,L-1}(\rho,P_\bfx) \coloneqq& - \rho \log\sqrbrkt{\sum_{(x_0,\cdots,x_{L-1})\in\cX^{L}}
\paren{\prod_{k = 0}^{L-1} P_\bfx(x_{k})}
\paren{ \sum_{y\in\cY} \prod_{i = 0}^{L-1} W_{\bfy|\bfx}(y|x_{i})^{1/L}}^{1/\rho}}. \notag
\end{align}
\end{theorem}

\begin{proof}
	We begin by considering a random codebook and ML decoder as in the proof of \Cref{thm:rce-dmc}. However, we will choose a different $\rho,s>0$ and later expurgate the codebook.
	
Taking $ \rho = 1 $ and $ s = 1/L $ in \Cref{eqn:pem-unaveraged}, we have
\begin{align}
P^{\mathrm{ML}}_{\e,L-1}(m_0,\cC) &\le \sum_{\vy\in\cY^n} W_{\bfy|\bfx}^{\tn}(\vy|\vx_{m_0})^{1/L} \sum_{\cL\in\binom{[M]\setminus\curbrkt{m}}{L-1}} \prod_{i\in\cL} W_{\bfy|\bfx}^{\tn}(\vy|\vx_{i})^{1/L} \notag \\
&= \sum_{\curbrkt{m_1,\cdots,m_{L-1}}\in\binom{[M]\setminus\curbrkt{m_0}}{L-1}} \sum_{(y_1,\cdots,y_n)\in\cY^n}
\prod_{j = 1}^n \prod_{i = 0}^{L-1} W_{\bfy|\bfx}(y_j|\vx_{m_i}(j))^{1/L} \notag \\
&= \sum_{\curbrkt{m_1,\cdots,m_{L-1}}\in\binom{[M]\setminus\curbrkt{m_0}}{L-1}} q(\vx_{m_0},\vx_{m_1},\cdots,\vx_{m_{L-1}}), \label{eqn:def-q}
\end{align}
where in \Cref{eqn:def-q} we define
\begin{align}
q(\vx_{m_0},\vx_{m_1},\cdots,\vx_{m_{L-1}}) \coloneqq& \prod_{j = 1}^n \sum_{y_j\in\cY} \prod_{i = 0}^{L-1} W_{\bfy|\bfx}(y_j|\vx_{m_i}(j))^{1/L}. \notag 
\end{align}
Now for any $ B>0 $,
\begin{align}
\probover{\cC}{P^{\mathrm{ML}}_{\e,L-1}(m_0,\cC) > B} &= \exptover{\cC}{\indicator{P^{\mathrm{ML}}_{\e,L-1}(m_0,\cC) > B}} \notag \\
&\le \exptover{\cC}{\indicator{\sum_{\curbrkt{m_1,\cdots,m_{L-1}}\in\binom{[M]\setminus\curbrkt{m_0}}{L-1}} \abs{\frac{q(\vbfx_{m_0},\vbfx_{m_1},\cdots,\vbfx_{m_{L-1}})}{B}} > 1}} \notag \\
&\le \exptover{\cC}{\indicator{\paren{\sum_{\curbrkt{m_1,\cdots,m_{L-1}}\in\binom{[M]\setminus\curbrkt{m_0}}{L-1}} \abs{\frac{q(\vbfx_{m_0},\vbfx_{m_1},\cdots,\vbfx_{m_{L-1}})}{B}}^s}^{1/s} > 1}} \label{eqn:norm-ineq} \\
&= \exptover{\cC}{\indicator{{\sum_{\curbrkt{m_1,\cdots,m_{L-1}}\in\binom{[M]\setminus\curbrkt{m_0}}{L-1}} \frac{q(\vbfx_{m_0},\vbfx_{m_1},\cdots,\vbfx_{m_{L-1}})^s}{B^s}} > 1}} \notag \\
&\le \exptover{\cC}{\sum_{\curbrkt{m_1,\cdots,m_{L-1}}\in\binom{[M]\setminus\curbrkt{m_0\textbf{}}}{L-1}} \frac{q(\vbfx_{m_0},\vbfx_{m_1},\cdots,\vbfx_{m_{L-1}})^s}{B^s}} \label[ineq]{eqn:elem-indicator-again} \\
&= \sum_{\curbrkt{m_1,\cdots,m_{L-1}}\in\binom{[M]\setminus\curbrkt{m_0}}{L-1}} B^{-s}\exptover{\cC}{q(\vbfx_{m_0},\vbfx_{m_1},\cdots,\vbfx_{m_{L-1}})^s}. \label{eqn:expurg-to-be-cont}
\end{align}
\Cref{eqn:norm-ineq} is valid for any $ 0\le s\le1 $ since $ \norm{s}{\cdot} $ is decreasing in $s$. 
\Cref{eqn:elem-indicator-again} follows from \Cref{eqn:elem-ineq-indicator}. 
We then bound the above expectation.
\begin{align}
\exptover{\cC}{q(\vbfx_{m_0},\vbfx_{m_1},\cdots,\vbfx_{m_{L-1}})^s} 
&= \sum_{(\vx_0,\cdots,\vx_{L-1})\in(\cX^n)^{L}} 
\sqrbrkt{\prod_{j = 1}^n \prod_{k = 0}^{L-1} P_\bfx(\vx_k(j))} 
\sqrbrkt{\prod_{j = 1}^n \paren{ \sum_{y_j\in\cY} \prod_{i = 0}^{L-1} W_{\bfy|\bfx}(y_j|\vx_{i}(j))^{1/L}}^s} \notag \\
&= \sum_{(\vx_0,\cdots,\vx_{L-1})\in(\cX^n)^{L}} \prod_{j = 1}^n
\paren{\prod_{k = 0}^{L-1} P_\bfx(\vx_k(j))}
\paren{ \sum_{y_j\in\cY} \prod_{i = 0}^{L-1} W_{\bfy|\bfx}(y_j|\vx_{i}(j))^{1/L}}^s \notag \\
&= \sum_{\substack{ (x_{0,1},\cdots,x_{L-1,1})\in\cX^{L}\\\dots\\(x_{0,n},\cdots,x_{L-1,n})\in\cX^{L} }}
\prod_{j = 1}^n
\paren{\prod_{k = 0}^{L-1} P_\bfx(x_{k,j})}
\paren{ \sum_{y_j\in\cY} \prod_{i = 0}^{L-1} W_{\bfy|\bfx}(y_j|x_{i,j})^{1/L}}^s \notag \\
&= \prod_{j = 1}^n \sum_{(x_{0,j},\cdots,x_{L-1,j})\in\cX^{L}}
\paren{\prod_{k = 0}^{L-1} P_\bfx(x_{k,j})}
\paren{ \sum_{y_j\in\cY} \prod_{i = 0}^{L-1} W_{\bfy|\bfx}(y_j|x_{i,j})^{1/L}}^s \notag \\
&= \sqrbrkt{\sum_{(x_0,\cdots,x_{L-1})\in\cX^{L}}
\paren{\prod_{k = 0}^{L-1} P_\bfx(x_{k})}
\paren{ \sum_{y\in\cY} \prod_{i = 0}^{L-1} W_{\bfy|\bfx}(y|x_{i})^{1/L}}^s}^n, \notag 
\end{align}
which is independent of $ m_0,\cdots,m_{L-1} $. 
Then \Cref{eqn:expurg-to-be-cont} becomes at most
\begin{align}
M^{L-1} B^{-s} \sqrbrkt{\sum_{(x_0,\cdots,x_{L-1})\in\cX^{L}}
\paren{\prod_{k = 0}^{L-1} P_\bfx(x_{k})}
\paren{ \sum_{y\in\cY} \prod_{i = 0}^{L-1} W_{\bfy|\bfx}(y|x_{i})^{1/L}}^s}^n. \notag 
\end{align}
Choose $B$ such that the above quantity equals $ 1/2 $, i.e., 
\begin{align}
B &= \paren{\frac{1}{2}M^{-(L-1)}
\sqrbrkt{\sum_{(x_0,\cdots,x_{L-1})\in\cX^{L}}
\paren{\prod_{k = 0}^{L-1} P_\bfx(x_{k})}
\paren{ \sum_{y\in\cY} \prod_{i = 0}^{L-1} W_{\bfy|\bfx}(y|x_{i})^{1/L}}^s}^{-n}}^{-1/s} \notag \\
&= (2M^{L-1})^{1/s}
\sqrbrkt{\sum_{(x_0,\cdots,x_{L-1})\in\cX^{L}}
\paren{\prod_{k = 0}^{L-1} P_\bfx(x_{k})}
\paren{ \sum_{y\in\cY} \prod_{i = 0}^{L-1} W_{\bfy|\bfx}(y|x_{i})^{1/L}}^s}^{n/s}. \notag 
\end{align}
Under the above choice of $B$, 
we get that 
\begin{align}
\exptover{\cC}{\card{\curbrkt{m\in[M]:P_{\e,L-1}^{\mathrm{ML}}(m,\cC)>B}}} 
= \sum_{m\in[M]} \probover{\cC}{P^{\mathrm{ML}}_{\e,L-1}(m,\cC) > B} 
= M/2. \notag
\end{align}
Therefore, if we expurgate all codewords in $\cC$ with probability of error exceeding $ B $, we get a code $ \cC'\subset\cC $ of expected size $ M/2 $ whose codewords all have probability of error at most 
$P_{\e,L-1}^{\mathrm{ML}}(m,\cC') \le P_{\e,L-1}^{\mathrm{ML}}(m,\cC) \le B$. 
The first inequality follows since the probability of error of each codeword does not increase if there are less competing codewords. 
Letting $ \rho\coloneqq1/s\ge1 $ and $ R = \frac{1}{n}\log\frac{M}{2} $, we get the following upper bound on the error probability
\begin{align}
B &= (2^L(M/2)^{L-1})^{\rho}
\sqrbrkt{\sum_{(x_0,\cdots,x_{L-1})\in\cX^{L}}
\paren{\prod_{k = 0}^{L-1} P_\bfx(x_{k})}
\paren{ \sum_{y\in\cY} \prod_{i = 0}^{L-1} W_{\bfy|\bfx}(y|x_{i})^{1/L}}^{1/\rho}}^{\rho n} \label{eqn:expurg-to-incorp-constr} \\
&= 2^{L\rho} \exp_2\paren{-n\sqrbrkt{-(L-1)\rho R + E_{\x,L-1}(\rho,P_\bfx)}}, \notag
\end{align}
where 
\begin{align}
E_{\x,L-1}(\rho,P_\bfx) \coloneqq& - \rho \log\sqrbrkt{\sum_{(x_0,\cdots,x_{L-1})\in\cX^{L}}
\paren{\prod_{k = 0}^{L-1} P_\bfx(x_{k})}
\paren{ \sum_{y\in\cY} \prod_{i = 0}^{L-1} W_{\bfy|\bfx}(y|x_{i})^{1/L}}^{1/\rho}}. \notag
\end{align}
The above bound can be translated to the following lower bound on the error exponent
\begin{align}
E_{\ex,L-1}(R,P_\bfx) \coloneqq& \max_{\rho\ge1} \curbrkt{-(L-1)\rho R + E_{\x,L-1}(\rho,P_\bfx)}. \qedhere
\end{align}
\end{proof}

\subsection{Input constraints}
\label{sec:ee-input-constr}

We now derive bounds on the the achievable error exponents for discrete memoryless channels with input constraints.

Over the input alphabet $ \cX $, we associate a cost function $ f\colon\cX\to\bR $. 
We impose the following constraint that every input sequence/codeword should satisfy $ \sum_{i = 1}^nf(\vx(i))\le0 $. 
We can alternatively write this constraint in terms of $ \tau_\vx $ by observing that
\begin{align}
\sum_{i = 1}^n f(\vx(i)) &= \sum_{x\in\cX} n\tau_\vx(x) f(x). \notag 
\end{align}
Therefore, $ \sum_{i = 1}^nf(\vx(i))\le0 $ is equivalent to the input type constraint $ \tau_\vx\in\cP_f $ where
\begin{align}
\cP_f &\coloneqq \curbrkt{P_\bfx\in\Delta(\cX):\sum_{x\in\cX} P_\bfx(x)f(x) \le0}. \notag 
\end{align}
For example, the standard $\ell^2$ norm constraint on any $\cX\subset \bR$ can be obtained by choosing $f(x) = x^2-P$, which implies that for every codeword $\vx$, we must have $\sum_{i=1}^nx^2(i)\leq nP$.

The following is our main result, which gives an upper bound on the probability of error.
\begin{theorem}
\label{thm:ee-dmc-ip-constr}
Let $ L\in\bZ_{\ge2} $. 
Consider any DMC $ W_{\bfy|\bfx}\in\Delta(\cY|\cX) $ with input constraints $ \cP_f\subset\Delta(\cX) $ for some cost function $ f\colon\cX\to\bR $. 
Then there exists a code $ \cC\subset\cX^n $ of rate $R$, satisfying the input constraints $ \cP_f $ and
\begin{multline}
P_{\e,\avg,L-1}(\cC) \le \min_{P_\bfx\in\cP_f} \inf_{\delta>0} \min_{s\ge0,0\le\rho\le1} 
2^{nR(L-1)\rho} 
\paren{\frac{e^{s\delta}}{Z}}^{1+(L-1)\rho}
\sqrbrkt{\sum_{y\in\cY} \paren{\sum_{x\in\cX} P_{\bfx}(x) e^{sf(x)} W_{\bfy|\bfx}(y|x)^{\frac{1}{1+(L-1)\rho}} }^{1+(L-1)\rho}}^n, \notag 
\end{multline}
where
\begin{align}
Z \coloneqq& \sum_{\vx\in\cX^n} P_{\bfx}^\tn(\vx) \indicator{ \sum_{i = 1}^nf(\vx(i))\in[-\delta,0] }. \label{eqn:def-z-dmc}
\end{align}
For the same channel, there also exists a code $ \cC\subset\cX^n $ of rate $R$, satisfying the input constraints $ \cP_f $ and 
\begin{multline}
P_{\e,\avg,L-1}(\cC) \le \min_{P_\bfx\in\cP_f} \inf_{\delta>0} \min_{s\ge0,\rho\ge1} 2^{L\rho} 2^{nR(L-1)\rho}
\paren{\frac{e^{s\delta}}{Z}}^{L\rho} \\
\sqrbrkt{\sum_{(x_0,\cdots,x_{L-1})\in\cX^{L}}
\paren{e^{s\sum_{k = 0}^{L-1}f(x_k)} \prod_{k = 0}^{L-1} P_\bfx(x_{k})}
\paren{ \sum_{y\in\cY} \prod_{i = 0}^{L-1} W_{\bfy|\bfx}(y|x_{i})^{1/L}}^{1/\rho}}^{\rho n}, \notag 
\end{multline}
where $Z$ is defined in the same way as in \Cref{eqn:def-z-dmc}. 
\end{theorem}

\begin{proof}
For the DMC $ W_{\bfy|\bfx} $ with input constraints $ \cP_f $, 
we sample codewords from $ Q_{\vbfx}\in\Delta(\cX^n) $ which is obtained by truncating the input distribution $ P_{\bfx}^\tn $ so that it satisfies the power constraint.
Specifically, for some $ \delta>0 $, for any $ \vx\in\cX^n $,
\begin{align}
Q_{\vbfx}(\vx) \coloneqq& Z^{-1} P_{\bfx}^\tn(\vx) \indicator{ \sum_{i = 1}^nf(\vx(i))\in[-\delta,0] }, \notag 
\end{align}
where 
\begin{align}
Z \coloneqq& \sum_{\vx\in\cX^n} P_{\bfx}^\tn(\vx) \indicator{ \sum_{i = 1}^nf(\vx(i))\in[-\delta,0] } \notag 
\end{align}
is a normalizing constant. 
Though $ Q_{\vbfx} $ is not a product distribution, we will upper bound it pointwise by a product distribution. 
Note that the indicator function of the power constraint can be bounded as follows,
\begin{align}
\indicator{ -\delta \le \sum_{i = 1}^nf(\vx(i)) \le 0 } 
&= \indicator{ 0 \le \sum_{i = 1}^nf(\vx(i)) + \delta \le \delta } \notag \\
&= \indicator{ e^{s\cdot0} \le e^{s\paren{\sum_{i = 1}^nf(\vx(i)) + \delta}} \le e^{s\delta} } \notag \\
&\le \indicator{ e^{s\paren{\sum_{i = 1}^nf(\vx(i)) + \delta}} \ge 1} \notag \\
&\le e^{s\paren{\sum_{i = 1}^nf(\vx(i)) + \delta}}, \label[ineq]{eqn:elem-indicator-3}
\end{align}
for any $ s\ge0 $.
\Cref{eqn:elem-indicator-3} is by \Cref{eqn:elem-ineq-indicator}. 
Therefore we have
\begin{align}
Q_{\vbfx}(\vx) &\le Z^{-1} P_{\bfx}^\tn(\vx) e^{s\paren{\sum_{i = 1}^nf(\vx(i)) + \delta}} 
= \prod_{i = 1}^n \paren{P_\bfx(\vx(i)) e^{sf(\vx(i))} e^{s\delta/n} Z^{-1/n}}. \notag 
\end{align}
Replacing $ P_\bfx(x) $ with $ {P_\bfx(x) e^{sf(x)} e^{s\delta/n}} Z^{-1/n} $ in \Cref{eqn:rce-to-incorp-constr}, we have a random coding bound with input constraints:
\begin{align}
\exptover{\cC}{P^{\mathrm{ML}}_{\e,L-1}(m,\cC)} &\le 
M^{(L-1)\rho} 
\paren{\frac{e^{s\delta}}{Z}}^{1+(L-1)\rho}
\sqrbrkt{\sum_{y\in\cY} \paren{\sum_{x\in\cX} P_{\bfx}(x) e^{sf(x)} W_{\bfy|\bfx}(y|x)^{\frac{1}{1+(L-1)\rho}} }^{1+(L-1)\rho}}^n, \label{eqn:rce-constr}
\end{align}
for $ s\ge0 $ and $ 0\le\rho\le1 $. 
A similar substitution for \Cref{eqn:expurg-to-incorp-constr} yields an expurgated bound with input constraints:
\begin{align}
2^{L\rho} M^{(L-1)\rho}
\paren{\frac{e^{s\delta}}{Z}}^{L\rho}
\sqrbrkt{\sum_{(x_0,\cdots,x_{L-1})\in\cX^{L}}
\paren{e^{s\sum_{k = 0}^{L-1}f(x_k)} \prod_{k = 0}^{L-1} P_\bfx(x_{k})}
\paren{ \sum_{y\in\cY} \prod_{i = 0}^{L-1} W_{\bfy|\bfx}(y|x_{i})^{1/L}}^{1/\rho}}^{\rho n}, \label{eqn:expurg-constr}
\end{align}
for $ s\ge0 $ and $ \rho\ge1 $. 
\end{proof}

\subsection{Continuous alphabets}
\label{sec:ee-cts-alpha}

It is easy to extend the same ideas to continuous alphabets such as $\cX=\bR$. The following theorem states our main result.
\begin{theorem}
\label{thm:ee-cts-ip-constr}
Let $ L\in\bZ_{\ge2} $. 
Consider any memoryless channel $ W_{\bfy|\bfx}\in\Delta(\bR|\bR) $ over the reals with input constraints 
\begin{align}
\cP_f &\coloneqq \curbrkt{P_\bfx\in\Delta(\bR):\int_{\bR} P_\bfx(x)f(x)\diff x\le0} \subset\Delta(\bR) \notag 
\end{align}
for some cost function $ f\colon\bR\to\bR $. 
Then there exists a code $ \cC\subset\bR^n $ of rate $R$, satisfying the input constraints $ \cP_f $ and 
\begin{multline}
P_{\e,\avg,L-1}(\cC) \le \min_{P_\bfx\in\cP_f} \inf_{\delta>0} \min_{s\ge0,0\le\rho\le1} 
e^{nR(L-1)\rho} 
\paren{\frac{e^{s\delta}}{Z}}^{1+(L-1)\rho}
\sqrbrkt{\int_\bR \paren{\int_\bR P_\bfx(x)e^{sf(x)}W_{\bfy|\bfx}(y|x)^{\frac{1}{1+(L-1)\rho}} \diff x}^{1+(L-1)\rho} \diff y}^n, \notag 
\end{multline}
where
\begin{align}
Z\coloneqq& \int_{\bR^n} P_\bfx^{\tn}(\vx) \indicator{ \sum_{i = 1}^nf(\vx(i))\in[-\delta,0] } \diff\vx. \label{eqn:def-z-cts}
\end{align}
For the same channel, there also exists a code $ \cC\subset\bR^n $ of rate $R$, satisfying the input constraints $ \cP_f $ and 
\begin{multline}
P_{\e,\avg,L-1}(\cC) \le \min_{P_\bfx\in\cP_f} \inf_{\delta>0} \min_{s\ge0,\rho\ge1} 2^{L\rho} e^{nR(L-1)\rho}
\paren{\frac{e^{s\delta}}{Z}}^{L\rho} \\
\sqrbrkt{\sum_{(x_0,\cdots,x_{L-1})\in\cX^{L}}
\paren{e^{s\sum_{k = 0}^{L-1}f(x_k)} \prod_{k = 0}^{L-1} P_\bfx(x_{k})}
\paren{ \sum_{y\in\cY} \prod_{i = 0}^{L-1} W_{\bfy|\bfx}(y|x_{i})^{1/L}}^{1/\rho}}^{\rho n}, \notag 
\end{multline}
where $Z$ is defined in the same way as in \Cref{eqn:def-z-cts}. 
\end{theorem}

\begin{proof}
\Cref{eqn:rce-constr,eqn:expurg-constr} can be generalized to channels over the reals in a straightforward manner:
\begin{align}
& M^{(L-1)\rho} 
\paren{\frac{e^{s\delta}}{Z}}^{1+(L-1)\rho}
\sqrbrkt{\int_\bR \paren{\int_\bR P_\bfx(x)e^{sf(x)}W_{\bfy|\bfx}(y|x)^{\frac{1}{1+(L-1)\rho}} \diff x}^{1+(L-1)\rho} \diff y}^n, \label{eqn:gallager-cts-constrained-rce} 
\end{align}
where $ s\ge0 $ and $ 0\le\rho\le1 $; 
\begin{align}
& 2^{L\rho} M^{(L-1)\rho}
\paren{\frac{e^{s\delta}}{Z}}^{L\rho}
\sqrbrkt{\int_{\bR^L} \paren{e^{s\sum_{k = 0}^{L-1}f(x_k)} \prod_{k = 0}^{L-1} P_\bfx(x_{k})}
\paren{\int_\bR \prod_{i = 0}^{L-1} W_{\bfy|\bfx}(y|x_{i})^{1/L} \diff y}^{1/\rho}
\diff(x_0,\cdots,x_{L-1})}^{\rho n}, \label{eqn:gallager-cts-constrained-exe}
\end{align}
where $ s\ge0,\rho\ge1 $ and 
\begin{align}
Z\coloneqq& \int_{\bR^n} P_\bfx^{\tn}(\vx) \indicator{ \sum_{i = 1}^nf(\vx(i))\in[-\delta,0] } \diff\vx. \label{eqn:def-z}
\end{align}
\end{proof}

\subsection{Random coding exponent for AWGN channels with input constraints}
\label{sec:awgn-rce}
\Cref{thm:ee-cts-ip-constr} gives non-explicit upper bounds on the probability of error.
In this section, we derive explicit lower bounds on \Cref{eqn:def-z-cts} in the case of AWGN channels with input constraint $P$ and noise variance $ \sigma^2 $ under $ (L-1) $-list-decoding. 
We prove the following theorem. 
\begin{theorem}
\label{thm:awgn-rce}
Let $ P,\sigma>0 $ and $ L\in\bZ_{\ge2} $. 
There exist codes of rate $R$ for the AWGN channel with input constraint $ P $ and noise variance $ \sigma^2 $ such that the rate satisfies $ 0\le R\le\frac{1}{2}\ln(1+P/\sigma^2) $ and the exponent $ E_{L-1}(R,P/\sigma^2) $ of the probability of error (normalized by $ \lim\limits_{n\to\infty}-\frac{1}{n}\ln(\cdot) $) under $(L-1)$-list-decoding is bounded as follows. 

Let $ \snr\coloneqq P/\sigma^2 $ and 
\begin{align}
R_{\x,L-1}(\snr) &\coloneqq \frac{1}{2}\paren{\ln\frac{\sqrt{L^2+\snr^2 - 2\snr(L-2)} + L + \snr}{2L} + \frac{1}{L-1}\ln\frac{\sqrt{L^2+\snr^2 - 2\snr(L-2)} + L - \snr}{2L}}, \label{eqn:def-rx} \\
R_{\crit,L-1}(\snr) &\coloneqq \frac{1}{2}\ln\paren{\frac{1}{2} + \frac{\snr}{2L} + \frac{1}{2}\sqrt{1 - \frac{2(L-2)}{L^2}\snr + \frac{\snr^2}{L^2}}}. \label{eqn:def-rcrit} 
\end{align}
\begin{enumerate}
 	
 	\item\label{itm:awgn-rce1} If 
 	$ R_{\crit,L-1}(\snr) \le R \le\frac{1}{2}\ln(1+\snr) $, then
 	\begin{align}
 	E_{L-1}(R,\snr) &\ge 
 	\frac{1}{2}\ln\sqrbrkt{e^{2R} - \frac{\snr(e^{2R} - 1)}{2}\paren{\sqrt{1+\frac{4e^{2R}}{\snr(e^{2R} - 1)}} - 1}} \notag \\
	&\quad + \frac{\snr}{4e^{2R}}\paren{e^{2R} + 1 - (e^{2R} - 1)\sqrt{1+\frac{4e^{2R}}{\snr(e^{2R} - 1)}}}. \label{eqn:ee-high-ld} 
 	\end{align}
 	
 	\item\label{itm:awgn-rce2} If 
 	$0 \le R \le R_{\crit,L-1}(\snr)$, 
 	then 
 	\begin{align}
	E_{L-1}(R,\snr) &\ge
	-R(L-1) + \frac{L-1}{2}\ln\paren{L + \snr + \sqrt{(L-\snr)^2 + 4\snr}} + \frac{1}{2}\ln\paren{L-\snr + \sqrt{(L-\snr)^2 + 4\snr}} \notag \\
	&\quad + \frac{1}{4}\paren{L + \snr - \sqrt{(L-\snr)^2 + 4\snr}} - \frac{L}{2}\ln(2L). \label{eqn:ee-mid-ld} 
	\end{align}

\end{enumerate} 
\end{theorem}

For an AWGN channel with input constraint $P$ and noise variance $ \sigma^2 $, 
the channel transition kernel is given by 
\begin{align}
W_{\bfy|\bfx}(y|x) &= \frac{1}{\sqrt{2\pi\sigma^2}} e^{-\frac{(y - x)^2}{2\sigma^2}}, \label{eqn:awgn-kernel}
\end{align}
and the cost function is given by 
\begin{align}
f(x) = x^2 - P. \label{eqn:cost-fn-awgn}
\end{align}
Let $ P_\bfx $ be the Gaussian density with variance $P$:
\begin{align}
P_\bfx(x) &= \frac{1}{\sqrt{2\pi P}}e^{-\frac{x^2}{2P}}. \label{eqn:gauss-density-p}
\end{align}

For a constant $ \delta>0 $, we claim that the factor $ (e^{s\delta}/Z)^{1+(L-1)\rho} $ that appears in \Cref{eqn:gallager-cts-constrained-rce} scales like $ \poly(n) $ for asymptotically large $n$ and therefore does not effectively contribute to the exponent. 
Indeed, the following lemma holds. 
\begin{lemma}
\label{lem:bound-z}
Let $ P,\sigma,\delta>0 $ be constants. 
Let $ P_\bfx $ be the Gaussian density with variance $P$ as defined in \Cref{eqn:gauss-density-p}. 
Let $ f(x) \coloneqq x^2 - P $. 
Let $ Z $ be defined by \Cref{eqn:def-z}. 
Then
$Z \stackrel{n\to\infty}{\asymp} \frac{\delta}{2P\sqrt{\pi n}}$. 
\end{lemma}

\begin{proof}
The proof follows from the central limit theorem. 
\begin{align}
Z &= \int_{\bR^n} P_\bfx^\tn(\vx) \indicator{-\delta\le\sum_{i = 1}^n(\vx(i)^2 - P)\le0}\diff \vx \notag \\
&= \prob{-\delta\le P\paren{\sum_{i = 1}^n\cN_i(0,1)^2 - n}\le0} \notag \\
&= \prob{\frac{-\delta}{P\sqrt{2n}}\le \frac{\chi^2(n) - n}{\sqrt{2n}}\le0} \notag \\
&\stackrel{n\to\infty}{\asymp} \prob{\frac{-\delta}{P\sqrt{2n}}\le \cN(0,1) \le0} \label{eqn:clt-chi-square} \\
&\stackrel{n\to\infty}{\asymp} \frac{\delta}{P\sqrt{2n}}\cdot\frac{1}{\sqrt{2\pi}} \label{eqn:rectangle} \\
&= \frac{\delta}{2P\sqrt{\pi n}}. \notag 
\end{align}
\Cref{eqn:clt-chi-square} follows since $ \frac{\chi^2(n) - n}{\sqrt{2n}} $ converges to $ \cN(0,1) $ in distribution as $ n\to\infty $. 
\Cref{eqn:rectangle} follows since the Gaussian measure of a thin interval $ \sqrbrkt{-\frac{\delta}{P\sqrt{2n}},0} $ is essentially the area of a rectangle with width $ \frac{\delta}{P\sqrt{2n}} $ and height $ P_\bfx(0) = 1/\sqrt{2\pi} $ for asymptotically large $n$. 
\end{proof}

We are now ready to evaluate the random coding bound (\Cref{eqn:gallager-cts-constrained-rce}) on the probability of the $(L-1)$-list-decoding error of AWGN channels with input constraint $P$ and noise variance $ \sigma^2 $. 
\begin{proof}[Proof of \Cref{thm:awgn-rce}]
The exponent (i.e., the probability of error normalized by $ -\frac{1}{n}\ln(\cdot) $) given by \Cref{eqn:gallager-cts-constrained-rce} specializes to 
\begin{align}
-R(L-1)\rho-\ln\sqrbrkt{\int_\bR \paren{\int_\bR \frac{1}{\sqrt{2\pi P}}e^{-\frac{x^2}{2P}} e^{s(x^2-P)}\paren{\frac{1}{\sqrt{2\pi\sigma^2}}e^{-\frac{(y - x)^2}{2\sigma^2}}}^{\frac{1}{1+(L-1)\rho}} \diff x}^{1+(L-1)\rho} \diff y}. \notag 
\end{align}

For notational convenience, let $ \gamma\coloneqq1+(L-1)\rho $. 
We first compute the inner integral
\begin{align}
I(y) &\coloneqq \int_\bR \frac{1}{\sqrt{2\pi P}}\frac{1}{\sqrt{2\pi\sigma^2}^{1/\gamma}} \exp\paren{-\frac{x^2}{2P} + s(x^2 - P) - \frac{(y - x)^2}{2\sigma^2\gamma}} \diff x \notag \\
&= \int_\bR \frac{1}{\sqrt{2\pi P}}\frac{1}{\sqrt{2\pi\sigma^2}^{1/\gamma}} \exp\paren{\paren{-\frac{1}{2P} + s - \frac{1}{2\sigma^2\gamma}}x^2 + \frac{y}{\sigma^2\gamma}x - sP - \frac{1}{2\sigma^2\gamma}y^2} \diff x, \notag
\end{align}
which is a Gaussian integral.
We let 
\begin{align}
a &\coloneqq \frac{1}{2P} - s + \frac{1}{2\sigma^2\gamma}, \quad
b \coloneqq \frac{y}{\sigma^2\gamma} \quad
c \coloneqq -sP - \frac{1}{2\sigma^2\gamma}y^2, \notag
\end{align}
and 
\begin{align}
A &\coloneqq \sqrt{\frac{\pi}{a}} 
= \sqrt{\pi\paren{\frac{1}{2P} - s + \frac{1}{2\sigma^2\gamma}}^{-1}}
= \sqrt{\frac{2\pi P\sigma^2\gamma}{\sigma^2\gamma(1 - 2sP) + P}}. \notag 
\end{align}
By \Cref{lem:gauss-int-1d}, the above integral $ I(y) $ equals
\begin{align}
A\cdot\frac{1}{\sqrt{2\pi P}}\frac{1}{\sqrt{2\pi\sigma^2}^{1/\gamma}} \exp\paren{\frac{b^2}{4a} + c}
&= A\cdot\frac{1}{\sqrt{2\pi P}}\frac{1}{\sqrt{2\pi\sigma^2}^{1/\gamma}} \exp\paren{\frac{y^2}{\sigma^4\gamma^2\cdot4\cdot\paren{\frac{1}{2P} - s + \frac{1}{2\sigma^2\gamma}}} - sP - \frac{y^2}{2\sigma^2\gamma}} \notag \\
&= A\cdot\frac{1}{\sqrt{2\pi P}}\frac{1}{\sqrt{2\pi\sigma^2}^{1/\gamma}} \exp\paren{-\frac{1-2sP}{2(\sigma^2\gamma(1-2sP) + P)}y^2 - sP}. \notag 
\end{align}

We then compute the outer integral: 
\begin{align}
\int_\bR I(y)^\gamma \diff y
&= \int_\bR A^\gamma \frac{1}{\sqrt{2\pi P}^\gamma} \frac{1}{\sqrt{2\pi\sigma^2}} \exp\paren{-\frac{(1-2sP)\gamma}{2(\sigma^2\gamma(1-2sP) + P)}y^2 - sP\gamma} \diff y. \notag
\end{align}
We get again a Gaussian integral. 
Let
\begin{align}
a' &\coloneqq \frac{(1-2sP)\gamma}{2(\sigma^2\gamma(1-2sP) + P)}, \quad
b' \coloneqq 0, \quad
c' \coloneqq -sP\gamma. \notag 
\end{align}
By \Cref{lem:gauss-int-1d}, 
\begin{align}
\int_\bR I(y)^\gamma \diff y 
&= A^\gamma \frac{1}{\sqrt{2\pi P}^\gamma} \frac{1}{\sqrt{2\pi\sigma^2}} \cdot \sqrt{\frac{\pi}{a'}} e^{c'} \notag \\
&= \sqrt{\frac{2\pi P\sigma^2\gamma}{\sigma^2\gamma(1 - 2sP) + P}}^\gamma \frac{1}{\sqrt{2\pi P}^\gamma} \frac{1}{\sqrt{2\pi\sigma^2}} \sqrt{\pi\frac{2(\sigma^2\gamma(1-2sP) + P)}{(1-2sP)\gamma}} e^{-sP\gamma}. \notag 
\end{align}

With this, the random coding exponent becomes
\begin{align}
E(s,\gamma) &\coloneqq -R(L-1)\rho - \ln\sqrbrkt{\sqrt{\frac{2\pi P\sigma^2\gamma}{\sigma^2\gamma(1 - 2sP) + P}}^\gamma \frac{1}{\sqrt{2\pi P}^\gamma} \frac{1}{\sqrt{2\pi\sigma^2}} \sqrt{\pi\frac{2(\sigma^2\gamma(1-2sP) + P)}{(1-2sP)\gamma}} e^{-sP\gamma}} \notag \\
&= -R(\gamma - 1) + \frac{\gamma - 1}{2}\ln\paren{1 - 2sP + \frac{P}{\sigma^2\gamma}} + \frac{1}{2}\ln(1-2sP) + sP\gamma. \label{eqn:rce-s-gamma} 
\end{align}
For the above bound to be valid, we need $ s<\frac{1}{2P} $. 

Recall that $ s\ge0,\rho\in[0,1] $ and $ \gamma = 1+(L-1)\rho $. 
We need to maximize $ E(s,\gamma) $ in the region $ s\in[0,1/(2P)],\gamma\in[1,L] $. 
To this end, we compute the stationary $ s $ and $\gamma$. 
\begin{align}
\frac{\partial}{\partial s}E(s,\gamma) &= P\paren{- \frac{\gamma - 1}{1-2sP + \frac{P}{\sigma^2\gamma}} + \frac{1}{1-2sP} + \gamma} = 0, \label{eqn:partial-s} \\
\frac{\partial}{\partial \gamma}E(s,\gamma) &= -R + \frac{1}{2}\ln\paren{1-2sP + \frac{P}{\sigma^2\gamma}} - \frac{P(\gamma - 1)}{2\gamma(P+\sigma^2\gamma(1-2sP))} + sP. \label{eqn:partial-gamma}
\end{align}

Let $ \snr\coloneqq P/\sigma^2 $ denote the signal-to-noise ratio (SNR). 
Solving $s$ from \Cref{eqn:partial-s}, we get
\begin{align}
s &= \frac{1}{4P}\paren{1 + \frac{\snr}{\gamma} - \frac{1}{\gamma}\sqrt{(\gamma - \snr)^2+4\snr}}. \label{eqn:stationary-s}
\end{align}
One can easily check that $ s\ge0 $ provided $ \gamma\ge1 $. 
Furthermore, $ s<\frac{1}{2P} $. 

Putting \Cref{eqn:stationary-s} into \Cref{eqn:partial-gamma} and solving $\gamma$ therein, we get
\begin{align}
\gamma &= \frac{\snr}{2e^{2R}}\paren{1 + \sqrt{1 + \frac{4\snr}{\snr(e^{2R} - 1)}}}. \label{eqn:stationary-gamma}
\end{align}
It can be easily verified that $ \gamma\ge1 $ for any $ R\le\frac{1}{2}\ln(1+\snr) $. 

Suppose $ \gamma\le L $. Then the minimum value of $ E(s,\gamma) $ is indeed achieved at the above $\gamma$ given by \Cref{eqn:stationary-gamma}. 
Note that the condition $ \gamma\le L $ is equivalent to 
\begin{align}
R &\ge \frac{1}{2}\ln\paren{\frac{1}{2} + \frac{\snr}{2L} + \frac{1}{2}\sqrt{1 - \frac{2(L-2)}{L^2}\snr + \frac{\snr^2}{L^2}}}. \label[ineq]{eqn:awgn-r-crit}
\end{align}
Substituting the stationary $ \gamma $ (\Cref{eqn:stationary-gamma}) back to \Cref{eqn:stationary-s}, we get the stationary $s$ as a function of only $ \snr $ and $R$. 
Note that here $s$ and $\gamma$ do not depend on $L$. 
Therefore the calculations in this case coincide with those for unique-decoding case as done in \cite[Theorem 7.4.4]{gallager} and we omit the details. 
Putting both $ s $ and $\gamma$ into \Cref{eqn:rce-s-gamma}, we finally get the random coding exponent
\begin{align}
\min_{s\in[0,1/(2P)], \gamma\in[1,L]} E(s,\gamma)
&= \frac{1}{2}\ln\sqrbrkt{e^{2R} - \frac{\snr(e^{2R} - 1)}{2}\paren{\sqrt{1+\frac{4e^{2R}}{\snr(e^{2R} - 1)}} - 1}} \notag \\
&\quad + \frac{\snr}{4e^{2R}}\paren{e^{2R} + 1 - (e^{2R} - 1)\sqrt{1+\frac{4e^{2R}}{\snr(e^{2R} - 1)}}}. \notag 
\end{align}
This proves \Cref{itm:awgn-rce1} in \Cref{thm:awgn-rce}. 

On the other hand, if the $ \gamma $ given by \Cref{eqn:stationary-gamma} is larger than $L$, i.e., \Cref{eqn:awgn-r-crit} holds in the reverse direction, then the minimum value of $ E(s,\gamma) $ is achieved at $ \gamma = L $. 
In this case, $s$ given by \Cref{eqn:stationary-s} becomes
\begin{align}
s &= \frac{1}{4P}\paren{1 + \frac{\snr}{L} - \frac{1}{L}\sqrt{(L-\snr)^2 + 4\snr}}, \label{eqn:crit-s-2}
\end{align}
and the minimum value of $ E(s,\gamma) $ is achieved at $ \gamma = L $ and the $ s $ given by \Cref{eqn:crit-s-2}:
\begin{align}
\min_{s\in[0,1/(2P)],\gamma\in[1,L]} E(s,\gamma)
&= -R(L-1) + \frac{L-1}{2}\ln\paren{L + \snr + \sqrt{(L-\snr)^2 + 4\snr}} \notag \\
&\quad + \frac{1}{2}\ln\paren{L-\snr + \sqrt{(L-\snr)^2 + 4\snr}} \notag \\
&\quad + \frac{1}{4}\paren{L + \snr - \sqrt{(L-\snr)^2 + 4\snr}} - \frac{L}{2}\ln(2L). \notag
\end{align}
This proves \Cref{itm:awgn-rce2} in \Cref{thm:awgn-rce}. 
\end{proof}

\subsection{Expurgated exponent for AWGN channels with input constraints}
\label{sec:awgn-exe}
We proceed to evaluate the expurgated exponent (\Cref{eqn:gallager-cts-constrained-exe}) in the case of AWGN channels with input constraint $P$ and noise variance $ \sigma^2 $ under $ (L-1) $-list-decoding. 
We prove the following theorem. 
\begin{theorem}
\label{thm:awgn-exe}
Let $ P,\sigma>0 $ and $ L\in\bZ_{\ge2} $. 
Consider an AWGN channel with input constraint $ P $ and noise variance $ \sigma^2 $. 
Let $ \snr\coloneqq P/\sigma^2 $. 
Let $ R_{\x,L-1}(\snr) $ be defined by \Cref{eqn:def-rx}. 
Then there exist codes of rate $ 0 \le R \le R_{\x,L-1}(\snr) $ for the above channel such that the exponent $ E_{L-1}(R,\snr) $ of the probability of error (normalized by $ \lim\limits_{n\to\infty}-\frac{1}{n}\ln(\cdot) $) under $(L-1)$-list-decoding is bounded as follows: 
\begin{align}
E_{L-1}(R,\snr) &\ge \frac{\snr(Lt - 1)}{2Lt}, \label{eqn:awgn-exe} 
\end{align}
where $ t $ is the unique solution of $ (Lt - 1)e^{2R} = (L-1)t^{\frac{L}{L-1}} $ in $ t\in[1/L,1] $. 
\end{theorem}

\begin{proof}
For the channel of interest, the channel transition kernel $ W_{\bfy|\bfx} $, the cost function $ f $ and the input distribution $ P_\bfx $ are given by \Cref{eqn:cost-fn-awgn,eqn:gauss-density-p,eqn:awgn-kernel}, respectively. 
For a constant $ \delta>0 $, by \Cref{lem:bound-z}, the factor $ 2^{L\rho}\paren{\frac{e^{s\delta}}{Z}}^{L\rho} $ is subexponential in $n$ and does not play a role in the exponent. 
Therefore, the exponent of \Cref{eqn:gallager-cts-constrained-exe} specializes to 
\begin{align}
& R (L-1)\rho + \rho\ln\sqrbrkt{
	\int_{\bR^L} \paren{e^{s\sum_{k = 0}^{L-1}x_k^2 - sLP}\frac{1}{\sqrt{2\pi P}^L}e^{-\frac{\sum_{k = 0}^{L-1}x_k^2}{2P}}}
	\paren{\int_\bR \frac{1}{\sqrt{2\pi\sigma^2}}e^{-\frac{\sum_{k = 0}^{L-1}(y - x_k)^2}{2\sigma^2L}} \diff y}^{1/\rho}
	\diff(x_0,\cdots,x_{L-1})
} . \label{eqn:gallager-awgn} 
\end{align}
The inner integral w.r.t.\ $y$ in \Cref{eqn:gallager-awgn} is a Gaussian integral and can be computed as follows using \Cref{lem:gauss-int-1d}. 
\begin{align}
\int_\bR \frac{1}{\sqrt{2\pi\sigma^2}} \exp\paren{-\frac{1}{2\sigma^2L}\sum_{i = 0}^{L-1}(y - x_i)^2}\diff y 
&= \frac{1}{\sqrt{2\pi\sigma^2}} \int_\bR 
\exp\paren{-\frac{1}{2\sigma^2}y^2 + \frac{1}{\sigma^2L}y\sum_{i = 0}^{L-1}x_i - \frac{1}{2\sigma^2L}\sum_{i = 0}^{L-1}x_i^2} \diff y \notag \\
&= \exp\paren{\frac{1}{2\sigma^2L^2}\paren{\paren{\sum_{i = 0}^{L-1}x_i}^2 - L\sum_{i = 0}^{L-1}x_i^2}} \notag \\
&= \exp\paren{\frac{1}{2\sigma^2L^2}\paren{\sum_{0\le i\ne j\le L-1}x_ix_j - (L-1)\sum_{i = 0}^{L-1}x_i^2}}. \notag 
\end{align}
Now the $L$-dimensional integral inside the logarithm in \Cref{eqn:gallager-awgn} equals
\begin{align}
& \int_{\bR^L} \frac{1}{\sqrt{2\pi P}^L} \exp\paren{s\sum_{i = 0}^{L-1}x_i^2 - sLP - \frac{1}{2P}\sum_{i = 0}^{L-1}x_i^2 + \frac{1}{2\sigma^2L^2\rho}\paren{\sum_{0\le i\ne j\le L-1}x_ix_j - (L-1)\sum_{i = 0}^{L-1}x_i^2}} \diff(x_0,\cdots,x_{L-1}) \notag \\
&= \frac{e^{-sLP}}{\sqrt{2\pi P}^L} 
\int_{\bR^L}
\exp\paren{\paren{s - \frac{1}{2P} - \frac{L-1}{2\sigma^2L^2\rho}}\sum_{i = 0}^{L-1}x_i^2 + \frac{1}{2\sigma^2L^2\rho}\sum_{0\le i\ne j\le L-1}x_ix_j}
\diff(x_0,\cdots,x_{L-1}) \notag \\
&= \frac{e^{-sLP}}{\sqrt{2\pi P}^L} 
\int_{\bR^L}\exp\paren{-\vec x^\top A\vec x} \diff \vec x, \label{eqn:gauss-int-comp} 
\end{align}
where $ \vec x = [x_0,\cdots,x_{L-1}]\in\bR^{L} $ and $ A\in\bR^{L\times L} $ is a matrix with all diagonal entries equal to 
\begin{align}
a \coloneqq& \frac{1}{2P} + \frac{L-1}{2\sigma^2L^2\rho} - s \notag 
\end{align}
and all off-diagonal entries equal to 
\begin{align}
b \coloneqq& -\frac{1}{2\sigma^2L^2\rho}. \notag 
\end{align}

By \Cref{lem:gauss-int}, the RHS of \Cref{eqn:gauss-int-comp} equals
\begin{align}
\frac{e^{-sLP}}{\sqrt{2\pi P}^L} \sqrt{\frac{\pi^L}{\det(A)}}
= \frac{e^{-sLP}}{\sqrt{2P}^L\sqrt{\det(A)}} . \label{eqn:to-log}
\end{align}
To compute $ \det(A) $, we note that 
$ A = (a-b)I_L + (\sqrt{-b}\vec\one_L)(-\sqrt{-b}\vec\one_L)^\top $ where $ \vec\one_L $ denotes the all-one vector of length $L$. 

\begin{lemma}[Matrix determinant lemma]
\label{lem:sherman-morrison}
Let $ A\in\bR^{n\times n} $ be a non-singular matrix and let $ \vu,\vv\in\bR^n $. 
Then
\begin{align}
\det\paren{A + \vu\vv^\top} &= \paren{1 + \vv^\top A^{-1}\vu}\det(A). \notag 
\end{align}
\end{lemma}

By \Cref{lem:sherman-morrison}, we have
\begin{align}
\det(A) &= \sqrbrkt{1 + (-\sqrt{-b}\vec\one_L)^\top\paren{(a-b)I_L}^{-1}(\sqrt{-b}\vec\one_L)} \det\paren{(a-b)I_L} \notag \\
&= \paren{1+\frac{b}{a-b}L}(a-b)^L \notag \\
&= (a + (L-1)b)(a-b)^{L-1} \notag \\
&= \paren{\frac{1}{2P} - s}\paren{\frac{1}{2P} + \frac{1}{2\sigma^2L\rho} - s}^{L-1}. \notag 
\end{align}
Therefore, the (natural) logarithm of the RHS of \Cref{eqn:to-log} equals
\begin{align}
& -sLP - \frac{L}{2}\ln(2P) - \frac{1}{2}\ln\paren{\frac{1}{2P} - s} - \frac{L-1}{2}\ln\paren{\frac{1}{2P} + \frac{1}{2\sigma^2L\rho} - s} \notag \\
&= -\paren{ sLP + \frac{1}{2}\ln(1-2sP) + \frac{L-1}{2}\ln\paren{1-2sP + \frac{P}{\sigma^2L\rho}} }. \notag 
\end{align}
Plugging the above expression back to \Cref{eqn:gallager-awgn}, we see that
to get the largest error exponent, we need to minimize the following expression over $ s\ge0 $ and $ \rho\ge1 $. 
\begin{align}
R (L-1)\rho - \rho\sqrbrkt{sLP + \frac{1}{2}\ln(1-2sP) + \frac{L-1}{2}\ln\paren{1-2sP + \frac{P}{\sigma^2L\rho}}}. \label{eqn:exprg-ee-to-opt} 
\end{align}
From the calculations in \Cref{sec:lb-ee-together}, one can obtain an expression of the solution to the above minimization problem. 
Specifically, negating \Cref{eqn:exprg-ee-to-opt}, by \Cref{eqn:crit-f-s-rho}, we know that the maximum value equals 
\begin{align}
\frac{P(L(1-2Ps) - 1)}{2L\sigma^2(1-2Ps)}
&= \frac{\snr(Lt - 1)}{2Lt}, \label{eqn:awgn-ex-ee} 
\end{align}
where $ t\coloneqq1-2Ps $. 
Recall that $ 0\le s\le\frac{1-1/L}{2P} $ satisfies \Cref{eqn:r-expr} which can be rewritten in term of $t$ as
\begin{align}
R = \frac{1}{2}\sqrbrkt{\ln\frac{(L-1)t}{Lt-1} + \frac{1}{L-1}\ln t}. \label{eqn:r-expr-t}
\end{align}
Equivalently, $t$ is the unique solution of the equation $(Lt - 1)e^{2R} = (L-1)t^{\frac{L}{L-1}}$ in $ t\in[1/L,1] $,

\Cref{eqn:awgn-ex-ee} is valid whenever $ \rho\ge1 $. 
Recall the relation between $\rho$ and $s$ (\Cref{eqn:rho-expr}). 
We rewrite it in terms of $t$:
\begin{align}
\rho &= \frac{(L-1)+L(t-1)}{2L^2\cdot\frac{1-t}{2P}\cdot t\cdot\sigma^2} = \frac{(Lt - 1)\snr}{L^2(1-t)t}. \notag 
\end{align}
By the above relation between $\rho$ and $t$, the condition $ \rho\ge1 $ is equivalent to 
\begin{align}
t &\ge \frac{L-\snr + \sqrt{L^2+\snr^2 - 2\snr(L-2)}}{2L}. \label[ineq]{eqn:rho-ge-one-t} 
\end{align}
Plugging the RHS of \Cref{eqn:rho-ge-one-t} to \Cref{eqn:r-expr-t}, the condition $ \rho\ge1 $ is further equivalent to 
\begin{align}
R &\le \frac{1}{2}\paren{\ln\frac{\sqrt{L^2+\snr^2 - 2\snr(L-2)} + L + \snr}{2L} + \frac{1}{L-1}\ln\frac{\sqrt{L^2+\snr^2 - 2\snr(L-2)} + L - \snr}{2L}}, \notag
\end{align}
the RHS of which is defined as $ R_{\x,L-1}(\snr) $. 
We conclude that the error exponent given by the RHS of \Cref{eqn:awgn-ex-ee} can be achieved for any $ R\le R_{\x,L-1}(\snr) $. 
\end{proof}

\subsection{List-decoding error exponents vs.\ unique-decoding error exponents}
\label{sec:ld-ee-vs-ud-ee}
Our bounds on the list-decoding error exponent of AWGN channels recover Gallager's results \cite{gallager-1965-simple-deriv}, \cite[Theorem 7.4.4]{gallager} for unique-decoding. 
Indeed, when $ L = 2 $, \Cref{eqn:def-rx,eqn:def-rcrit} become
\begin{align}
R_{\x,1}(\snr) &= \frac{1}{2}\paren{\ln\frac{\sqrt{4+\snr^2} + 2 + \snr}{4} + \ln\frac{\sqrt{4+\snr^2} + 2 - \snr}{4}} = \frac{1}{2}\ln\paren{\frac{1}{2} + \frac{1}{2}\sqrt{1+\frac{\snr^2}{4}}} , \label{eqn:rx-ud} \\
R_{\crit,1}(\snr) &= \frac{1}{2}\ln\paren{\frac{1}{2} + \frac{\snr}{4} + \frac{1}{2}\sqrt{1 + \frac{\snr^2}{4}}} , \label{eqn:rcrit-ud} 
\end{align}
and the random coding exponent in \Cref{thm:awgn-rce} specializes to 
\begin{align}
E_1(R,\snr) &\ge \frac{1}{2}\ln\sqrbrkt{e^{2R} - \frac{\snr(e^{2R} - 1)}{2}\paren{\sqrt{1+\frac{4e^{2R}}{\snr(e^{2R} - 1)}} - 1}} \notag \\
&\quad + \frac{\snr}{4e^{2R}}\paren{e^{2R} + 1 - (e^{2R} - 1)\sqrt{1+\frac{4e^{2R}}{\snr(e^{2R} - 1)}}} \label{eqn:ee-high-ud} 
\end{align}
for $ R_{\crit,1}(\snr)\le R\le\frac{1}{2}\ln(1+\snr) $, and
\begin{align}
E_1(R,\snr) &= -R + \frac{1}{2}\ln\paren{2+\snr+\sqrt{4+\snr^2}} + \frac{1}{2}\ln\paren{2-\snr+\sqrt{4+\snr^2}}+\frac{1}{4}\paren{2+\snr - \sqrt{4+\snr^2}}-\ln4 \notag \\
&= -R + \frac{1}{2}\ln\paren{\frac{1}{2}+\frac{1}{2}\sqrt{1+\frac{\snr^2}{4}}}+\frac{1}{2}+\frac{\snr}{4}-\frac{1}{2}\sqrt{1+\frac{\snr^2}{4}} \label{eqn:ee-mid-ud}
\end{align}
for $ 0\le R\le R_{\crit,1}(\snr) $. 

As for the expurgated exponent, to evaluate the bound in \Cref{thm:awgn-exe}, we first solve $t\in[1/2,1]$ from the equation $ (2t-1)e^{2R} = t^2 $ and get
$$t = e^{2R}\paren{1 - \sqrt{1 - e^{-2R}}}.$$ 
Substituting $t$ in \Cref{eqn:awgn-exe} yields
\begin{align}
E_1(R,\snr) &\ge \frac{\snr(2t-1)}{4t} = \frac{\snr \cdot t^2e^{-2R}}{4t} = \frac{\snr \cdot te^{-2R}}{4} = \frac{\snr}{4}\paren{1 - \sqrt{1-e^{-2R}}}. \label{eqn:ee-low-ud} 
\end{align}

It has been long known that for DMCs and AWGN channels, list-decoding under any subexponential (in $n$) list-sizes does not increase the channel capacity. 
Interestingly, our results show that list-decoding under constant list-sizes does not increase the error exponent of capacity-achieving codes. 
Indeed, for any $ \snr>0 $ and any constant $ L\in\bZ_{\ge2} $, the error exponent remains the same under $(L-1)$-list-decoding for any $ R_{\x,1}(\snr)\le R\le\frac{1}{2}\ln(1+\snr) $. 
However, list-decoding does boost the error exponent for any $ 0\le R\le R_{\x,1}(\snr) $. 
In particular, the critical rates under list-decoding move, i.e., $ R_{\crit,L-1}(\snr)>R_{\crit,1}(\snr) $ and $ R_{\x,L-1}(\snr)>R_{\x,1}(\snr) $ for any $ L\in\bZ_{>2} $.

Gallager's exponents and our list-decoding error exponents (for $ L=3 $) are plotted in \Cref{fig:LDUDEEFig} for $ \snr = 1 $.

\section{List-decoding error exponents of AWGN channels without input constraints}
\label{sec:ld-ee-unbdd}

In this section, we obtain bounds on the $(L-1)$-list-decoding error exponent of an AWGN channel with \emph{no} input constraint and noise variance $ \sigma^2 $. 
An \emph{unbounded} code for such a channel contains codewords whose norm can be arbitrarily large. 
The rate of such a code is measured by \Cref{eqn:density-unbounded}. 

\subsection{Random coding exponent}
\label{sec:rce-ppp}

\begin{theorem}
\label{thm:rce-ppp}
For any $ \sigma>0,\alpha\ge1 $ and $ L\in\bZ_{\ge2} $, there exists an unbounded code $ \cC\subset\bR^n $ of rate $ R = \frac{1}{2}\ln\frac{1}{2\pi e\sigma^2\alpha^2} $ such that when used over an AWGN channel with noise variance $ \sigma^2 $ and no input constraint, the exponent of the average probability of $(L-1)$-list-decoding error of $ \cC $ (normalized by $ \lim\limits_{n\to\infty}-\frac{1}{n}\ln(\cdot) $) is at least $ E_{\mathrm{r},L-1}(\alpha) $ defined as
\begin{align}
E_{\mathrm{r},L-1}(\alpha) &= \begin{cases}
\frac{\alpha^2}{2} - \ln\alpha - \frac{1}{2}, &1\le\alpha\le\sqrt{L} \\
\frac{L-1}{2} - \frac{L}{2}\ln L + (L-1)\ln\alpha, &\alpha>\sqrt{L}
\end{cases}. \notag
\end{align}
\end{theorem}

\begin{proof}
Let $ \alpha\ge1 $ and $ R = \frac{1}{2}\ln\frac{1}{2\pi e\sigma^2\alpha^2} $. 
Let $ \cC\subset\bR^n $ be a Poisson Point Process with intensity $ \lambda = e^{nR} = (2\pi e\sigma^2\alpha^2)^{-n/2} $. 
By translating $\cC$, we assume without loss of generality that $ \vzero\in\cC $. 
By \Cref{itm:ppp-prop-stationary} of \Cref{fact:ppp-properties}, the distribution of the translated process remains the same.

Let $ \cE_{L-1}^{\mathrm{ML}}(\cC) $ denote the error event of $\cC$ under ML $(L-1)$-list-decoding given $ \vzero $ is transmitted.
\begin{align}
\cE_{L-1}^{\mathrm{ML}}(\cC) 
&\coloneqq \curbrkt{\exists\curbrkt{\vx_1,\cdots,\vx_{L-1}}\in\binom{\cC\setminus\curbrkt{\vzero}}{L-1},\;\forall i\in[L-1],\;\normtwo{\vx_i - \vbfg}<\normtwo{\vbfg}}. \notag 
\end{align} 
For any instantiated $ \cC $, we can bound the probability of $ \cE_{L-1}^{\mathrm{ML}}(\cC) $ as follows. 
\begin{align}
\prob{\cE_{L-1}^{\mathrm{ML}}(\cC)}
&= \exptover{\bfr}{\prob{\cE_{L-1}^{\mathrm{ML}}(\cC)\condon\normtwo{\vbfg} = \bfr}} \notag \\
&= \int_0^\infty f_{\normtwo{\vbfg}}(r) \prob{\cE_{L-1}^{\mathrm{ML}}(\cC)\condon\normtwo{\vbfg} = r}\diff r \notag \\
&\le \int_0^{r^*} f_{\normtwo{\vbfg}}(r) \prob{\cE_{L-1}^{\mathrm{ML}}(\cC)\condon\normtwo{\vbfg} = r}\diff r + \int_{r^*}^\infty f_{\normtwo{\vbfg}}(r) \diff r. \label{eqn:campbell-first-second} 
\end{align}
The function $ f_{\normtwo{\vbfg}} $ denotes the p.d.f.\ of the $ \ell_2 $-norm of a Gaussian vector $ \vbfg\sim\cN(\vzero,\sigma^2I_n) $. 
The randomness of the above probability and expectation comes from the Gaussian noise $ \vbfg $. 
In \Cref{eqn:campbell-first-second}, $ r^*>0 $ is to be specified.

Conditioned on $ \vzero\in\cC $ being transmitted, the rest of $\cC$ follows the Palm distribution denoted by $ \bE^{\mathrm{Palm}} $ and $ \Pr^{\mathrm{Palm}} $. 
We now average \Cref{eqn:campbell-first-second} over the PPP $ \cC $. 
The second term is independent of of $ \cC $ and remains the same under averaging. 
As for the first term, we note that 
\begin{align}
\prob{\cE_{L-1}^{\mathrm{ML}}(\cC)\condon\normtwo{\vbfg} = r}
&= \prob{\exists\curbrkt{\vx_1,\cdots,\vx_{L-1}}\in\binom{\cC\setminus\curbrkt{\vzero}}{L-1},\;\forall i\in[L-1],\;\normtwo{\vx_i - \vbfg}<r\condon\normtwo{\vbfg} = r} \notag \\
&\le \min\curbrkt{ \sum_{\curbrkt{\vx_1,\cdots,\vx_{L-1}}\in\binom{\cC\setminus\curbrkt{\vzero}}{L-1}} \prob{\forall i\in[L-1],\;\vx\in\interior(\cB^n(\vbfg,r))\condon\normtwo{\vbfg} = r}, 1}. \label{eqn:bound-cond-error-prob} 
\end{align}
The first term in \Cref{eqn:campbell-first-second} then is at most
\begin{align}
& \sum_{\curbrkt{\vx_1,\cdots,\vx_{L-1}}\in\binom{\cC\setminus\curbrkt{\vzero}}{L-1}} \int_0^{r^*} f_{\normtwo{\vbfg}}(r) \prob{\forall i\in[L-1],\;\vx_i\in\interior(\cB^n(\vbfg,r))\condon\normtwo{\vbfg} = r} \diff r, \label{eqn:rce-term1-bd}
\end{align}
where we only used the first term of the minimization in \Cref{eqn:bound-cond-error-prob}. 
Now, the first term in \Cref{eqn:campbell-first-second} averaged over $\cC$ can be bounded as follows:
\begin{align}
& \exptpalmover{\cC}{\int_0^{r^*} f_{\normtwo{\vbfg}}(r) \probover{\vbfg}{\cE_{L-1}^{\mathrm{ML}}(\cC)\condon\normtwo{\vbfg} = r}\diff r} \notag \\
&= \exptover{\cC}{\int_0^{r^*} f_{\normtwo{\vbfg}}(r) \probover{\vbfg}{\cE_{L-1}^{\mathrm{ML}}(\cC)\condon\normtwo{\vbfg} = r}\diff r} \label{eqn:slivnyak-rce} \\
&\le \exptover{\cC}{\sum_{\curbrkt{\vbfx_1,\cdots,\vbfx_{L-1}}\in\binom{\cC\setminus\curbrkt{\vzero}}{L-1}} \int_0^{r^*} f_{\normtwo{\vbfg}}(r) \prob{\forall i\in[L-1],\;\vbfx_i\in\interior(\cB^n(\vbfg,r))\condon\normtwo{\vbfg} = r} \diff r} \label[ineq]{eqn:rce-use-term1-bd} \\
&= \int_{\bR^{n(L-1)}} \paren{\int_0^{r^*} f_{\normtwo{\vbfg}}(r) \probover{\vbfg}{\forall i\in[L-1],\;\vx_i\in\interior(\cB^n(\vbfg,r))\condon\normtwo{\vbfg} = r} \diff r} \lambda^{L-1}\diff(\vx_1,\cdots,\vx_{L-1}) \label{eqn:rce-campbell} \\
&= \int_0^{r^*} f_{\normtwo{\vbfg}}(r)\lambda^{L-1}\int_{\bR^{n(L-1)}}\probover{\vbfg}{\forall i\in[L-1],\;\vx_i\in\interior(\cB^n(\vbfg,r))\condon\normtwo{\vbfg} = r}\diff(\vx_1,\cdots,\vx_{L-1})\diff r \notag \\
&= \int_0^{r^*} f_{\normtwo{\vbfg}}(r)\lambda^{L-1}\int_{\bR^{n(L-1)}}\exptover{\vbfg}{\prod_{i = 1}^{L-1}\indicator{\vx_i\in\interior(\cB^n(\vbfg,r))}\condon\normtwo{\vbfg} = r}\diff(\vx_1,\cdots,\vx_{L-1})\diff r \notag \\
&= \int_0^{r^*} f_{\normtwo{\vbfg}}(r)\lambda^{L-1}\exptover{\vbfg}{\int_{\bR^{n(L-1)}}\prod_{i = 1}^{L-1}\indicator{\vx_i\in\interior(\cB^n(\vbfg,r))}\diff(\vx_1,\cdots,\vx_{L-1})\condon\normtwo{\vbfg} = r}\diff r \notag \\
&= \int_0^{r^*} f_{\normtwo{\vbfg}}(r)\lambda^{L-1}\exptover{\vbfg}{\prod_{i = 1}^{L-1}\int_{\bR^n}\indicator{\vx_i\in\interior(\cB^n(\vbfg,r))}\diff\vx_i\condon\normtwo{\vbfg} = r}\diff r \notag \\
&= \int_0^{r^*} f_{\normtwo{\vbfg}}(r)\lambda^{L-1} \exptover{\vbfg}{|\cB^n(r)|^{L-1}\condon\normtwo{\vbfg} = r}\diff r \notag \\
&= \lambda^{L-1}V_n^{L-1}\int_0^{r^*}f_{\normtwo{\vbfg}}(r)r^{n(L-1)}\diff r. \label{eqn:campbell-first-avg} 
\end{align}
\Cref{eqn:slivnyak-rce} is by Slivnyak's theorem (\Cref{thm:slivnyak}). 
\Cref{eqn:rce-use-term1-bd} follows from \Cref{eqn:rce-term1-bd}. 
\Cref{eqn:rce-campbell} is by Campbell's theorem (\Cref{thm:campbell}). 

We choose $ r^* $ such that the sum of \Cref{eqn:campbell-first-avg} and the second term in \Cref{eqn:campbell-first-second} is minimized. 
That is, $ r^* $ is a zero of the derivative (w.r.t.\ $ r^* $) of the sum. 
Recall the way one takes derivative w.r.t.\ the limit of an integral. 
If
\begin{align}
F(x) &= \int_a^xf(t)\diff t, \notag
\end{align}
then 
\begin{align}
\frac{\diff}{\diff x} F(x) &= f(x). \notag 
\end{align}
Therefore, $ r^* $ satisfies
\begin{align}
\lambda^{L-1}V_n^{L-1}f_{\normtwo{\vbfg}}(r^*)(r^*)^{n(L-1)} - f_{\normtwo{\vbfg}}(r^*) = 0
&&\implies&& r^* = \lambda^{-1/n}V_n^{-1/n}. 
\notag 
\end{align}
By the choice of $ \lambda $, we further have
\begin{align}
r^* &= e^{-R}\paren{\frac{2\pi e}{n}}^{-1/2}(1+o(1)) 
= \exp\paren{-\frac{1}{2}\ln\frac{1}{2\pi e\sigma^2\alpha^2}}(2\pi e)^{-1/2}\sqrt{n}(1+o(1))
= \alpha\sigma\sqrt{n}(1+o(1)). \label{eqn:def-rstar}
\end{align}

Next, we evaluate the bound we got for the error probability
\begin{align}
\lambda^{L-1}V_n^{L-1}\int_0^{r^*}f_{\normtwo{\vbfg}}(r)r^{n(L-1)}\diff r + \prob{\normtwo{\vbfg}>r^*}. \label{eqn:first-second-eval} 
\end{align}
The density of the $ \ell_2 $-norm of a Gaussian vector of variance $ \sigma^2 $ is 
\begin{align}
f_{\normtwo{\vbfg}}(r) = \sigma^{-1} f(r/\sigma), 
\label{eqn:gaussian-norm-density-relation}
\end{align} 
where $ f(\cdot) $ is the density of the $ \ell_2 $-norm $ \normtwo{\vbfg_0} $ of a standard Gaussian vector $ \vbfg_0\sim\cN(\vzero,I_n) $. 
Neglecting the $o(1)$ factor in $ r^* $ (\Cref{eqn:def-rstar}), we get that the first term of \Cref{eqn:first-second-eval} (dot) equals
\begin{align}
& (2\pi e\sigma^2\alpha^2)^{-\frac{1}{2}n(L-1)}\paren{\frac{2\pi e}{n}}^{\frac{1}{2}n(L-1)}\int_0^{\alpha\sigma\sqrt{n}} \sigma^{-1}f(r/\sigma) r^{n(L-1)} \diff r \notag \\
&= (\sigma^2\alpha^2n)^{-\frac{1}{2}n(L-1)} \int_0^{\alpha} \sigma^{-1} f(s\sqrt{n})(s\sigma\sqrt{n})^{n(L-1)} \sigma\sqrt{n} \diff s \label{eqn:change-var-first} \\
&= \alpha^{-n(L-1)}\sqrt{n}\int_0^\alpha f(s\sqrt{n})s^{n(L-1)} \diff s. \notag 
\end{align}
In \Cref{eqn:change-var-first}, we let $ s = \frac{r}{\sigma\sqrt{n}} $. 

The following asymptotics of $ f(\cdot) $ was obtained in \cite[Eqn.\ (129)]{anantharam-baccelli-2010-pp-long}. 
\begin{lemma}[\cite{anantharam-baccelli-2010-pp-long}]
\label{lem:bound-norm-gaussian-pdf}
The p.d.f.\ $f(\cdot)$ of the $ \ell_2 $-norm of an $n$-dimensional standard Gaussian vector satisfies the following pointwise estimate:
\begin{align}
f(s\sqrt{n}) &= \exp\paren{-n\paren{\frac{s^2}{2} - \ln s - \frac{1}{2}} + o(n)}, \notag 
\end{align}
for any $ s\ge0 $. 
\end{lemma}

By \Cref{lem:bound-norm-gaussian-pdf}, the first term of \Cref{eqn:first-second-eval} dot equals
\begin{align}
\int_0^\alpha \exp\paren{-n\sqrbrkt{\paren{\frac{s^2}{2} - \ln s - \frac{1}{2}} - (L-1)\ln s + (L-1)\ln\alpha}} \diff s 
&= \int_0^\alpha \exp\paren{-n\sqrbrkt{\frac{s^2}{2} - L\ln s - \frac{1}{2} + (L-1)\ln\alpha}} \diff s \label{eqn:ppprce-int-tocomp} 
\end{align}
where we have suppressed the polynomial factor $\sqrt{n}$. 
To evaluate the integral in \Cref{eqn:ppprce-int-tocomp}, we will apply the Laplace's method (\Cref{thm:lap-1d}).
It is easy to check that the function $ F(s) \coloneqq \frac{s^2}{2} - L\ln s - \frac{1}{2} + (L-1)\ln\alpha $ is decreasing in $ s\in[0,\sqrt{L}] $ and is increasing in $ s\in[\sqrt{L},\infty] $. 

If $ \alpha\le\sqrt{L} $, the minimum value of $F(s)$ in $ [0,\alpha] $ is achieved at $ s=\alpha $. 
By \Cref{thm:lap-1d}, the integral in \Cref{eqn:ppprce-int-tocomp} dot equals
\begin{align}
\exp\paren{-n\sqrbrkt{\frac{\alpha^2}{2} - L\ln\alpha - \frac{1}{2}+(L-1)\ln\alpha}}
= \exp\paren{-n\sqrbrkt{\frac{\alpha^2}{2} - \ln\alpha - \frac{1}{2}}}. \label{eqn:first-case1} 
\end{align}
If $ \alpha>\sqrt{L} $, the minimum value of $ F(s) $ in $ [0,\alpha] $ is achieved at $ s = \sqrt{L} $. 
By \Cref{thm:lap-1d}, the integral in \Cref{eqn:ppprce-int-tocomp} dot equals
\begin{align}
\exp\paren{-n\sqrbrkt{\frac{L}{2} - \frac{L}{2}\ln L - \frac{1}{2} + (L-1)\ln\alpha}}
= \exp\paren{-n\sqrbrkt{\frac{L-1}{2} - \frac{L}{2}\ln L + (L-1)\ln\alpha}}. \label{eqn:first-case2}
\end{align}

Let $ E_1(\alpha,L) $ and $ E_2(\alpha,L) $ be the normalized first-order exponent of the first and second term in \Cref{eqn:first-second-eval}, respectively, i.e., 
\begin{align}
E_1(\alpha,L) &\coloneqq -\lim_{n\to\infty}\frac{1}{n}\ln\paren{\lambda^{L-1}V_n^{L-1}\int_0^{r^*}f_{\normtwo{\vbfg}}(r)r^{n(L-1)}\diff r}, \notag \\
E_2(\alpha,L) &\coloneqq -\lim_{n\to\infty}\frac{1}{n}\ln\prob{\normtwo{\vbfg}>r^*}. \notag 
\end{align}

By \Cref{eqn:first-case1,eqn:first-case2}, $ E_1(\alpha,L) $ is given by 
\begin{align}
E_1(\alpha,L) &= \begin{cases}
\frac{\alpha^2}{2} - \ln\alpha - \frac{1}{2}, &\alpha\le\sqrt{L} \\
\frac{L-1}{2} - \frac{L}{2}\ln L + (L-1)\ln\alpha, &\alpha>\sqrt{L}
\end{cases}. \notag 
\end{align}
Let $ C \coloneqq \frac{1}{2}\ln\frac{1}{2\pi e\sigma^2} $. 
Note that $ R = \frac{1}{2}\ln\frac{1}{2\pi e\sigma^2\alpha^2} = C - \ln\alpha $. 
The exponent $ E_2(\alpha,L) $ is the large deviation exponent of the tail of a chi-square random variable which is given by \Cref{fact:asymp-chi-sq}. 
In fact, it was shown in \cite[Eqn.\ (29)]{ingber-2012-finite-dim-infinite-const} and \cite{poltyrev1994coding} that, under the choice of $ r^* $ given by \Cref{eqn:def-rstar}, we have
\begin{align}
E_2(\alpha,L) = \frac{1}{2}\sqrbrkt{e^{2(C-R)} - 1 - 2(C-R)}
= \frac{1}{2}\paren{e^{2\ln\alpha} - 1 - 2\ln\alpha}
= \frac{1}{2}(\alpha^2 - 1 - 2\ln\alpha)
= \frac{\alpha^2}{2} - \ln\alpha - \frac{1}{2}. \notag 
\end{align}
Note that $ E_2(\alpha,L) $ coincides with $ E_1(\alpha,L) $ for $ 1\le\alpha\le\sqrt{L} $ whereas it strictly dominates $ E_1(\alpha,L) $ when $ \alpha>\sqrt{L} $. 

Finally, 
\begin{align}
-\lim_{n\to\infty}\frac{1}{n}\ln \exptover{\cC}{\prob{\cE_{L-1}^{\mathrm{ML}}(\cC)\condon\cC}}
&\ge -\lim_{n\to\infty}\frac{1}{n}\ln\Cref{eqn:first-second-eval} \notag \\
&\ge \min\curbrkt{E_1(\alpha,L), E_2(\alpha,L)} \notag \\
&= \begin{cases}
\frac{\alpha^2}{2} - \ln\alpha - \frac{1}{2}, &1\le\alpha\le\sqrt{L} \\
\frac{L-1}{2} - \frac{L}{2}\ln L + (L-1)\ln\alpha, &\alpha>\sqrt{L}
\end{cases}. \qedhere
\end{align}
\end{proof}

\subsection{Expurgated exponent}
\label{sec:exe-matern}

The bound on error exponent proved in the last section (\Cref{sec:rce-ppp}) can be improved using the expurgation technique when the rate is sufficiently low. 
In this section, we prove the following theorem. 
\begin{theorem}
\label{thm:exe-matern}
For any $ \sigma>0,\alpha\ge1 $ and $ L\in\bZ_{\ge2} $, there exists an unbounded code $ \cC\subset\bR^n $ of rate $ R = \frac{1}{2}\ln\frac{1}{2\pi e\sigma^2\alpha^2} $ such that when used over an AWGN channel with noise variance $ \sigma^2 $ and no input constraint, the exponent of the average probability of $(L-1)$-list-decoding error of $ \cC $ (normalized by $ \lim\limits_{n\to\infty}-\frac{1}{n}\ln(\cdot) $) is at least $ E_{\ex,L-1}(\alpha) $ defined as
\begin{align}
E_{\ex,L-1}(\alpha) &= \begin{cases}
\frac{\alpha^2}{2} - \ln\alpha - \frac{1}{2}, &1\le\alpha\le\sqrt{L} \\
\frac{L-1}{2} - \frac{L}{2}\ln L + (L-1)\ln\alpha, &\sqrt{L}<\alpha\le\sqrt{2L} \\
F(\alpha,L), &\alpha>\sqrt{2L}
\end{cases}, \label{eqn:ee-ld-unbdd}
\end{align}
where
\begin{align}
F(\alpha,L) &\coloneqq \frac{\alpha^2}{16} + \frac{1}{16}\sqrt{\alpha^4 + 8\alpha^2(2L-3)+16} - \frac{L-1}{2}\ln\paren{\sqrt{\alpha^4 + 8\alpha^2(2L-3) + 16} - \alpha^2 + 4} \notag \\
&\quad + \frac{L-2}{2}\ln\paren{\sqrt{\alpha^4 + 8\alpha^2(2L-3) + 16} + \alpha^2 + 4} + \frac{3}{2}\ln2 - \frac{1}{4} . \notag 
\end{align}
\end{theorem}

\begin{proof}
Let $ \alpha\ge1 $ and $ R = \frac{1}{2}\ln\frac{1}{2\pi e\sigma^2\alpha^2} $. 
Let $ \cC\subset\bR^n $ be a \matern process obtained from a PPP with intensity $ \lambda = e^{nR} = (2\pi e\sigma^2\alpha^2)^{-n/2} $ and exclusion radius $ \xi \coloneqq \wt\alpha\sigma\sqrt{n} $ where $ \wt\alpha \coloneqq \alpha(1 - \eps_n) $ for a proper choice of $ \eps_n\xrightarrow{n\to\infty}0 $ to be specified momentarily. 
The intensity $ \lambda' $ of the \matern process is 
\begin{align}
\lambda' &= \lambda \exp\paren{-\lambda|\cB^n(\xi)|} \notag \\
&= \lambda\exp\paren{-\lambda V_n(\alpha(1 - \eps_n)\sigma\sqrt{n})^n} \notag \\
&\asymp \lambda\exp\paren{- (2\pi e\sigma^2\alpha^2)^{-n/2} \frac{1}{\sqrt{\pi n}}\paren{\frac{2\pi e}{n}}^{n/2}(\alpha^2(1-\eps_n)^2\sigma^2n)^{n/2}} \notag \\
&= \lambda\exp\paren{-\frac{(1-\eps_n)^n}{\sqrt{\pi n}}}. \notag 
\end{align}
Taking $ \eps_n = \frac{\ln n}{n} = o(1) $, we have 
\begin{align}
\lambda' &\asymp \lambda\exp\paren{-\frac{e^{-\ln n}}{\sqrt{\pi n}}}
= \lambda \exp\paren{-\pi^{-1/2}n^{-3/2}}
\asymp \lambda. \notag 
\end{align}
In the following analysis, we will ignore the $o(1)$ factor $ \eps_n $ and assume for simplicity $ \wt\alpha = \alpha $. 

Suppose $ \vzero\in\cC $. 
Under the Palm distribution, the order-$(L-1)$ factorial moment measure $ \lambda'(\vx_1,\cdots,\vx_{L-1}) $ of $\cC$ can be bounded as follows
\begin{align}
\lambda'(\vx_1,\cdots,\vx_{L-1}) 
&\le \lambda^{L-1} \prod_{i = 1}^{L=1} \indicator{\vx_i\in\cB^n( \xi)^c}. \label[ineq]{eqn:bound-intensity-matern} 
\end{align}

Following similar arguments to those in \Cref{sec:rce-ppp}, we have
\begin{align}
\prob{\cE_{L-1}^{\mathrm{ML}}(\cC)} 
&= \int_0^\infty f_{\normtwo{\vbfg}}(r)\prob{\cE_{L-1}^{\mathrm{ML}}(\cC)\condon\normtwo{\vbfg} = r}\diff r. \notag 
\end{align}
The above identity holds for any instantiated $ \cC\subset\bR^n $ and the randomness in the probability comes from the channel noise $ \vbfg $. 
Averaging the RHS of the above equation over the \matern process $ \cC $, we have
\begin{align}
& \exptpalmover{\cC}{\int_0^\infty f_{\normtwo{\vbfg}}(r) \probover{\vbfg}{\cE_{L-1}^{\mathrm{ML}}(\cC)\condon\normtwo{\vbfg} = r} \diff r} \notag \\
&\le \int_{\bR^{n(L-1)}} \paren{\int_0^\infty f_{\normtwo{\vbfg}}(r)\probover{\vbfg}{\forall i\in[L-1],\;\vx_i\in\interior(\cB^n(\vbfg,r))\condon\normtwo{\vbfg} = r} \diff r} \lambda^{L-1}\prod_{i = 1}^{L-1}\indicator{\vx_i\in\cB^n(\xi)^c}\diff(\vx_1,\cdots,\vx_{L-1}) \label[ineq]{eqn:jump-steps} \\
&= \int_0^\infty f_{\normtwo{\vbfg}}(r)\lambda^{L-1}\int_{\bR^{n(L-1)}} \exptover{\vbfg}{\prod_{i = 1}^{L-1}\indicator{\vx_i\in\interior(\cB^n(\vbfg,r))}\condon\normtwo{\vbfg} = r}\prod_{i = 1}^L\indicator{\vx_i\in\cB^n(\xi)^c} \diff(\vx_1,\cdots,\vx_{L-1})\diff r \notag \\
&= \int_0^\infty f_{\normtwo{\vbfg}}(r)\lambda^{L-1}\int_{\bR^{n(L-1)}} \exptover{\vbfg}{\prod_{i = 1}^{L-1}\indicator{\vx_i\in\interior(\cB^n(\vbfg,r))} \indicator{\vx_i\in\cB^n(\xi)^c} \condon\normtwo{\vbfg} = r} \diff(\vx_1,\cdots,\vx_{L-1})\diff r \notag \\
&= \int_0^\infty f_{\normtwo{\vbfg}}(r)\lambda^{L-1} \exptover{\vbfg}{\prod_{i = 1}^{L-1}\int_{\bR^{n}}\indicator{\vx_i\in\interior(\cB^n(\vbfg,r))\cap\cB^n(\xi)^c} \diff\vx_i \condon\normtwo{\vbfg} = r} \diff r \label{eqn:int-dir-g} \\
&= \int_0^\infty f_{\normtwo{\vbfg}}(r)\lambda^{L-1} \card{\cB^n(r\ve,r)\cap\cB^n(\xi)^c}^{L-1} \diff r, \label{eqn:direction-g} 
\end{align}
where $\ve = [1,0,\cdots,0]\in\bR^n $. 
In \Cref{eqn:jump-steps}, we skipped several steps which are similar to \Cref{eqn:slivnyak-rce}, \Cref{eqn:rce-use-term1-bd} and \Cref{eqn:rce-campbell}. 
In particular, we used Slivnyak's theorem (\Cref{thm:slivnyak}), the first bound of the minimum in \Cref{eqn:bound-cond-error-prob}, Campbell's theorem (\Cref{thm:campbell}) and the bound on the (Palm) intensity of \matern processes (\Cref{eqn:bound-intensity-matern}). 
In \Cref{eqn:direction-g}, we take the direction of $ \vbfg $ to be $ \ve $ since the integral in \Cref{eqn:int-dir-g} does not depend on the direction of $ \vbfg $. 

Incorporating the second term of the minimum in \Cref{eqn:bound-cond-error-prob}, we get
\begin{align}
\exptover{\cC}{\int_0^\infty f_{\normtwo{\vbfg}}(r) \probover{\vbfg}{\cE_{L-1}^{\mathrm{ML}}(\cC)\condon\normtwo{\vbfg} = r} \diff r}
&\le \int_0^\infty f_{\normtwo{\vbfg}}(r)\min\curbrkt{\lambda^{L-1} \card{\cB^n(r\ve,r)\cap\cB^n(\xi)^c}^{L-1}, 1} \diff r. \notag 
\end{align}
We apply the relation \Cref{eqn:gaussian-norm-density-relation}, change variable $ s = \frac{r}{\sigma\sqrt{n}} $ and get
\begin{align}
\exptover{\cC}{\prob{\cE_{L - 1}^{\mathrm{ML}}(\cC)\condon\cC}}
&\le \int_0^\infty \sigma^{-1}f(r/\sigma)\min\curbrkt{\lambda^{L-1}\card{\cB^n(r\ve,r)\cap\cB^n(\xi)^c}^{L-1},1} \diff r \notag \\
&= \int_0^\infty \sigma^{-1}f(s\sqrt{n}) \min\curbrkt{\lambda^{L-1}\card{\cB^n(r\ve,r)\cap\cB^n(\xi)^c}^{L-1},1} \sigma\sqrt{n}\diff s \notag \\
&\doteq \int_0^\infty f(s\sqrt{n}) \min\curbrkt{\lambda^{L-1}\card{\cB^n(r\ve,r)\cap\cB^n(\alpha\sigma\sqrt{n})^c}^{L-1},1} \diff s. \label{eqn:tbc-matern} 
\end{align}
The following upper bound on $ \card{\cB^n(r\ve,r)\cap\cB^n(\alpha\sigma\sqrt{n})^c} $ was shown in \cite[Eqn.\ (106)]{anantharam-baccelli-2010-pp-long}. 
\begin{lemma}[\cite{anantharam-baccelli-2010-pp-long}]
\label{lem:asymp-intersection-ball}
Let $ \ve = [1,0,\cdots,0]\in\bR^n,\alpha\ge1,\sigma>0 $. 
Then for any $ r>0 $, 
\begin{align}
\card{\cB^n(r\ve,r)\cap\cB^n(\alpha\sigma\sqrt{n})^c} &\le \card{\cB^n(c(s)\sigma\sqrt{n})} = V_n(c(s)\sigma\sqrt{n})^n, \notag
\end{align}
where
\begin{align}
c(s) &= \begin{cases}
0, &0<s\le\alpha/2 \\
\sqrt{s^2 - \paren{s - \frac{\alpha^2}{2s}}^2}, &\alpha/2<s\le\alpha/\sqrt{2} \\
s, &s > \alpha/\sqrt{2}
\end{cases}. 
\notag 
\end{align}
\end{lemma}
Using \Cref{lem:asymp-intersection-ball}, continuing with \Cref{eqn:tbc-matern}, we have
\begin{align}
\exptover{\cC}{\prob{\cE_{L - 1}^{\mathrm{ML}}(\cC)\condon\cC}}
&\dot\le \int_0^\infty f(s\sqrt{n})\min\curbrkt{\lambda^{L-1}V_n^{L-1}(c(s)\sigma\sqrt{n})^{n(L-1)},1} \diff s \notag \\
&\doteq \int_0^\infty f(s\sqrt{n})\min\curbrkt{ (2\pi e\sigma^2\alpha^2)^{-\frac{1}{2}n(L-1)} (2\pi en^{-1})^{\frac{1}{2}n(L-1)} (c(s)^2\sigma^2n)^{\frac{1}{2}n(L-1)},1} \diff s \notag \\
&= \int_0^\infty f(s\sqrt{n}) \min\curbrkt{\alpha^{-n(L-1)}c(s)^{n(L-1)},1} \diff s\notag \\
&\le \int_0^\infty \exp\paren{-n\curbrkt{\frac{s^2}{2} - \ln s - \frac{1}{2} + (L-1)\sqrbrkt{\ln\alpha - \ln c(s)}^+}} \diff s \notag \\
&= \int_0^{\alpha/2} \exp\paren{-nF_1(s)} \diff s + \int_{\alpha/2}^{\alpha/\sqrt{2}}\exp\paren{-nF_2(s)}\diff s + \int_{\alpha/\sqrt{2}}^\infty \exp\paren{-nF_3(s)} \diff s, \label{eqn:matern-three-int} 
\end{align}
where $ F_1(s),F_2(s),F_3(s) $ are defined as follows
\begin{align}
F_1(s) &\coloneqq \frac{s^2}{2} - \ln s - \frac{1}{2} + (L-1)(\ln\alpha - \ln0) = \infty, \quad 0\le s\le\alpha/2; \notag \\
F_2(s) &\coloneqq \frac{s^2}{2} - \ln s - \frac{1}{2} + (L-1)\sqrbrkt{\ln\alpha - \frac{1}{2}\ln\paren{s^2 - \paren{s - \frac{\alpha^2}{2s}}^2}}^+ \notag \\
&= \frac{s^2}{2} - \ln s - \frac{1}{2} + (L-1)\sqrbrkt{\ln\alpha - \frac{1}{2}\ln\paren{s^2 - \paren{s - \frac{\alpha^2}{2s}}^2}}, \quad \alpha/2<s\le\alpha/\sqrt{2}; \notag \\
F_3(s) &\coloneqq \frac{s^2}{2} - \ln s - \frac{1}{2} + (L-1)[\ln\alpha - \ln s]^+ \notag \\
&= \begin{cases}
\frac{s^2}{2} - \ln s - \frac{1}{2} + (L-1)(\ln\alpha - \ln s), & \alpha/\sqrt{2}< s\le\alpha \\
\frac{s^2}{2} - \ln s - \frac{1}{2}, &s>\alpha 
\end{cases} . \notag 
\end{align}
For $ F_2(s) $, we can remove the function $ [\cdot]^+ $ since the function $ f_2(s)\coloneqq\sqrt{s^2 - \paren{s - \frac{\alpha^2}{2s}}^2} $ attains its maximum value $ \alpha/\sqrt{2} $ at $ s = \alpha/\sqrt{2} $. 
Therefore $ \ln\alpha - \ln f_2(s)\ge\ln\sqrt{2}>0 $. 

Define
\begin{align}
E_1(\alpha,L) &\coloneqq - \lim_{n\to\infty}\frac{1}{n}\ln \int_0^{\alpha/2} \exp\paren{-nF_1(s)} \diff s , \notag \\
E_2(\alpha,L) &\coloneqq - \lim_{n\to\infty}\frac{1}{n}\ln \int_{\alpha/2}^{\alpha/\sqrt{2}}\exp\paren{-nF_2(s)}\diff s, \notag \\
E_3(\alpha,L) &\coloneqq - \lim_{n\to\infty}\frac{1}{n}\ln \int_{\alpha/\sqrt{2}}^\infty \exp\paren{-nF_3(s)}. \notag 
\end{align}

We compute $ E_1(\alpha,L),E_2(\alpha,L),E_3(\alpha,L) $ using Laplace's method (\Cref{thm:lap-1d}). 

For $ E_1(\alpha,L) $, we have
\begin{align}
E_1(\alpha,L) &= \min_{s\in[0,\alpha/2]} F_1(s) = \infty. \notag 
\end{align}

For $ E_2(\sigma,L) $, $ F_2(s) $ has a unique stationary point 
\begin{align}
s_0 &= \sqrt{\frac{\alpha^2 + \sqrt{\alpha^4 + 8\alpha^2(2L-3) + 16} + 4}{8}}. \notag 
\end{align}
One can check that $ s_0\ge\sqrt{2}/\alpha $ if $ \alpha\le\sqrt{2L} $ and $ s_0<\sqrt{2}/\alpha $ if $ \alpha>\sqrt{2L} $. 
Therefore
\begin{align}
E_2(\alpha,L) &= \min_{s\in(\alpha/2,\alpha/\sqrt{2}]} F_2(s) 
= \begin{cases}
F_2(\sqrt{2}/\alpha), &\alpha\le\sqrt{2L} \\
F_2(s_0), &\alpha>\sqrt{2L}
\end{cases}, \notag 
\end{align}
where 
\begin{align}
F_2(\sqrt{2}/\alpha) &= \frac{\alpha^2}{4} + \ln\alpha + \frac{L}{2}\ln2 - \frac{1}{2}, \notag \\
F_2(s_0) &= \frac{\alpha^2}{16} + \frac{1}{16}\sqrt{\alpha^4 + 8\alpha^2(2L-3)+16} - \frac{L-1}{2}\ln\paren{\sqrt{\alpha^4 + 8\alpha^2(2L-3) + 16} - \alpha^2 + 4} \notag \\
&\quad + \frac{L-2}{2}\ln\paren{\sqrt{\alpha^4 + 8\alpha^2(2L-3) + 16} + \alpha^2 + 4} + \frac{3}{2}\ln2 - \frac{1}{4} . \label{eqn:f2-s0}
\end{align}

For $ E_3(\alpha,L) $, we let 
\begin{align}
F_{3,1}(s) &\coloneqq \frac{s^2}{2} - \ln s - \frac{1}{2} + (L-1)(\ln\alpha - \ln s), \quad
F_{3,2}(s) \coloneqq \frac{s^2}{2} - \ln s - \frac{1}{2}. \notag 
\end{align}
The function $ F_{3,1}(s) $ has a unique minimum point $ s = \sqrt{L} $. Therefore, for $ s\in(\alpha/\sqrt{2},\alpha] $, the minimum value of $ F_{3,1}(s) $ is 
\begin{align}
\min_{s\in(\alpha/\sqrt{2},\alpha]}F_{3,1}(s) &=
\begin{cases}
F_{3,1}(\alpha) = \frac{\alpha^2}{2} - \ln\alpha - \frac{1}{2}, & 1\le\alpha\le\sqrt{L} \\
F_{3,1}(\sqrt{L}) = \frac{L-1}{2} - \frac{L}{2}\ln L + (L-1)\ln\alpha, &\sqrt{L}<\alpha\le\sqrt{2L} \\
F_{3,1}(\alpha/\sqrt{2}) = \frac{\alpha^2}{4} - \ln\alpha + \frac{L}{2}\ln2 - \frac{1}{2}, &\alpha>\sqrt{2L}
\end{cases}. \label{eqn:f31} 
\end{align}
The function $ F_{3,2}(s) $ has a unique minimum point $ s = 1\le\alpha $. 
Therefore, for $ s\in(\alpha,\infty) $, the minimum value of $ F_{3,2}(s) $ is 
\begin{align}
\min_{s\in(\alpha,\infty)}F_{3,2}(s) &= F_{3,2}(1) = \frac{\alpha^2}{2} - \ln\alpha - \frac{1}{2}. \label{eqn:f32} 
\end{align}

One can easily check that \Cref{eqn:f32} is at least \Cref{eqn:f31} for any $ \alpha\ge1 $.
Therefore,
\begin{align}
E_3(\alpha,L) &= \min_{s\in(\alpha/\sqrt{2},\infty)} F_3(s)
= \min\curbrkt{\min_{s\in(\alpha/\sqrt{2},\alpha]}F_{3,1}(s),\min_{s\in(\alpha,\infty)}F_{3,2}(s)}
= \min_{s\in(\alpha/\sqrt{2},\alpha]}F_{3,1}(s). \notag 
\end{align}

Finally, 
\begin{align}
-\lim_{n\to\infty}\frac{1}{n}\ln \exptover{\cC}{\prob{\cE_{L-1}^{\mathrm{ML}}(\cC)\condon\cC}}
&\ge \min\curbrkt{E_1(\alpha,L),E_2(\alpha,L),E_3(\alpha,L)} \notag \\
&= \min\curbrkt{E_2(\alpha,L),E_3(\alpha,L)} \notag \\
&= \begin{cases}
\min\curbrkt{F_2(\alpha/\sqrt{2}), F_{3,1}(\alpha)}, &1\le\alpha\le\sqrt{L} \\
\min\curbrkt{F_2(\alpha/\sqrt{2}), F_{3,1}(\sqrt{L})}, &\sqrt{L}<\alpha\le\sqrt{2L} \\
\min\curbrkt{F_2(s_0), F_{3,1}(\alpha/\sqrt{2})}, &\alpha>\sqrt{2L}
\end{cases} \notag \\
&= \begin{cases}
F_{3,1}(\alpha), &1\le\alpha\le\sqrt{L} \\
F_{3,1}(\sqrt{L}), &\sqrt{L}<\alpha\le\sqrt{2L} \\
F_2(s_0), &\alpha>\sqrt{2L}
\end{cases} \notag \\
&= \begin{cases}
\frac{\alpha^2}{2} - \ln\alpha - \frac{1}{2}, &1\le\alpha\le\sqrt{L} \\
\frac{L-1}{2} - \frac{L}{2}\ln L + (L-1)\ln\alpha, &\sqrt{L}<\alpha\le\sqrt{2L} \\
F_2(s_0), &\alpha>\sqrt{2L}
\end{cases}. \notag
\end{align}
Recall that the quantity $ F_2(s_0) $ was defined in \Cref{eqn:f2-s0}. 
This finishes the proof. 
\end{proof}

\subsection{List-decoding error exponents vs.\ unique-decoding error exponents}
\label{sec:ld-ee-vs-ud-ee-unbdd}

Our results on list-decoding error exponents of AWGN channels without input constraints recover those by Poltyrev \cite[Theorem 3]{poltyrev1994coding} for unique-decoding\footnote{See also \cite[Eqn.\ (108)]{anantharam-baccelli-2010-pp-long} for a parameterization of Poltyrev's bound using $\alpha$.}. 
Indeed, when $L=2$, \Cref{eqn:ee-ld-unbdd} specializes to 
\begin{align}
E_{\ex,L-1}(\alpha) &= \begin{cases}
\frac{\alpha^2}{2} - \ln\alpha-\frac{1}{2}, &1\le\alpha\le\sqrt{2} \\
\frac{1}{2}-\ln2+\ln\alpha, &\sqrt{2}\le\alpha\le2 \\
\frac{\alpha^2}{8}, &\alpha>2
\end{cases}. \label{eqn:ee-ud-unbdd}
\end{align}

The situation here is similar to the bounded case as discussed in \Cref{sec:ld-ee-vs-ud-ee}. 
List-decoding for input unconstrained AWGN channels does not increase the capacity and moreover does not increase the error exponent for any $ 1\le\alpha\le\sqrt{2} $. 
However, for any $ \alpha>\sqrt{2} $, list-decoding does increase the error exponent. 
Furthermore, the critical values of $\alpha$ move from $ \sqrt{2} $ and $ 2 $ to $ \sqrt{L} $ and $ \sqrt{2L} $, respectively, under list-decoding. 

We plot Poltyrev's exponents and our exponents (for $ L=3 $) in \Cref{fig:LDUDEEunbddFig}.

\section{Open questions}
\label{sec:open}

The problem of packing spheres in $ \ell_p $ space was also addressed in the literature \cite{rankin-sphericalcap-1955,spence-1970-lp-packing,ball-1987-lp-packing,samorodnitsky-l1}. 
Recently, there was an exponential improvement on the optimal packing density in $ \ell_p $ space \cite{sah-2020-lp-sphere-packing} relying on the Kabatiansky--Levenshtein bound \cite{kabatiansky-1978}. 
It is worth exploring the $ \ell_p $ version of the multiple packing problem. 

Our lower bound is proved via a very interesting connection to error exponents. 
We do not know how to directly analyze the tail probability of the Chebyshev radius, even for Gaussian codes. 
One can view it as the tail of the maximum of a certain Gaussian process. 
This looks like a proper venue where the chaining method \cite{vanhandel-2014-high-dim-prob} is applicable. 
However, it seems unlikely that one can extract a meaningful exponent using the generic chaining machinery. 
Note that for the purpose of maximizing the rate, we do care about the exact exponent, not only an exponentially decaying bound. 

For large $L$, our results imply that the list sizes must scale as $O(\frac{1}{\varepsilon}\ln\frac{1}{\varepsilon})$ for rates that are $\varepsilon$-close to capacity. The same can be obtained using different approaches~\cite{zhang2022listtrans-it}. An interesting open question is to resolve whether this is indeed the best possible scaling as a function of $\varepsilon$.

			

\section{Acknowledgement}
YZ also would like to thank Nir Ailon and Ely Porat for several helpful conversations throughout this project, and Alexander Barg for insightful comments on the manuscript. 

YZ has received funding from the European Union's Horizon 2020 research and innovation programme under grant agreement No 682203-ERC-[Inf-Speed-Tradeoff]. 
The work of SV was supported by a seed grant from IIT Hyderabad and the start-up research grant from the Science and Engineering Research Board, India (SRG/2020/000910). 


\bibliographystyle{alpha}
\bibliography{ref} 

\newcommand{\etalchar}[1]{$^{#1}$}
\begin{thebibliography}{BADTS20}

\bibitem[AB08]{ahlswede-blinovsky-2008-book}
Rudolf Ahlswede and Vladimir Blinovsky.
\newblock {\em Lectures on advances in combinatorics}.
\newblock Universitext. Springer-Verlag, Berlin, 2008.

\bibitem[AB10]{anantharam-baccelli-2010-pp-long}
Venkat Anantharam and Fran{\c{c}}ois Baccelli.
\newblock Information-theoretic capacity and error exponents of stationary
  point processes under random additive displacements.
\newblock {\em arXiv preprint arXiv:1012.4924}, 2010.

\bibitem[ABL00]{abl-2000-list-size-2}
Alexei Ashikhmin, Alexander Barg, and Simon Litsyn.
\newblock A new upper bound on codes decodable into size-2 lists.
\newblock In {\em Numbers, Information and Complexity}, pages 239--244.
  Springer, 2000.

\bibitem[BADTS20]{ben-aroya-doron-ta-shma-2018-explicit-erasure-ld}
Avraham Ben-Aroya, Dean Doron, and Amnon Ta-Shma.
\newblock Near-optimal erasure list-decodable codes.
\newblock In {\em 35th Computational Complexity Conference (CCC 2020)}. Schloss
  Dagstuhl-Leibniz-Zentrum f{\"u}r Informatik, 2020.

\bibitem[Bal87]{ball-1987-lp-packing}
Keith Ball.
\newblock Inequalities and sphere-packing inl p.
\newblock {\em Israel Journal of Mathematics}, 58(2):243--256, 1987.

\bibitem[Bas65]{bassalygo-pit1965}
L.~A. Bassalygo.
\newblock {New upper bounds for error-correcting codes}.
\newblock {\em Probl. of Info. Transm.}, 1:32--35, 1965.

\bibitem[BBJ19]{bhattacharya2019shared}
Sagnik Bhattacharya, Amitalok~J Budkuley, and Sidharth Jaggi.
\newblock Shared randomness in arbitrarily varying channels.
\newblock In {\em 2019 IEEE International Symposium on Information Theory
  (ISIT)}, pages 627--631. IEEE, 2019.

\bibitem[BF63]{blachman-few-1963-multiple-packing}
NM~Blachman and L~Few.
\newblock Multiple packing of spherical caps.
\newblock {\em Mathematika}, 10(1):84--88, 1963.

\bibitem[BHL08]{benhaim-litsyn-2006-reliability-gaussian}
Yael Ben-Haim and Simon Litsyn.
\newblock Improved upper bounds on the reliability function of the gaussian
  channel.
\newblock {\em IEEE Transactions on Information Theory}, 54(1):5--12, 2008.

\bibitem[BL11]{blinovsky-litsyn-2011}
Vladimir Blinovsky and Simon Litsyn.
\newblock New asymptotic bounds on the size of multiple packings of the
  euclidean sphere.
\newblock {\em Discrete \& Computational Geometry}, 46(4):626--635, 2011.

\bibitem[Bla62]{blachman-1962}
N.~Blachman.
\newblock {On the capacity of bandlimited channel perturbed by statistically
  dependent interference}.
\newblock {\em IRE Transactions on Information Theory}, 8:48--55, 1962.

\bibitem[Bli86]{blinovsky-1986-ls-lb-binary}
Vladimir~M Blinovsky.
\newblock {Bounds for codes in the case of list decoding of finite volume}.
\newblock {\em {Problems of Information Transmission}}, 22:7--19, 1986.

\bibitem[Bli99]{blinovsky-1999-list-dec-real}
V~Blinovsky.
\newblock Multiple packing of the euclidean sphere.
\newblock {\em IEEE Transactions on Information Theory}, 45(4):1334--1337,
  1999.

\bibitem[Bli05a]{blinovsky-2005-ls-lb-qary}
Vladimir~M Blinovsky.
\newblock {Code bounds for multiple packings over a nonbinary finite alphabet}.
\newblock {\em {Problems of Information Transmission}}, 41:23--32, 2005.

\bibitem[Bli05b]{blinovsky-2005-random-packing}
Vladimir~M Blinovsky.
\newblock Random sphere packing.
\newblock {\em Problems of Information Transmission}, 41(4):319--330, 2005.

\bibitem[Bli08]{blinovsky-2008-ls-lb-qary-supplementary}
Vladimir~M Blinovsky.
\newblock {On the convexity of one coding-theory function}.
\newblock {\em {Problems of Information Transmission}}, 44:34--39, 2008.

\bibitem[Bli12]{blinovsky2012book}
Volodia Blinovsky.
\newblock {\em Asymptotic combinatorial coding theory}, volume 415.
\newblock Springer Science \& Business Media, 2012.

\bibitem[CKM{\etalchar{+}}17]{cohn-2017-24spherepacking}
Henry Cohn, Abhinav Kumar, Stephen~D Miller, Danylo Radchenko, and Maryna
  Viazovska.
\newblock The sphere packing problem in dimension 24.
\newblock {\em Annals of Mathematics}, pages 1017--1033, 2017.

\bibitem[CS13]{conway-sloane-book}
John~Horton Conway and Neil James~Alexander Sloane.
\newblock {\em {Sphere packings, lattices and groups}}, volume 290.
\newblock Springer Science \& Business Media, 2013.

\bibitem[Del73]{delsarte-1973}
Philippe Delsarte.
\newblock An algebraic approach to the association schemes of coding theory.
\newblock {\em Philips Res. Rep. Suppl.}, 10:vi+--97, 1973.

\bibitem[DG21]{d2021new}
Arkady~G D’yachkov and D~Yu Goshkoder.
\newblock New lower bounds on the fraction of correctable errors under list
  decoding in combinatorial binary communication channels.
\newblock {\em Problems of Information Transmission}, 57(4):301--320, 2021.

\bibitem[Eli57]{elias-1957-listdec}
Peter Elias.
\newblock {\em List decoding for noisy channels}.
\newblock Massachusetts Institute of Technology, Research Laboratory of
  Electronics, Cambridge, Mass., 1957.
\newblock Rep. No. 335.

\bibitem[Gal65]{gallager-1965-simple-deriv}
R~Gallager.
\newblock A simple derivation of the coding theorem and some applications.
\newblock {\em IEEE Transactions on Information Theory}, 11(1):3--18, 1965.

\bibitem[Gal68]{gallager}
Robert~G. Gallager.
\newblock {\em {Information Theory and Reliable Communication}}.
\newblock MIT Press, 1968.

\bibitem[GHS20]{guruswami-2020-listdec-insdel}
Venkatesan Guruswami, Bernhard Haeupler, and Amirbehshad Shahrasbi.
\newblock Optimally resilient codes for list-decoding from insertions and
  deletions.
\newblock In {\em Proceedings of the 52nd Annual ACM SIGACT Symposium on Theory
  of Computing}, pages 524--537, 2020.

\bibitem[Gil52]{gilbert1952}
Edgar~N Gilbert.
\newblock A comparison of signalling alphabets.
\newblock {\em The Bell system technical journal}, 31(3):504--522, 1952.

\bibitem[Gop77]{goppa1977}
Valerii~Denisovich Goppa.
\newblock Codes associated with divisors.
\newblock {\em Problemy Peredachi Informatsii}, 13(1):33--39, 1977.

\bibitem[GP12]{grigorescu-peikert-2012-list-dec-barnes-wall}
Elena Grigorescu and Chris Peikert.
\newblock List decoding barnes-wall lattices.
\newblock In {\em 2012 IEEE 27th Conference on Computational Complexity}, pages
  316--325. IEEE, 2012.

\bibitem[Gur06]{guruswami-it2003}
V~Guruswami.
\newblock List decoding from erasures: bounds and code constructions.
\newblock {\em IEEE Transactions on Information Theory}, 49(11):2826--2833,
  2006.

\bibitem[HAB{\etalchar{+}}17]{hales2017formal}
Thomas Hales, Mark Adams, Gertrud Bauer, Tat~Dat Dang, John Harrison, Hoang
  Le~Truong, Cezary Kaliszyk, Victor Magron, Sean McLaughlin, Tat~Thang Nguyen,
  et~al.
\newblock A formal proof of the kepler conjecture.
\newblock In {\em Forum of mathematics, Pi}, volume~5. Cambridge University
  Press, 2017.

\bibitem[Hae12]{haenggi2012stochastic}
Martin Haenggi.
\newblock {\em Stochastic geometry for wireless networks}.
\newblock Cambridge University Press, 2012.

\bibitem[HF11]{hales1998kepler}
Thomas Hales and Samuel Ferguson.
\newblock {\em The {K}epler conjecture}.
\newblock Springer, New York, 2011.
\newblock The Hales-Ferguson proof, Including papers reprinted from Discrete
  Comput. Geom. {{\bf{3}}6} (2006), no. 1, Edited by Jeffrey C. Lagarias.

\bibitem[HK19]{hosseinigoki2019list}
Fatemeh Hosseinigoki and Oliver Kosut.
\newblock List-decoding capacity of the gaussian arbitrarily-varying channel.
\newblock {\em Entropy}, 21(6):575, 2019.

\bibitem[Hug97]{hughes-1997-list-avc}
Brian~L. Hughes.
\newblock The smallest list for the arbitrarily varying channel.
\newblock {\em IEEE Transactions on Information Theory}, 43(3):803--815, 1997.

\bibitem[IZF12]{ingber-2012-finite-dim-infinite-const}
Amir Ingber, Ram Zamir, and Meir Feder.
\newblock Finite-dimensional infinite constellations.
\newblock {\em IEEE transactions on information theory}, 59(3):1630--1656,
  2012.

\bibitem[Jos58]{joshi1958singleton}
D.~D. Joshi.
\newblock A note on upper bounds for minimum distance codes.
\newblock {\em Information and Control}, 1:289--295, 1958.

\bibitem[Kep11]{kepler-1611}
Johannes Kepler.
\newblock Strena seu de nive sexangula (the six-cornered snowflake).
\newblock {\em Frankfurt: Gottfried. Tampach}, 1611.

\bibitem[KL78]{kabatiansky-1978}
Grigorii~Anatolevich Kabatiansky and Vladimir~Iosifovich Levenshtein.
\newblock {On bounds for packings on a sphere and in space}.
\newblock {\em Problemy Peredachi Informatsii}, 14(1):3--25, 1978.

\bibitem[KL95]{kalai-linial-1995-distance-distribution}
Gil Kalai and Nathan Linial.
\newblock On the distance distribution of codes.
\newblock {\em IEEE Transactions on Information Theory}, 41(5):1467--1472,
  1995.

\bibitem[Kom53]{komamiya1953singleton}
Y~Komamiya.
\newblock Application of logical mathematics to information theory.
\newblock {\em Proc. 3rd Japan. Nat. Cong. Appl. Math}, 437, 1953.

\bibitem[Lan04]{langberg-focs2004}
M.~Langberg.
\newblock Private codes or succinct random codes that are (almost) perfect.
\newblock In {\em 45th Annual IEEE Symposium on Foundations of Computer
  Science}, pages 325--334, 2004.

\bibitem[Lee82]{lee-1982-high-order-voronoi}
Der-Tsai Lee.
\newblock On k-nearest neighbor voronoi diagrams in the plane.
\newblock {\em IEEE transactions on computers}, 100(6):478--487, 1982.

\bibitem[Lit99]{litsyn-1999}
Simon Litsyn.
\newblock New upper bounds on error exponents.
\newblock {\em IEEE Transactions on Information Theory}, 45(2):385--398, 1999.

\bibitem[Mer14]{merhav2014list}
Neri Merhav.
\newblock List decoding—random coding exponents and expurgated exponents.
\newblock {\em IEEE Transactions on Information Theory}, 60(11):6749--6759,
  2014.

\bibitem[Min10]{minkowski-sphere-pack}
Hermann Minkowski.
\newblock {\em Geometrie der zahlen}.
\newblock BG Teubner, 1910.

\bibitem[MP22]{mook-peikert-2020-lattice}
Ethan Mook and Chris Peikert.
\newblock Lattice (list) decoding near {M}inkowski's inequality.
\newblock {\em IEEE Trans. Inform. Theory}, 68(2):863--870, 2022.

\bibitem[MRRW77]{mrrw2}
Robert McEliece, Eugene Rodemich, Howard Rumsey, and Lloyd Welch.
\newblock New upper bounds on the rate of a code via the delsarte-macwilliams
  inequalities.
\newblock {\em IEEE Transactions on Information Theory}, 23(2):157--166, 1977.

\bibitem[Plo60]{plotkin-1960}
Morris Plotkin.
\newblock Binary codes with specified minimum distance.
\newblock {\em IRE Transactions on Information Theory}, 6(4):445--450, 1960.

\bibitem[Pol94]{poltyrev1994coding}
Gregory Poltyrev.
\newblock On coding without restrictions for the awgn channel.
\newblock {\em IEEE Transactions on Information Theory}, 40(2):409--417, 1994.

\bibitem[Pol16]{polyanskiy-2016-list-dec}
Yury Polyanskiy.
\newblock Upper bound on list-decoding radius of binary codes.
\newblock {\em IEEE Transactions on Information Theory}, 62(3):1119--1128,
  2016.

\bibitem[PZ21]{polyanskii-zhang-2021-z}
Nikita Polyanskii and Yihan Zhang.
\newblock Codes for the z-channel.
\newblock {\em arXiv preprint arXiv:2105.01427}, 2021.

\bibitem[Ran55]{rankin-sphericalcap-1955}
R.A. Rankin.
\newblock {The closest packing of spherical caps in n dimensions}.
\newblock {\em Proc.\ Glasgow Math.\ Assoc.}, 2:139--144, 1955.

\bibitem[RS60]{reed-solomon}
Irving~S Reed and Gustave Solomon.
\newblock Polynomial codes over certain finite fields.
\newblock {\em Journal of the society for industrial and applied mathematics},
  8(2):300--304, 1960.

\bibitem[Sam13]{samorodnitsky-l1}
Alex Samorodnitsky.
\newblock A bound on l1 codes.
\newblock \url{https://www.cs.huji.ac.il/~salex/papers/L1_codes.pdf}, 2013.

\bibitem[Sar08]{sarwate-thesis}
Anand~D. Sarwate.
\newblock {\em Robust and adaptive communication under uncertain interference}.
\newblock PhD thesis, EECS Department, University of California, Berkeley, Jul
  2008.

\bibitem[SEW13]{swannack-erez-wornell-2013-geometric-relation}
Charles~H Swannack, Uri Erez, and Gregory~W Wornell.
\newblock Geometric relationships between gaussian and modulo-lattice error
  exponents.
\newblock {\em arXiv preprint arXiv:1308.1609}, 2013.

\bibitem[SG12]{sarwate-gastpar-2012-listdec}
Anand~D Sarwate and Michael Gastpar.
\newblock List-decoding for the arbitrarily varying channel under state
  constraints.
\newblock {\em IEEE transactions on information theory}, 58(3):1372--1384,
  2012.

\bibitem[Sha48]{shannon1948mathematical}
Claude~E Shannon.
\newblock A mathematical theory of communication.
\newblock {\em The Bell system technical journal}, 27(3):379--423, 1948.

\bibitem[Sha59]{shannon-1959-ee-gaussian}
C.~E. Shannon.
\newblock {Probability of error for optimal codes in a {G}aussian channel}.
\newblock {\em Bell Syst. Tech. J.}, 38:611--656, May 1959.

\bibitem[Sin64]{singleton1964}
Richard Singleton.
\newblock Maximum distance q-nary codes.
\newblock {\em IEEE Transactions on Information Theory}, 10(2):116--118, 1964.

\bibitem[Spe70]{spence-1970-lp-packing}
E~Spence.
\newblock Packing of spheres in lp.
\newblock {\em Glasgow Mathematical Journal}, 11(1):72--80, 1970.

\bibitem[SSSZ20]{sah-2020-lp-sphere-packing}
Ashwin Sah, Mehtaab Sawhney, David Stoner, and Yufei Zhao.
\newblock Exponential improvements for superball packing upper bounds.
\newblock {\em Advances in Mathematics}, 365:107056, 2020.

\bibitem[Thu11]{thue1911-2dspherepacking}
Axel Thue.
\newblock {\em {\ "U} about the densest compilation of congruent circles in a
  plane}.
\newblock Number~1. J. Dybwad, 1911.

\bibitem[T{\'o}t40]{toth-1940-2dspherepacking}
L~Fejes T{\'o}th.
\newblock Uber einen geometrischen satz.
\newblock {\em Math}, 2(46):79--83, 1940.

\bibitem[TVZ82]{tvz}
Michael~A Tsfasman, SG~Vl{\u{a}}dutx, and Th~Zink.
\newblock Modular curves, shimura curves, and goppa codes, better than
  varshamov-gilbert bound.
\newblock {\em Mathematische Nachrichten}, 109(1):21--28, 1982.

\bibitem[Var57]{varshamov1957}
Rom~Rubenovich Varshamov.
\newblock Estimate of the number of signals in error correcting codes.
\newblock {\em Docklady Akad. Nauk, SSSR}, 117:739--741, 1957.

\bibitem[VH14]{vanhandel-2014-high-dim-prob}
Ramon Van~Handel.
\newblock Probability in high dimension.
\newblock Technical report, PRINCETON UNIV NJ, 2014.

\bibitem[Via17]{viazovska-2017-8dspherepacking}
Maryna~S Viazovska.
\newblock The sphere packing problem in dimension 8.
\newblock {\em Annals of Mathematics}, pages 991--1015, 2017.

\bibitem[VO13]{viterbi2013principles}
Andrew~J Viterbi and Jim~K Omura.
\newblock {\em Principles of digital communication and coding}.
\newblock Courier Corporation, 2013.

\bibitem[Woz58]{wozencraft-1958-listdec}
John~M Wozencraft.
\newblock List decoding.
\newblock {\em Quarterly Progress Report}, 48:90--95, 1958.

\bibitem[ZBJ20]{zhang-2019-list-dec-general}
Yihan Zhang, Amitalok~J. Budkuley, and Sidharth Jaggi.
\newblock {Generalized List Decoding}.
\newblock In Thomas Vidick, editor, {\em 11th Innovations in Theoretical
  Computer Science Conference (ITCS 2020)}, volume 151 of {\em Leibniz
  International Proceedings in Informatics (LIPIcs)}, pages 51:1--51:83,
  Dagstuhl, Germany, 2020. Schloss Dagstuhl--Leibniz-Zentrum fuer Informatik.

\bibitem[ZJB20]{zhang-2020-obli-list-dec}
Yihan Zhang, Sidharth Jaggi, and Amitalok~J Budkuley.
\newblock {Tight List-Sizes for Oblivious AVCs under Constraints}.
\newblock {\em arXiv preprint arXiv:2009.03788}, 2020.

\bibitem[ZV22a]{zhang2022listtrans-it}
Yihan Zhang and Shashank Vatedka.
\newblock List decoding random euclidean codes and infinite constellations.
\newblock {\em IEEE Transactions on Information Theory}, 2022.

\bibitem[ZV22b]{zhang2022lowerisit}
Yihan Zhang and Shashank Vatedka.
\newblock Lower bounds on list decoding capacity using error exponents.
\newblock In {\em 2022 IEEE International Symposium on Information Theory
  (ISIT)}, pages 1324--1329. IEEE, 2022.

\bibitem[ZV22c]{zhang-split-misc}
Yihan Zhang and Shashank Vatedka.
\newblock Multiple packing: Lower and upper bounds.
\newblock {\em arXiv preprint arXiv:2211.04406}, 2022.

\bibitem[ZV22d]{zhang-split-ppp}
Yihan Zhang and Shashank Vatedka.
\newblock Multiple packing: Lower bounds via infinite constellations.
\newblock {\em arXiv preprint arXiv:2211.04407}, 2022.

\bibitem[ZVJ20]{zhang-2020-twoway}
Yihan Zhang, Shashank Vatedka, and Sidharth Jaggi.
\newblock Quadratically constrained two-way adversarial channels.
\newblock {\em arXiv preprint arXiv:2001.02575}, 2020.

\bibitem[ZVJS22]{zhang2022quadratically}
Yihan Zhang, Shashank Vatedka, Sidharth Jaggi, and Anand~D Sarwate.
\newblock Quadratically constrained myopic adversarial channels.
\newblock {\em IEEE Transactions on Information Theory}, 2022.

\end{thebibliography}


\appendices

\section{Collection of useful results}
\label{sec:prelim}

In this section, we collect some known results that are used in various proofs.

\begin{definition}[Gamma function]
	\label{def:gamma-fn}
	For any $ z\in\bC $ with $ \Re(z)>0 $, the \emph{Gamma function} $ \Gamma(z) $ is defined as 
	\begin{align}
	\Gamma(z) &\coloneqq \int_0^\infty v^{z-1}e^{-v}\diff v. \notag 
	\end{align}
\end{definition}


\begin{lemma}[Markov]
	\label{lem:markov}
	If $\bfx$ is a nonnegative random variable, then for any $a>0$, $ \prob{\bfx\ge a}\le \expt{\bfx}/a $. 
\end{lemma}



\begin{definition}[$Q$-function]
	\label{def:q-fn}
	The \emph{$ Q $-function} is defined as 
	\begin{align}
	Q(x) &\coloneqq \prob{\cN(0,1)>x} = \frac{1}{\sqrt{2\pi}} \int_x^\infty e^{-g^2/2} \diff g. \notag 
	\end{align}
\end{definition}

\begin{lemma}
	\label{lem:q-fn-bd}
	For any $ x>0 $, 
	\begin{align}
	Q(x) &= \frac{1}{12} e^{-x^2/2} (1+e^{-\Omega(x)}). \notag 
	\end{align}
	As a direct corollary, for any $ x>0 $, 
	\begin{align}
	\prob{\cN(0,\sigma^2)>x} &= Q(x/\sigma) = \frac{1}{12}e^{-\frac{x^2}{2\sigma^2}}(1+e^{-\Omega(x)}). \notag 
	\end{align}
\end{lemma}

\begin{lemma}[Integration in polar coordinates]
	\label{lem:int-polar}
	For any integrable function $ f\colon\bR^n\to\bR $, we have
	\begin{align}
	\int_{\bR^n}f(\vx)\diff\vx = \int_{\cS^{n-1}}\int_0^\infty f(r\vtheta)r^{n-1}\card{\cS^{n-1}} \diff r\diff \mu(\vtheta)
	\notag 
	\end{align}
	where $ \mu $ is the uniform probability measure on $ \cS^{n-1} $, i.e., for $ \cA\subset\cS^{n-1} $, $ \mu(\cA)\coloneqq\frac{|\cA|}{|\cS^{n-1}|} $. 
\end{lemma}


\begin{theorem}[Laplace's method]
	\label{thm:lap-1d}
	Let $ a<b\in\bR $ 
	and $ f,g\colon\bR\to\bR $. 
	\begin{enumerate}
		\item If $t^*\in(a,b)$ is the unique minimum point of $ f $ in $ [a,b] $ such that $ f'(t^*) = 0,f''(t^*)>0,g(t^*)\ne0 $, then 
		\begin{align}
		\int_a^bg(t)e^{-Mf(t)}\diff t
		&\stackrel{M\to\infty}{\asymp} e^{-Mf(t^*)}g(t^*)\sqrt{\frac{2\pi}{Mf''(t^*)}}. \notag 
		\end{align}
		
		\item If $a$ is the unique minimum point of $ f $ in $ [a,b] $ such that $ f'(a) = 0,f''(a)>0,g(a)\ne0 $, then 
		\begin{align}
		\int_a^bg(t)e^{-Mf(t)}\diff t
		&\stackrel{M\to\infty}{\asymp} e^{-Mf(a)}g(a)\sqrt{\frac{\pi}{2Mf''(a)}}. 
		\notag
		\end{align}
		
		\item If $ a $ is the unique minimum point of $f$ in $ [a,b] $ such that $ f'(a)>0,g(a)\ne0 $, then 
		\begin{align}
		\int_a^bg(t)e^{-Mf(t)}\diff t
		&\stackrel{M\to\infty}{\asymp} e^{-Mf(a)}\frac{g(a)}{Mf'(a)}. 
		\notag 
		\end{align}
	\end{enumerate}
\end{theorem}

\begin{theorem}[Laplace's method]
	\label{thm:lap-fine}
	Let $ a\in\bR $ and $ L\ge2 $ be an integer. 
	Suppose $ g\colon\bR\to\bR $ satisfies $ g(a) = g^{(1)}(a) = g^{(2)}(a) =\cdots= g^{(L-2)}(a) = 0 $ and $ g^{(L-1)}(a)\ne0 $ where $ g^{(i)} $ denotes the $i$-th derivative of $g$. 
	Suppose $ f\colon\bR\to\bR $ attains its unique minimum at $ a $ in the interval $ [a,\infty) $ and $ f^{(1)}(a)>0 $. 
	Then we have 
	\begin{align}
	\int_a^\infty g(t)e^{-Mf(t)}\diff t
	&\stackrel{M\to\infty}{\asymp} e^{-Mf(a)}\frac{g^{(L-1)}(a)}{(Mf^{(1)}(a))^L} . \notag 
	\end{align}
\end{theorem}

\begin{proof}
	The proof follows closely that of the standard Laplace's formula and we only present a sketch of the former.\footnote{In the following derivation, the approximate equalities $ \approx $ hide relative errors that we are not going to specify. } 
	The deviation is two-fold: $(i)$ the function $g$ is degenerate at a higher order; $(ii)$ the extreme point $a$ of $f$ is on the boundary of the integration domain and is \emph{not} a stationary point. 
	\begin{align}
	& \int_a^\infty g(t)e^{-Mf(t)}\diff t \notag \\
	&\approx e^{-Mf(a)} \int_a^{a+\eps} g(t) e^{-M(f(t) - f(a))} \diff t\notag \\
	&\approx e^{-Mf(a)} \int_a^{a+\eps} \sqrbrkt{g(a) + g^{(1)}(a)(t - a) + \frac{g^{(2)}(a)}{2}(t - a)^2 + \cdots + \frac{g^{(L-2)}(a)}{(L-2)!}(t - a)^{L-2} + \frac{g^{(L-1)}(a)}{(L-1)!}(t - a)^{L-1}} \notag \\
	& \quad\quad\quad\quad\quad\quad e^{-M\paren{[f(a) + f^{(1)}(a)(t - a)] - f(a)}} \diff t \label{eqn:taylor} \\
	&= e^{-Mf(a)}\int_a^{a+\eps} \frac{g^{(L-1)}(a)}{(L-1)!}(t-a)^{L-1}e^{-Mf^{(1)}(a)(t-a)} \diff t \label{eqn:taylor-vanish} \\
	&\approx e^{-Mf(a)}\frac{g^{(L-1)}(a)}{(L-1)!}\int_a^{\infty} (t-a)^{L-1}e^{-Mf^{(1)}(a)(t-a)} \diff t \notag \\
	&= e^{-Mf(a)}\frac{g^{(L-1)}(a)}{(L-1)!}\int_0^\infty \paren{\frac{u}{Mf^{(1)}(a)}}^{L-1}e^{-u}(Mf^{(1)}(a))^{-1}\diff u\label{eqn:change-var-lap} \\
	&= e^{-Mf(a)}\frac{g^{(L-1)}(a)}{(L-1)!}(Mf^{(1)}(a))^{-L} \int_0^\infty u^{L-1} e^{-u} \diff u \notag \\
	&= e^{-Mf(a)}\frac{g^{(L-1)}(a)}{(L-1)!}(Mf^{(1)}(a))^{-L}\Gamma(L) \label{eqn:by-def-gamma} \\
	&= e^{-Mf(a)}\frac{g^{(L-1)}(a)}{(Mf^{(1)}(a))^L}. \label{eqn:gamma-at-int}
	\end{align}
	In \Cref{eqn:taylor}, we take the $(L-1)$-st Taylor polynomial of $ g $ at $ a $ and the first Taylor polynomial of $f$ at $a$. 
	\Cref{eqn:taylor-vanish} follows since, by the assumption, the first $L-1$ terms of the Taylor polynomial of $g$ vanish at $a$. 
	In \Cref{eqn:change-var-lap}, we let $ u = Mf^{(1)}(a)(t - a) $. 
	\Cref{eqn:by-def-gamma} follows from the definition of Gamma function (\Cref{def:gamma-fn})
	and \Cref{eqn:gamma-at-int} is because the Gamma function coincides with the factorial function at positive integer points. 
\end{proof}


\begin{theorem}[\cramer]
	\label{thm:cramer-ldp}
	Let $ \curbrkt{\bfx_i}_{i = 1}^n $ be a sequence of i.i.d.\ real-valued random variables. 
	Let $ \bfs_n \coloneq\frac{1}{n}\sum_{i = 1}^n\bfx_i $. 
	Then for any closed $ \cF\subset\bR $,
	\begin{align}
	\limsup_{n\to\infty} \frac{1}{n}\ln\prob{\bfs_n\in \cF} &\le -\inf_{x\in \cF} \sup_{\lambda\in\bR}\curbrkt{ \lambda x - \ln\expt{e^{\lambda \bfx_1}} }; \notag 
	\end{align}
	and for any open $ \cG\subset\bR $, 
	\begin{align}
	\liminf_{n\to\infty} \frac{1}{n}\ln\prob{\bfs_n\in \cG} &\ge -\inf_{x\in \cG} \sup_{\lambda\in\bR}\curbrkt{\lambda x - \ln\expt{e^{\lambda \bfx_1}}}. \notag 
	\end{align}
	Furthermore, when $\cF$ or $\cG$ corresponds to the upper (resp.\ lower) tail of $ \bfs_n $, the maximizer $ \lambda\ge0 $ (resp.\ $ \lambda\le0 $).
\end{theorem}


\begin{lemma}[Gaussian integral]
	\label{lem:gauss-int-1d}
	Let $ a>0 $ and $ b,c\in\bR $. We have
	\begin{align}
	\int_\bR e^{-ax^2 + bx + c}\diff x &= \sqrt{\frac{\pi}{a}}\cdot e^{\frac{b^2}{4a} + c}. \notag 
	\end{align}
\end{lemma}

\begin{lemma}[Gaussian integral]
	\label{lem:gauss-int}
	Let $ A\in\bR^{n\times n} $ be a positive-definite matrix. Then 
	\begin{align}
	\int_{\bR^n} \exp\paren{-\vx^\top A\vx} \diff\vx &= \sqrt{\frac{\pi^n}{\det(A)}}. \notag 
	\end{align}
\end{lemma}

\begin{definition}
	\label{def:chi-sqaure}
	The \emph{chi-square distribution} $ \chi^2(k) $ with \emph{degree of freedom} $k$ is defined as the distribution of $ \sum_{i = 1}^k\bfg_i^2 $ where $ \bfg_i\iid\cN(0,1) $ for $ 1\le i\le k $. 
\end{definition}

\begin{fact}
	\label{fact:mgf-chi-sq}
	If $ \bfx\sim\chi^2(k) $, then for $ \lambda<1/2 $,
	\begin{align}
	\expt{e^{\lambda\bfx}} &= \sqrt{1-2\lambda}^{-k}. \notag 
	\end{align}
\end{fact}

Plugging the formula in \Cref{fact:mgf-chi-sq} into \cramer's theorem (\Cref{thm:cramer-ldp}), we get the first order asymptotics of the tail of a chi-square random variable. 
\begin{lemma}
	\label{fact:asymp-chi-sq}
	If $ \bfx\sim\chi^2(k) $, then 
	\begin{align}
	\lim_{k\to\infty}\frac{1}{k}\ln\prob{\bfx>(1+\delta)k} &= \frac{1}{2}(-\delta+\ln(1+\delta)), && \text{for $ \delta>0 $}; \notag \\
	\lim_{k\to\infty}\frac{1}{k}\ln\prob{\bfx<(1-\delta)k} &= \frac{1}{2}(\delta+\ln(1-\delta)), && \text{for $ \delta\in(0,1) $}. \notag 
	\end{align}
\end{lemma}

\subsection*{Poisson Point Processes}
We use the following standard results on Poisson Point Processes. See~\cite{haenggi2012stochastic} for a reference.

\begin{definition}[PPP]
	\label{def:ppp}
	A \emph{homogeneous Poisson Point Process (PPP) }$ \cC $ in $ \bR^n $ with intensity $ \lambda>0 $ is a point process satisfying the following two conditions. 
	\begin{enumerate}
		\item\label{itm:ppp-def-poisson} For any bounded Borel set $ \cB\subset\bR^n $, $ |\cC\cap\cB|\sim\pois(\lambda|\cB|) $, that is, 
		\begin{align}
		\prob{|\cC\cap\cB| = k} &= e^{-\lambda|\cB|} \frac{(\lambda|\cB|)^k}{k!} \notag
		\end{align}
		for any $ k\in\bZ_{\ge0} $. 
		
		\item\label{itm:ppp-def-poisson-joint} For any $ \ell\in\bZ_{\ge2} $ and any collection of $\ell$ disjoint bounded Borel sets $ \cB_1,\cdots,\cB_\ell\subset\bR^n $, the random variables $ |\cC\cap\cB_1|,\cdots,|\cC\cap\cB_\ell| $ are independent, that is,
		\begin{align}
		\prob{\forall i\in[\ell],\; |\cC\cap\cB_i| = k_i} &= \prod_{i = 1}^\ell e^{-\lambda|\cB_i|} \frac{(\lambda|\cB_i|)^{k_i}}{k_i!} \notag
		\end{align}
		for any $ k_1,\cdots,k_\ell\in\bZ_{\ge0} $. 
	\end{enumerate}
\end{definition}

\begin{remark}
	All PPPs in this paper will be homogeneous, that is, the intensity is a constant and does not depend on the location of a point. 
\end{remark}

\begin{definition}[Intensity and factorial moment measure]
	Let $ \cC $ be a point process in $ \bR^n $. 
	The \emph{intensity measure} $ \Lambda(\cdot) $ induced by $ \cC $ is defined as the measure on $ \bR^n $ satisfying
	$ \Lambda(\cB) = \expt{\card{\cC\cap\cB}} $ for any Borel set $ \cB\subset\bR^n $. 
	The \emph{intensity field} (a.k.a.\ \emph{intensity} for short) $ \lambda(\cdot) $ is the density of $ \Lambda(\cdot) $ (whenever exists), i.e., 
	\begin{align}
	\Lambda(\cB) &= \int_\cB \lambda(\vx) \diff \vx. \notag  
	\end{align} 
	More generally, for any $ L\ge1 $, the \emph{$L$-th factorial moment measure} $ \Lambda^{(L)}(\cdot) $ induced by $ \cC $ is defined as the measure on $ (\bR^n)^L $ satisfying
	\begin{align}
	\Lambda^{(L)}(\cB_1\times\cdots\times\cB_L) &= \expt{\sum_{\substack{(\vbfx_1,\cdots,\vbfx_L)\in\cC^L\\\mathrm{distinct}}}\prod_{i = 1}^L\one_{\cB_i}(\vbfx_i)}, \notag 
	\end{align}
	for any $L$-tuple of Borel sets $ \cB_1,\cdots,\cB_L $ in $ \bR^n $ (not necessarily disjoint). 
	The \emph{$L$-th factorial moment density} $\lambda^{(L)}(\cdot,\cdots,\cdot)$ is the density of $ \Lambda^{(L)}(\cdot) $ (whenever exists): 
	\begin{align}
	\Lambda^{(L)}(\cB_1\times\cdots\times\cB_L) &= \int_{\cB_1}\cdots\int_{\cB_L} \lambda^{(L)}(\vx_1,\cdots,\vx_L) \diff\vx_1\cdots\diff\vx_L. \notag 
	\end{align}
	Note that the first factorial moment measure/density coincides with the intensity measure/field.  
\end{definition}

\begin{fact}
	\label{fact:ppp-intensity}
	For a homogeneous PPP with intensity $\lambda$, the $L$-th factorial moment measure $ \Lambda^{(L)}(\cdot) $ is given by
	\begin{align}
	\Lambda^{(L)}(\cB_1\times\cdots\times\cB_L) &= \lambda^L\prod_{i = 1}^n\card{\cB_i} \notag 
	\end{align}
	and the $L$-th factorial moment density $ \lambda^{(L)}(\cdot,\cdots,\cdot) $ is given by $ \lambda^{(L)}(\vx_1,\cdots,\vx_L) = \lambda^L $. 
\end{fact}

\begin{fact}
	\label{fact:ppp-properties}
	A homogeneous PPP $ \cC $ satisfies the following properties. 
	\begin{enumerate}
		\item\label{itm:ppp-prop-stationary} A homogeneous PPP is stationary, i.e., invariant to translation.
		
		\item\label{itm:ppp-prop-isotropic} A homogeneous PPP is isotropic, i.e., invariant to rotation. 
		
		\item\label{itm:ppp-prop-uniform} For any box $ \cQ \coloneqq \prod_{i = 1}^n(a_i,b_i] $ where $ a_i\le b_i $ for all $ i\in[n] $, the points in $ \cC\cap\cQ $ are independent and uniformly distributed in $ \cQ $, that is, the $i$-th ($ i\in[n] $) coordinate of any vector in $ \cC\cap\cQ $ is uniformly distributed in $ (a_i,b_i] $ and is independent of any other coordinates (in or not in the same vector). 
	\end{enumerate}
\end{fact}

\begin{definition}[\matern process]
	A \emph{\matern process} $ \cC' $ in $ \bR^n $ with \emph{exclusion radius} $ r>0 $ can be obtained from a PPP $ \cC $ in $ \bR^n $ with intensity $ \lambda $ by removing all pairs of points in $ \cC $ with distance at most $ r $. 
	The intensity $ \lambda' $ of the resulting \matern process $ \cC' $ is given by 
	$ \lambda' = \lambda e^{-\lambda|\cB^n(r)|} $. 
\end{definition}

\begin{theorem}[Campbell]
	\label{thm:campbell}
	For any $ L\in\bZ_{\ge1} $, any point process $\cC$ on $ \bR^n $ with $L$-th factorial moment measure $ \Lambda^{(L)}(\cdot) $ and any measurable function $ f\colon(\bR^n)^L\to\bR $, the following equation holds
	\begin{align}
	\expt{\sum_{\substack{(\vbfx_1,\cdots,\vbfx_L)\in\cC^L\\\mathrm{distinct}}}f(\vbfx_1,\cdots,\vbfx_L)} &= \int_{(\bR^n)^L} f(\vx_1,\cdots,\vx_L) \Lambda^{(L)}(\diff \vx_1\times\cdots\times\diff\vx_L). \notag 
	\end{align}
	If $ \Lambda^{(L)}(\cdot) $ has a density $ \lambda^{(L)}(\cdot,\cdots,\cdot) $, then the equation becomes
	\begin{align}
	\expt{\sum_{\substack{(\vbfx_1,\cdots,\vbfx_L)\in\cC^L\\\mathrm{distinct}}}f(\vbfx_1,\cdots,\vbfx_L)} &= \int_{\bR^n}\cdots\int_{\bR^n} f(\vx_1,\cdots,\vx_L) \lambda^{(L)}(\vx_1,\cdots,\vx_L) \diff\vx_1\cdots\diff\vx_L . \notag 
	\end{align}
\end{theorem}

\begin{theorem}[Slivnyak]
	\label{thm:slivnyak}
	Conditioned on a point (WLOG the origin, by \Cref{itm:ppp-prop-stationary} of \Cref{fact:ppp-properties}) in a homogeneous PPP, the distribution of the rest of the PPP (which is called the \emph{Palm distribution}) is equal to that of the original PPP. 
\end{theorem}

\end{document}